\documentclass{m2an}

\usepackage[ruled]{algorithm2e}
\usepackage{graphicx}
\usepackage{mathrsfs}
\usepackage{bbm}
\usepackage{amsfonts}
\usepackage{amssymb}
\usepackage{amsmath,amsthm}
\usepackage{amsmath}
\numberwithin{equation}{section}
\numberwithin{figure}{section}

\usepackage{indentfirst}
\usepackage{color}
\usepackage{multirow}
\usepackage{enumerate}
\usepackage{epstopdf}
\usepackage{cite,stmaryrd,txfonts}
\usepackage{tikz}
\usepackage{slashbox}
\usepackage{makecell}
\usepackage{algorithmicx}
\usepackage{booktabs}%±í¸ñÏß¼Ó´Ö
\usepackage{subfigure}
\newtheorem{theorem}{Theorem}[section]

\newtheorem{lemma}{Lemma}[section]
\newtheorem{remark}{Remark}[section]
\newtheorem{example}{Example}[section]

\begin{document}
%%-----------------------------
%%      the top matter
%%-----------------------------
\title{Adaptive finite element approximation of sparse optimal control with integral fractional Laplacian}% At most 5 thanks

\author{Fangyuan Wang}\address{School of Mathematics and Statistics, Shandong Normal University, Jinan, 250014, China }
\author{Qiming Wang}\address{School of Mathematical Sciences, Beijing Normal University, Zhuhai, 519087, China}
\author{Zhaojie Zhou$^*$}\address{School of Mathematics and Statistics, Shandong Normal University, Jinan, 250014, China}
\date{\footnotetext{$^*$Corresponding author: zhouzhaojie@sdnu.edu.cn}}

\begin{abstract}
In this paper we present and analyze a weighted residual a posteriori error estimate for an optimal control problem. The problem involves a nondifferentiable cost functional, a state equation with an integral fractional Laplacian, and control constraints. We employ subdifferentiation in the context of nondifferentiable convex analysis to obtain first-order optimality conditions. Piecewise linear polynomials are utilized to approximate the solutions of the state and adjoint equations. The control variable is discretized using the variational discretization method. Upper and lower bounds for the a posteriori error estimate of the finite element approximation of the optimal control problem are derived.
In the region where $\frac{3}{2} < \alpha < 2$, the residuals do not satisfy the $L^2(\Omega)$ regularity. To address this issue, an additional weight is included in the weighted residual estimator, which is based on a power of the distance from the mesh skeleton.
Furthermore, we propose an h-adaptive algorithm driven by the posterior view error estimator, utilizing the D$\rm{\ddot{o}}$rfler labeling criterion. The convergence analysis results show that the approximation sequence generated by the adaptive algorithm converges at the optimal algebraic rate. Finally, numerical experiments are conducted to validate the theoretical results.
 \end{abstract}
%
%\begin{resume}
% \end{resume}
%%
\subjclass[]{49J20,49M25,65N12,65N30,65N50}
\keywords{adaptive finite element; optimal control; sparse control; fractional Laplacian; a posteriori error estimate}
\maketitle
%%-----------------------------
%%      your text
%%-----------------------------
\section{Introduction}

\ \ \ In this paper, we present and analyze a weighted residual a posteriori error estimate for an optimal control problem involving a nondifferentiable cost functional, a state equation with an integral fractional Laplacian, and control constraints. For a bounded Lipschitz domain $\Omega\subset R^d, \Omega^c :=R^d\backslash\overline{\Omega}$, we consider
    \begin{eqnarray}\label{object}
  \min\limits_{u\in U_{ad}} J(y,u):= \frac{1}{2} \|y-y_d\|^2_{L^2(\Omega)}
  +\frac{\gamma}{2} \|u\|^2_{L^2(\Omega)}+\beta \|u\|_{L^1(\Omega)}
\end{eqnarray}
subject to
\begin{eqnarray}\label{state}\left\{ \begin{aligned}
 (-\Delta)^{\frac{\alpha}{2}} y&=f+u,\ &\mbox{in}\ \Omega,\\
   y&=0, \ &\mbox{on}\ \Omega^c,
\end{aligned}\right.
   \end{eqnarray}
and the control constraints
\begin{eqnarray*}
U_{ad}=\Big\{v\in L^{2}(\Omega)| a\leq v\leq b,  \ a, b\in R\Big\}.
  \end{eqnarray*}
Here parameters $\gamma>0$ and $\beta>0$.  In the sequel, $y$
is state and $u$ is the control variable. The function $y_d\in L^2(\Omega)$ is referred to as desired state. To focus on the scenario of nondifferentiability, it is assumed that $a, b \in R$ satisfy the condition that $a < 0 < b$. We notice that the set $U_{ad}$, is a nonempty, bounded, closed and convex subset of $L^2(\Omega)$.

The introduction of nonsmooth regularization term $\beta \|u\|_{L^1(\Omega)}$  in the PDE-constrained optimization problems promotes sparsity in the solutions. This allows the control variable in the optimization process to tend towards zero in regions where it has negligible impact on the cost function, therefore minimizing the cost function. Sparse optimization is widely used in many practical applications, especially in the processing and analysis of high dimensional data, such as noise processing, machine learning, face recognition, etc.
  %Therefore, we consider a target functional that combines the $L^1(\Omega)$ norm and $L^2(\Omega)$ norm, as expressed in Equation (\ref{object}), which allows for specific control of certain physical quantities or locations using the $L^1$ norm while maintaining smoothness and continuity through the $L^2$ norm.

Previous research has addressed the analysis of optimal control problems with a cost term containing $L^1(\Omega)$ norm \cite{cla,ca2,ca3,ca4,ca,ot,sta,wach,lengchen}. For instance, in \cite{sta}, the authors investigated the $L^1(\Omega)$ control problem constrained by a linear elliptic PDE, where the objective functional incorporated a regularization technique based on the $L^2(\Omega)$ control cost term. The authors analyzed the optimality conditions and proposed a semismooth Newton method that achieves local convergence with superlinear speed. Building upon this work, \cite{wach} provided a priori and posteriori error estimates through finite element analysis. Furthermore, in \cite{ca}, the authors considered a semilinear elliptic PDE as the state equation and analyzed second-order optimality conditions. Additionally, the authors in \cite{ot} studied the sparse control problem with a fractional diffusion equation as the state equation. They analyzed a priori error estimate for the fully discrete case using finite element methods. More recently, Ot$\rm{\acute{a}}$rola et al. studied a sparse optimal control problem with a non-differentiable cost functional, where the state equations are Poisson's problem and fractional diffusion equation, respectively in \cite{otzheng,otfen}. In \cite{otzheng}, the authors studied three different strategies for approximating the control variable, they proposed and analyzed a reliable and efficient a posteriori error estimate, and designed an adaptive strategy to achieve optimal convergence rates. In \cite{otfen},  the authors investigated an adaptive finite element method for sparse optimal control of fractional diffusion, taking into account the spectral definition of the fractional Laplacian operator.

In comparison to the priori error analysis of finite element approximations for PDE-constrained optimization, the design and analysis of a posteriori error estimate is not much. The initial work on reliable a posteriori error estimation for optimal control problems was presented in \cite{liu20011}, followed by a series of related studies \cite{liu20012,liu2002,liu20031,liu20032,liu2008}. Residual-based a posteriori error estimates incorporating data oscillation were introduced in \cite{hin}. Later, a unified framework for the a posterior error analysis of linear quadratic optimal control problems with control constraints was established in \cite{koh}, and the pure convergence of an adaptive finite element method for optimal control problems with variable divergence control was proved, that is, convergence without convergence rate. In \cite{gongyan}, the authors rigorously proved convergence and quasi-optimality of AFEM for optimal control problem involving state and adjoint state variable.

However, to the best of our knowledge, no previous research has combined adaptive finite element methods (AFEMs) with integral fractional Laplacian sparse optimal control to address such problems. Therefore, in this paper, we focus on the adaptive finite element approximation for sparse optimal control with the integral fractional Laplacian. We outline and analyze the solution methodology for problem (\ref{object})-(\ref{state}) based on the following considerations:

$\bullet$ The optimal control problem involving the fractional Laplacian operator can effectively simulate groundwater pollution \cite{benson}, turbulent flow \cite{turb}, and chaotic dynamics \cite{chaos}. Unlike integer-order diffusion equations, the fractional Laplacian operator exhibits power-law decay, which accurately captures heavy-tailed power-law decay phenomena observed in these applications. Hence, studying fractional optimal control problem is essential.

$\bullet$ Objective functional by introducing the $L^1(\Omega)$ norm to control some specific physical quantities or locations, and the $L^2(\Omega)$ norm to maintain smoothness and continuity, can better solve the practical problems that need to control the optimal cost. The existence of $L^1(\Omega)$ term in the objective function
 requires us to derive the first-order condition using a subdifferential approach \cite{ ca,sta, wach}, which is different from distributed optimal control problems.

$\bullet$
Due to the non-locality, non-differentiability, and intrinsic constraints of the fractional Laplacian operator, by adopting adaptive strategies and a posteriori error estimate, we can identify singularities and refine the mesh accordingly, which can more effectively allocate computational resources and achieve higher accuracy with lower computational costs. One of the challenges in designing the a posterior error estimator is the nature of the residual, that is, it is not necessarily in $L^2(\Omega)$. We refer to \cite{Fau} and introduce the weighted residual estimator, where the weights are given by the power of the distance to the grid skeleton
 \begin{eqnarray*}
 E^2_y(y_{\mathcal{T}_h},K):=\|\widetilde{h}^{\frac{\alpha}{2}}_K(f+u_{\mathcal{T}_h}-(-\Delta )^sy_{\mathcal{T}_h})\|^2_{L^2(K)},\\
 E^2_p(p_{\mathcal{T}_h},K):=\|\widetilde{h}^{\frac{\alpha}{2}}_K(y_{\mathcal{T}_h}-y_d-(-\Delta )^sp_{\mathcal{T}_h})\|^2_{L^2(K)},
  \end{eqnarray*}
 (see section 4).

$\bullet$ The adaptive finite element method is widely used, but there are not many convergence analyses of the algorithm. The optimal control problem we studied is a coupled system with nonlinear properties, which leads to the lack of orthogonality presented in \cite{jm} and brings difficulties to our convergence analysis. In order to address this issue, we refer to reference \cite{Car} and prove its quasi-orthogonality.

Recently, the only work on a posteriori error analysis for sparse optimal control constrained by fractional order equations, as in (\ref{object})-(\ref{state}), is found in \cite{otfen}. Compared with \cite{otfen}, this paper studies the integral definition of the fractional Laplacian operator, which plays an important role in the modeling of complex non-local and nonlinear phenomena such as diffusion, heat transfer, resistance and elasticity. And the main difference is that the convergence of the adaptive algorithm is also analyzed in this paper. In this paper, we use piecewise linear polynomial dispersion for the state variable and variational discretization for the control variable. We design a posterior error estimator that requires only discretization of the state variable and adjoint variable. Notably, in the $\frac{3}{2} < \alpha < 2$, the residual does not satisfy the $L^2(\Omega)$-regularity. To address this issue, an additional weight based on the power of the distance from the mesh skeleton is included in the weighted residual estimator. An h-adaptive algorithm driven by the D$\rm{\ddot{o}}$rfler marking criterion based on the a posteriori error estimator is proposed and its convergence is proved.

The organization of the paper is as follows: In section 2, we introduce the symbols used and provide a brief overview of elements in convex analysis, along with the regularity of solutions to optimal control problem. In section 3, we analyze the first-order optimality conditions for the problem. In section 4, we introduce the finite element discretization of the optimal control problem (\ref{object})-(\ref{state}) and design a weighted residual estimator. The core of our work is presented in sections 4 and 5. For the discretization introduced  at the beginning of section 4, we first derive upper and lower bounds for the a posteriori error estimate of the finite element approximation for the optimal control problem. An h-adaptive algorithm driven by the posterior view error estimator based on D$\rm{\ddot{o}}$rfler labeling criterion is proposed. In section 5, we show that the sequence of approximations produced by the adaptive algorithm converges at the optimal algebraic rate. In section 6, a series of numerical examples are provided to demonstrate the effectiveness of our theoretical findings.
\section{Preliminaries}
\ \ \ In this section we introduce some preliminaries about fractional Sobolev spaces, subdifferential and fractional Laplacian.
For a bounded domain $\Lambda\subset R^d, L^2(\Lambda)$ denotes the Banach spaces of standard 2-th Lebesgue integrable functions on $\Lambda$. For $s\in(0,1), $ $H^s(\Lambda)$ denotes the fractional Sobolev space. $H_0^s(\Lambda)$ is the subspace of $H^s(\Lambda)$ consisting of functions whose trace is zero on $\partial\Lambda$. Let $(\cdot,\cdot)$ and $\|\cdot\|$ denote the inner product and norm in $L^2(\Lambda)$, respectively. The
seminorm $|\cdot|_{H^s(\Lambda)}$ and the full norm $\|\cdot\|_{H^s(\Lambda)}$ are denoted as follows
$$|y|^2_{H^s(\Lambda)}=\int\int_{\Lambda\times\Lambda}\frac{y(v)-y(w)}{|v-w|^{d+2s}}dvdw,$$
$$\|y\|^2_{H^s(\Lambda)}=\|y\|^2+|y|^2_{H^s(\Lambda)}.$$
Moreover, we introduce the following space, which will be used in the weak formulation of state equation
$$\widetilde{H}^s(\Omega)=\{v\in H^s(R^d): v=0 \ \ \textrm{in}\ \  \Omega^c \}.$$

Next, we will review some concepts with respect to subdifferentials from convex analysis that will be useful in our upcoming analysis. For details, please refer to \cite{w}.
Consider a real and normed vector space $G$. Suppose $\phi : G \rightarrow R \cup \{\infty\}$ be a convex and proper functional. Let $v \in G $ be such that $\phi(v)<\infty$. A subgradient of $G$ at $v$ is an element $v^*\in G^*$ that satisfies
\begin{eqnarray}
\langle v^*,w-v\rangle_{G^*,G}\leq \phi(w)-\phi(v),\ \forall w\in G.
\end{eqnarray}
Here, $\langle\cdot,\cdot\rangle_{G^*,G}$ represents the duality pairing between $G^*$ and $G$. The set of all subgradients of $\phi$ at $\bar{v}$, denoted by $\partial \phi(\bar{v})$, refers to the subdifferential of $\phi$ at $\bar{v}$.
$$\partial \phi(\bar{v})=\{v\in L^2(\Omega):\phi(w)-\phi(\bar{v})\geq (v,w-\bar{v}),\ w\in L^2(\Omega)\}.$$
As $\phi$ is a convex functional, the subdifferential at any point $v$  within the effective domain of $\phi$ is not empty. Additionally, it is important to note that the subdifferential is monotone, i.e.,
\begin{eqnarray}\label{subg}
\langle v^*-w^*,v-w\rangle_{G^*,G}\geq 0,\ \forall v^*\in \partial \phi(v),\ \forall w^*\in \partial \phi(w).
\end{eqnarray}

Finally, we introduce the definition of fractional Laplacian:
 \begin{eqnarray}\label{lapla}
  (-\Delta)^{\frac{\alpha}{2}}y(x):=C(d,\alpha) \ {\textrm{p.v.}}\int_{R^d} \frac{y(x)-y(w)}{| x-w|^{d+\alpha}}dw.
  \end{eqnarray}
Here   $0<\alpha<2$, and $$C(d,\alpha)=\frac{ 2^{\alpha}\Gamma(\frac{\alpha}{2}+\frac{d}{2})}{\pi^{d/2}\Gamma(-\frac{\alpha}{2})}$$
and ''p.v.'' denotes the principal value of the integral:
\begin{eqnarray}
{\textrm{p.v.}}\int_{R^d} \frac{y(x)-y(w)}{|x-w|^{d+\alpha}}dw=\lim\limits_{\epsilon\rightarrow  0}\int_{R^d\setminus B_{\epsilon}(v)} \frac{y(x)-y(w)}{|x-w|^{d+\alpha}}dw,
\end{eqnarray}
where $B_{\epsilon}(v)$ is a ball of radius $\epsilon$ centered at $x$. The difference $y(x)-y(w)$ in the numerator of (\ref{lapla}), which vanishes at the singularity, provides a regularization, which together with averaging of positive and negative parts allows the principal value to exist.
A consequence of this definition is the mapping property (see \cite{bon}).
  \begin{eqnarray*}
  (-\Delta)^{\frac{\alpha}{2}}: H^s(R^d)\rightarrow H^{s-\alpha}(R^d),\ s\geq \frac{\alpha}{2}.
  \end{eqnarray*}
% We consider the following  problem
%\begin{eqnarray}\label{state_eq1}\left\{ \begin{aligned}
%(-\Delta)^{\frac{\alpha}{2}} y&=f,\ \mbox{in}&\ \Omega, \\
%y&=0,\ \mbox{in}&\ \Omega^c .
%\end{aligned}\right.
%   \end{eqnarray}

\section{Optimal control problem}

 \ \  \ The weak formulation of state equation (\ref{state}) reads: Find $y\in \widetilde{H}^{\frac{\alpha}{2}}(\Omega)$ such that
\begin{eqnarray}\label{weak_state_eq1}
a(y,v)=(f+u,v), \ \ \forall v\in \widetilde{H}^{\frac{\alpha}{2}}(\Omega).
   \end{eqnarray}
Here \begin{eqnarray*}
a(y,v)=\frac{C(d,\alpha) }{2}\int\int_{R^d\times R^d} \frac{(y(x)-y(w))(v(x)-v(w))}{|x-w|^{d+\alpha}}dxdw.
\end{eqnarray*}
We define
$$ \|y\|_{\widetilde{H}^{\frac{\alpha}{2}}(\Omega)}^2:=a(y,y)=\frac{C(d,\alpha)}{2}|y|^2_{H^{\frac{\alpha}{2}}(R^d)}.$$
 As the $H^{\frac{\alpha}{2}}(R^d)$ seminorm is equivalent to the $H^{\frac{\alpha}{2}}(R^d)$ norm on $ \widetilde{H}^{\frac{\alpha}{2}}(\Omega)$ (see \cite{Acosta2}), by Lax-Milgram theorem, the solution $y\in \widetilde{H}^{\frac{\alpha}{2}}(\Omega)$ exists and is unique.

 For the state equation with the right hand $f$ we can define a linear and bounded solution operator $\mathcal{S}: L^2(\Omega)\longrightarrow\widetilde{H}^{\frac{\alpha}{2}}(\Omega)$ such that $y = \mathcal{S}f$.
Moreover, the following regularity result holds for the state equation.
\begin{lemma}(\cite{bor})\label{state_regularity2}
For $f(x)+u(x)\in L^2(\Omega)$, there exists a solution $y\in  {H}^{{\frac{\alpha}{2}}+\sigma-\epsilon}(\Omega)$ satisfies
\begin{eqnarray*}
|y|_{{H}^{{\frac{\alpha}{2}}+\sigma-\epsilon}(\Omega)}\leq \frac{C(\Omega,d,\alpha)}{\epsilon^{\tau}}\|f+u\|_{L^2(\Omega)}, \forall 0<\epsilon <{\frac{\alpha}{2}}.
\end{eqnarray*}
 Here $\sigma=\min\{\frac{\alpha}{2},\frac{1}{2}\}$, $\tau=\frac{1}{2}$ for $1<\alpha<2$ and $\tau=\frac{1}{2}+\zeta$ for $0<\alpha\leq1$ as well as a constant $\zeta$ depending on $\Omega$ and $d$.
 %For $\alpha=1,$ we have
% \begin{eqnarray*}
%|y|_{{H}^{1-\epsilon}(\Omega)}\leq \frac{C(\Omega,d)}{\epsilon}\|f\|_{L^{\infty}(\Omega)}, \forall \epsilon >0.
%\end{eqnarray*}
\end{lemma}

The weak formulation of the optimal control problem (\ref{object})-(\ref{state}) reads:
 \begin{eqnarray}\label{weak_object}
  \min\limits_{y\in \widetilde{H}^{\frac{\alpha}{2}}(\Omega),\ u\in  U_{ad}} J(y,u)
\end{eqnarray}
subject to
\begin{eqnarray}\label{weak_state}
a(y,v)=(f+u,v), \ \ \forall v\in \widetilde{H}^{\frac{\alpha}{2}}(\Omega).
   \end{eqnarray}
Since the $J$ is strictly convex and weakly lower semicontinuous, this problem admits a unique optimal solution $(y,u)\in\widetilde{H}^{\frac{\alpha}{2}}(\Omega)\times L^2(\Omega)$.

In order to obtain optimality conditions for (\ref{weak_object})-(\ref{weak_state}), we introduce the following adjoint state $p$ as follows:
\begin{eqnarray}\label{weak_adjoint}
a(w,p)=(y-y_d,w), \ \ \forall w\in \widetilde{H}^{\frac{\alpha}{2}}(\Omega).
   \end{eqnarray}
Set
$$
j_1(u)= \frac{1}{2} \|\mathcal{S}(u+f)-y_d\|^2_{L^2(\Omega)}
  +\frac{\gamma}{2} \|u\|^2_{L^2(\Omega)}
$$
and  $j_2(u)=\|u\|_{L^1(\Omega)}$. Then we obtain the reduced  problem of (\ref{object}):
 \begin{eqnarray}\label{object1}
  \min\limits_{ u\in  U_{ad}}\hat{J}(u)= \min\limits_{ u\in  U_{ad}} j_1(u)+\beta j_2(u).
\end{eqnarray}
Although the reduced cost functional (\ref{object1}) is nonsmooth, it consists in the sum of a regular part and a convex nondifferentiable term. Thanks to the structure, optimality conditions can still be established according to the following result.
\begin{lemma}(\cite{loffe})\label{lamdastar}
Let $\hat{J}(u)$ be defined as in (\ref{object1}). The element $u\in U_{ad}$ is a minimizer of $\hat{J}(u)$ over $U_{ad}$ if and only if there exists a subgradient $\lambda^*\in \partial \hat{J}(u)$ such that
 \begin{eqnarray}
 (\lambda^*,v-u)\geq 0,\ \forall v\in U_{ad}.
\end{eqnarray}
\end{lemma}
\begin{theorem}(Optimality conditions)
 If $(y,u)$ is an optimal solution to (\ref{object1}), then it satisfies the following variational inequality
 \begin{eqnarray}\label{variational inequality}
 (p+\gamma u+\beta \lambda, v-u)\geq0,\ \forall v\in  U_{ad},
\end{eqnarray}
where $p$ denotes the solution to (\ref{weak_adjoint}) and $\lambda\in \partial j_2(u)$.
\end{theorem}
\begin{proof}
 Since the convex functional $j_1(u)$ is Fr$\rm{\acute{e}}$chet differentiable we immediately have that $\partial j_1(u)=j'_1(u).$
%From the optimality of $u$ we obtain that
%$$ j_1(u)+\beta j_2(u)\leq j_1(w)+\beta j_2(w),\ \forall w\in B_\epsilon(u).$$
%Let $w=u+t(v-u),\ v\in U_{ad},\ t>0$ sufficiently small, it follows from the definition of convex functional
% \begin{align*}
%0&\leq j_1(u+t(v-u))-j_1(u)+\beta(j_2(u+t(v-u))-j_2(u))\\
%&\leq j_1(u+t(v-u))-j_1(u)+\beta(tj_2(v)+(1-t)j_2(u)-j_2(u)).
%  \end{align*}
%  Further, dividing by $t$ and taking the limit on both sides we obtain
%   \begin{align}
%\lim \limits_{t \to 0} j'_1(u)(v-u)+\beta(j_2(v)-j_2(u))\geq0.
%  \end{align}
In view of the fact $j_2(u)$ is convex, that
$$\partial \hat{J}(u)=j_1'(u)+\beta\partial j_2(u).$$
  By simple calculations, we have that
\begin{align*}
 j_1'(u)(v-u)&=\lim\limits_{t \to 0}\frac{1}{2t}\int_{\Omega}\Big((y(u+t(v-u))-y_d)^2-(y(u)-y_d)^2\Big)dx  +\lim \limits_{t \to 0}\frac{\gamma}{2t}\int_{\Omega}\left((u+t(v-u))^{2}-u^{2}\right)dx\nonumber\\
 &=\int_{\Omega}(y(u)-y_{d})y'(u)(v-u)dx+\gamma\int_{\Omega}u(v-u)\\
 &=(p+\gamma u,v-u).
 \end{align*}
 According to Lemma \ref{lamdastar}, there exists a multiplier $\lambda\in \partial j_2(u),$ such that
$$
(p+\gamma u+\beta \lambda, v-u)\geq0.
$$
\end{proof}
For $a,b\in R$,  we introduce a projection operator $\Pi_{[a, b]}: L^2(\Omega)\rightarrow U_{ad}$ defined by
 \begin{eqnarray}\label{projection operator}
 \Pi_{[a,b]}(v)=\min\{b, \max\{a,v\}\}.
  \end{eqnarray}
  Then we have the following projection formulas.
  \begin{theorem}(Projection formulas)   Suppose $(y,p,u,\lambda)$ are the optimal variables associated to (\ref{variational inequality}), then we obtain
 \begin{align}
 &u=\Pi_{[a,b]}\left(-\frac{1}{\gamma}(p+\beta\lambda)\right),\label{u}\\
 &|p|\leq \beta,\ \mbox{in}\ \{x\in\Omega,\ u=0 \},\label{p}\\
 &\lambda=\Pi_{[-1,1]}\left(-\frac{1}{\beta}p\right)\label{lambda}.
  \end{align}
 It guarantees the uniqueness of the subgradient $\lambda.$
\end{theorem}
\begin{proof}
The derivation of the formula (\ref{u}) is standard in control theory. According to \cite{loffe}, we know that $\lambda\in \partial j_2(u)$ if and only if
     \begin{eqnarray}\label{lambda3}\left\{ \begin{aligned}
&\lambda(x)=1,  &u(x)>0,\\
&\lambda(x)=-1, &u(x)<0,\\
& \lambda(x)\in[-1,1], &u(x)=0.
   \end{aligned}\right.
  \end{eqnarray}
  By (\ref{u}), (\ref{lambda3}) and $a<0<b$, we arrive at
      \begin{eqnarray*}\left\{ \begin{aligned}
 &u(x)=0 \xrightarrow{(\ref{lambda3})}\lambda(x)\in[-1,1] \xrightarrow{(\ref{u})}\mid p\mid\leq \beta,\\
 &u(x)<0 \xrightarrow{(\ref{lambda3})}\lambda=-1\xrightarrow{(\ref{u})}p+\beta\lambda>0 \Rightarrow p>\beta,\\
 &u(x)>0\xrightarrow{(\ref{lambda3})}\lambda=1\xrightarrow{(\ref{u})}p+\beta\lambda <0 \Rightarrow p< -\beta.
   \end{aligned}\right.
  \end{eqnarray*}
These three properties are equivalent to (\ref{p}). Therefore, (\ref{u}), (\ref{p}), (\ref{lambda3}) the previous estimate allow us to deduce (\ref{lambda})
     \begin{eqnarray*}\left\{ \begin{aligned}
 &\mid p\mid\leq \beta\xrightarrow{(\ref{p})}u(x)=0 \xrightarrow{(\ref{lambda3})} \lambda(x)\in[-1,1] \xrightarrow{(\ref{u})}p+\beta\lambda=0 \Rightarrow\lambda=\Pi_{[-1,1]}\left(-\frac{1}{\beta}p\right),\\
 & p>\beta\Rightarrow u(x)<0 \xrightarrow{(\ref{lambda3})}\lambda=-1\Rightarrow\lambda=\Pi_{[-1,1]}\left(-\frac{1}{\beta}p\right),\\
 &p< -\beta \Rightarrow u(x)>0\xrightarrow{(\ref{lambda3})}\lambda=1\Rightarrow\lambda=\Pi_{[-1,1]}\left(-\frac{1}{\beta}p\right) ,
   \end{aligned}\right.
  \end{eqnarray*}
  which completes the proof.
\end{proof}

At end  we present the following first order optimality conditions for above optimal control problems.
\begin{theorem} Let $(y,u)$ be the solution of the optimal control problem  (\ref{weak_object})-(\ref{weak_state}). Then there exists  an adjoint state $p$, $\lambda\in \partial j_2(u)$ such that
   \begin{eqnarray}\label{zuiyou}\left\{ \begin{aligned}
&a(y,v)= (f+u,v)\ &\forall v \in \widetilde{H}^{\frac{\alpha}{2}}(\Omega),\\
&a (w,p)=(y-y_d,w)\ &\forall w \in \widetilde{H}^{\frac{\alpha}{2}}(\Omega),\\
& (p+\gamma u+\beta \lambda, v-u)\geq0,\ &\forall v\in  U_{ad},
   \end{aligned}\right.
  \end{eqnarray}
 \end{theorem}

\section{Finite element approximation method and a posteriori error estimate}
\ \ \ We begin by partitioning the domain $\Omega$ into simplices $K$ with size $h_K:=| K | ^{\frac{1}{d}}$, forming a conforming partition $ \mathcal{T}_h= \{K\}$. We then define $h_{\mathcal{T}_h}=\max\limits_{K\in \mathcal{T}_h}h_K $ and denote by $\mathbb{T}$ the collection of conforming and shape regular meshes that are refinements of an initial mesh $\mathcal{T}_{h_{0}}$.
For $\mathcal{T}_h\in\mathbb{T},$  let  $\mathbb{V}_{\mathcal{T}_h}$ be the finite element space consisting of continuous piecewise linear functions over the triangulation $\mathcal{T}_h$
  \begin{eqnarray*}
  \mathbb{V}_{\mathcal{T}_h}=\{v_{\mathcal{T}_h}\in C(\bar{\Omega})\cap H_0^1(\Omega);\  v_{\mathcal{T}_h}| _K\in \mathbb{P}_1(K), \forall K\in \mathcal{T}_h\}.
    \end{eqnarray*}
%Let $\mathbb{V}_{\mathcal{T}_h} =\mathbb{V}_{\mathcal{T}_h}\cap H_0^1(\Omega).$
  For all elements $K\in \mathcal{T}_h$ and $k\in \mathbb{N}_0$, we introduce the $k$-th order element patch inductively by
 \begin{align*}
  & \Omega_h^0(K):=K, \mathcal{T}_h^0(K):=\{K\}\\
 &\Omega_h^k(K):=\mathrm{interior}(\bigcup\limits_{{K'}\in  \mathcal{T}_h^k(K)}\overline{K'}),\ \mbox{where}\ \mathcal{T}_h^k(K):=\{K'\in \mathcal{T}_h: \overline{K'}\cap \overline{\Omega_h^{k-1}(K)}\neq\emptyset \}.
 \end{align*}

 The finite element approximation of the optimal control problem  (\ref{weak_object})-(\ref{weak_state}) can be characterized as
$$
  \min\limits_{(y_{\mathcal{T}_h},u_{\mathcal{T}_h})\in \mathbb{V}_{\mathcal{T}_h}\times U_{ad}} J(y_{\mathcal{T}_h},u_{\mathcal{T}_h})
$$
subject to
\begin{eqnarray}\label{lisanstate}
a(y_{\mathcal{T}_h},v_{\mathcal{T}_h})=(f+u_{\mathcal{T}_h},v_{\mathcal{T}_h}), \ \ \forall v_{\mathcal{T}_h}\in \mathbb{V}_{\mathcal{T}_h}.
   \end{eqnarray}
 Here the admissible set of control $U_{ad}$ is not discretized, i.e.,  the so-called
variational discretization approach.
Similar to the continuous case we have the discrete first order optimality condition
  \begin{eqnarray}\label{lisan}\left\{ \begin{aligned}
&a(y_{\mathcal{T}_h},v_{\mathcal{T}_h})=(f+u_{\mathcal{T}_h},v_{\mathcal{T}_h}), &\forall v_{\mathcal{T}_h}\in \mathbb{V}_{\mathcal{T}_h},\\
&a(w_{\mathcal{T}_h},p_{\mathcal{T}_h})=(y_{\mathcal{T}_h}-y_d,w_{\mathcal{T}_h}), &\forall w_{\mathcal{T}_h}\in \mathbb{V}_{\mathcal{T}_h},\\
& (p_{\mathcal{T}_h}+\gamma u_{\mathcal{T}_h}+\beta \lambda_{\mathcal{T}_h}, v_{\mathcal{T}_h}-u_{\mathcal{T}_h})\geq0,\ &\forall v_{\mathcal{T}_h}\in  U_{ad},
   \end{aligned}\right.
  \end{eqnarray}
  where $\lambda_{\mathcal{T}_h} \in \partial j_2(u_{\mathcal{T}_h}).$
Next, we give the following  discrete projection formula.
  \begin{lemma}
   Suppose $(y_{\mathcal{T}_h},p_{\mathcal{T}_h},u_{\mathcal{T}_h},\lambda_{\mathcal{T}_h})$ are the optimal variables associated to (\ref{lisan}), then we obtain
  \begin{align}
 &u_{\mathcal{T}_h}=\Pi_{[a,b]}\left(-\frac{1}{\gamma}(p_{\mathcal{T}_h}+\beta\lambda_{\mathcal{T}_h})\right),\\
 &|p_{\mathcal{T}_h}|\leq \beta,\ \mbox{in}\ \{x\in\Omega,\ u_{\mathcal{T}_h}=0 \},\\
 &\lambda_{\mathcal{T}_h}=\Pi_{[-1,1]}\left(-\frac{1}{\beta}p_{\mathcal{T}_h}\right)\label{lamdath}.
  \end{align}
 \end{lemma}

   Similar to the continuous case, we can define a discrete control-to-state mapping $\mathcal{S}_{h}: L^2(\Omega)\longrightarrow  \mathbb{V}_{\mathcal{T}_h}$.
Set $ \Theta(h)=\sup\limits_{f\in L^2(\Omega), \|f\|=1}\inf\limits_{\chi_{\mathcal{T}_h}\in \mathbb{V}_{\mathcal{T}_h}}\|\mathcal{S}f-\chi_{\mathcal{T}_h}\|_{\widetilde{H}^{\frac{\alpha}{2}}(\Omega)}
.$  Then we have $\Theta(h)\ll 1 $ for $h\in (0,h_0)$ with $h_0\ll 1.$

   \begin{lemma}\label{etah}
Assume  that $\mathcal{S}f \in\widetilde{H}^{\frac{\alpha}{2}}(\Omega)$ and $\mathcal{S}_hf\in \mathbb{V}_{\mathcal{T}_h} $ are the solutions of the continuous and discretised state equation with right hand term $f\in L^2(\Omega)$. Then the following error estimates hold
 \begin{eqnarray*}
 \|\mathcal{S}f-\mathcal{S}_{h}f\|_{\widetilde{H}^{\frac{\alpha}{2}}(\Omega)}&\leq C\Theta(h)\|f\|_{L^{2}(\Omega)}
\end{eqnarray*}
and
\begin{eqnarray*}
\|\mathcal{S}f-\mathcal{S}_{h}f\|&\leq C\Theta(h)\|\mathcal{S}f-\mathcal{S}_{h}f\|_{\widetilde{H}^{\frac{\alpha}{2}}(\Omega)}.
\end{eqnarray*}
 \end{lemma}
\begin{proof}
We invoke the Galerkin orthogonality to arrive at
\begin{align*}
\|\mathcal{S}f-\mathcal{S}_{h}f\|_{\widetilde{H}^{\frac{\alpha}{2}}(\Omega)}^2&\leq a(\mathcal{S}f-\mathcal{S}_{h}f,\mathcal{S}f-\mathcal{S}_{h}f)=a(\mathcal{S}f-\mathcal{S}_{h}f,\mathcal{S}f-\chi_{\mathcal{T}_h}+\chi_{\mathcal{T}_h}-\mathcal{S}_{h}f)\\
&\leq C\| \mathcal{S}f-\mathcal{S}_{h}f\|_{\widetilde{H}^{\frac{\alpha}{2}}(\Omega)}\| \mathcal{S}f-\chi_{\mathcal{T}_h}\|_{\widetilde{H}^{\frac{\alpha}{2}}(\Omega)}.
\end{align*}
Thus we have
 \begin{align*}
 \|\mathcal{S}f-\mathcal{S}_{h}f\|_{\widetilde{H}^{\frac{\alpha}{2}}(\Omega)}&\leq C\Theta(h)\|f\|_{L^{2}(\Omega)}.
\end{align*}
 Next, let $m=\mathcal{S}g$ be the solution of the following problem with $g(x)\in L^2(\Omega)$
  \begin{eqnarray*}
 a(w,m)=(g,w),\  w\in \widetilde{H}^{\frac{\alpha}{2}}(\Omega).
\end{eqnarray*}
Then we have
%\begin{eqnarray*}\cases{
% (-\Delta)^{\frac{\alpha}{2}} m(x)=g(x),& $x \in \Omega,$\\
%   m(x)=0, & $x\in \Omega^c.$
%}
%   \end{eqnarray*}
$$
\| \mathcal{S}g-\mathcal{S}_{h}g\|_{\widetilde{H}^{\frac{\alpha}{2}}(\Omega)}\leq C\Theta(h)\|g\|_{L^{2}(\Omega)}.
$$
Setting $w=\mathcal{S}f-\mathcal{S}_{h}f$, to prove the second estimate, we invoke the Galerkin orthogonality and the previous estimate to arrive at
\begin{align*}
(\mathcal{S}f-\mathcal{S}_{\mathcal{T}_h}f,g)
&=a(\mathcal{S}g,\mathcal{S}f-\mathcal{S}_{\mathcal{T}_h}f)=a(\mathcal{S}g-\mathcal{S}_{\mathcal{T}_h}g+\mathcal{S}_{\mathcal{T}_h}g,\mathcal{S}f-\mathcal{S}_{\mathcal{T}_h}f)\\
&=a(\mathcal{S}g-\mathcal{S}_{\mathcal{T}_h}g,\mathcal{S}f-\mathcal{S}_{\mathcal{T}_h}f)\\
&\leq C\| \mathcal{S}g-\mathcal{S}_{\mathcal{T}_h}g\|_{\widetilde{H}^{\frac{\alpha}{2}}(\Omega)} \| \mathcal{S}f-\mathcal{S}_{\mathcal{T}_h}f\|_{\widetilde{H}^{\frac{\alpha}{2}}(\Omega)}\\
&\leq C \Theta(h)\|g\|_{L^{2}(\Omega)}\| \mathcal{S}f-\mathcal{S}_{\mathcal{T}_h}f\|_{\widetilde{H}^{\frac{\alpha}{2}}(\Omega)}.
\end{align*}
Consequently,
$$
\|\mathcal{S}f-\mathcal{S}_{h}f\|\leq C\Theta(h)\|\mathcal{S}f-\mathcal{S}_{h}f\|_{\widetilde{H}^{\frac{\alpha}{2}}(\Omega)}.
$$
\end{proof}

To derive  a posteriori error analysis we  need to introduce the following auxiliary problems
   \begin{eqnarray}\left\{ \begin{aligned}
a(\tilde{y},v)&=(f+u_{\mathcal{T}_h},v), \ \ \ \ \ \ \ \forall v\in \widetilde{H}^{\frac{\alpha}{2}}(\Omega),\label{yhu}\\
a(w,\tilde{p})&=(y_{\mathcal{T}_h}-y_{d},w), \ \ \ \ \ \ \ \forall w\in \widetilde{H}^{\frac{\alpha}{2}}(\Omega)\label{phy}.
\end{aligned}\right.
 \end{eqnarray}
Note that the residuals  do not satisfy the $L^2(\Omega)$-regularity for $\frac{3}{2} < \alpha < 2$. To address this issue, we require a weight function to measure the distance from the mesh skeleton. For a mesh $\mathcal{T}_h$, we introduce the weight function defined in \cite{Fau}
$$\omega_{\mathcal{T}_h}(x):=\inf\limits_{K\in\mathcal{T}_h}\inf\limits_{y\in\partial K}|x-y|.$$
We further define the corresponding weighted residual errors as follows
 \begin{eqnarray}
 E^2_y(y_{\mathcal{T}_h},K):=\|\widetilde{h}^{\frac{\alpha}{2}}_K(f+u_{\mathcal{T}_h}-(-\Delta )^{\frac{\alpha}{2}}y_{\mathcal{T}_h})\|^2_{L^2(K)},\label{Eydefin}\\
 E^2_p(p_{\mathcal{T}_h},K):=\|\widetilde{h}^{\frac{\alpha}{2}}_K(y_{\mathcal{T}_h}-y_d-(-\Delta )^{\frac{\alpha}{2}}p_{\mathcal{T}_h})\|^2_{L^2(K)},\label{Epdefin}
  \end{eqnarray}
where
   \begin{eqnarray*}\label{AD}\widetilde{h}^{\frac{\alpha}{2}}_K=\left\{\begin{aligned}
& h^{\frac{\alpha}{2}}_K,\ \ \ \ \ \ \ \ \ \ \ \ &\alpha\in(0,1],\\
&h^{{\frac{\alpha}{2}}-\sigma}_K\omega_{\mathcal{T}_h}^{\sigma},\ &\alpha\in(1,2),\ \sigma:=\frac{\alpha}{2}-{\frac{1}{2}}.
\end{aligned}\right.\end{eqnarray*}
Then on a subset $\omega\subset \Omega$, we define the error estimators of the state and adjoint state  by
 \begin{eqnarray*}
 E^2_y(y_{\mathcal{T}_h},\omega):=\sum\limits_{K\in \mathcal{T}_h, K\subset \omega}E^2_y(y_{\mathcal{T}_h},K),\ \ \ \ E^2_p({\mathcal{T}_h},\omega):=\sum\limits_{K\in \mathcal{T}_h, K\subset \omega}E^2_p(p_{\mathcal{T}_h},K).
  \end{eqnarray*}
Thus, $E_y(y_{\mathcal{T}_h},\mathcal{T}_h)$ and $E_p(p_{\mathcal{T}_h},\mathcal{T}_h)$ constitute the error estimators for the state equation
and the adjoint state equation on $\Omega$ with respect to $\mathcal{T}_h$ as follows
 \begin{eqnarray*}
 E^2_y(y_{\mathcal{T}_h},\mathcal{T}_h):=\sum\limits_{K\in \mathcal{T}_h}E^2_y(y_{\mathcal{T}_h},K),\ \ \ \ E^2_p({\mathcal{T}_h},\mathcal{T}_h):=\sum\limits_{K\in \mathcal{T}_h}E^2_p(p_{\mathcal{T}_h},K).
  \end{eqnarray*}
Moreover we also need the  Scott-Zhang operator(\cite{Fau}) $\Pi_{\mathcal{T}_h}:L^2(\Omega)\rightarrow \mathbb{V}_{\mathcal{T}_h}$ that satisfy the following properties
  \begin{align}
&(1): \Pi_{\mathcal{T}_h}v=v,\ \forall v\in  \mathbb{V}_{\mathcal{T}_h}.\label{sz1}\\
&(2):\|\Pi_{\mathcal{T}_h}v\|_{\widetilde{H}^{\frac{\alpha}{2}}(\Omega)}\leq \mathbb{C}_{sz} \|v\|_{\widetilde{H}^{\frac{\alpha}{2}}(\Omega)}, \forall v\in \widetilde{H}^{\frac{\alpha}{2}}(\Omega). \label{sz2}\\
&(3):\|\widetilde{h}^{-\frac{\alpha}{2}}_{\mathcal{T}_h}(1-\Pi_{\mathcal{T}_h})v\|\leq \mathbb{C}_{sz} \|v\|_{\widetilde{H}^{\frac{\alpha}{2}}(\Omega)}, \forall v\in \widetilde{H}^{\frac{\alpha}{2}}(\Omega). \label{sz3}
   \end{align}

 \begin{lemma}\label{upper-lower} For $0 < \alpha< 2$, $f+u_{\mathcal{T}_h}\in L^2(\Omega)$ and $y_{\mathcal{T}_h}-y_d\in L^2(\Omega)$ the weighted residual error estimator is reliable:
  \begin{eqnarray*}
\|\tilde{y}-y_{\mathcal{T}_h}\|_{\widetilde{H}^{\frac{\alpha}{2}}(\Omega)}\leq  \mathbb{C}_{yrel} E_y(y_{\mathcal{T}_h},\mathcal{T}_h),\ \|\tilde{p}-p_{\mathcal{T}_h}\|_{\widetilde{H}^{\frac{\alpha}{2}}(\Omega)}\leq \mathbb{C}_{prel} E_p(p_{\mathcal{T}_h},\mathcal{T}_h).
 \end{eqnarray*}
Moreover,  for $0<\alpha\leq 1$ and $\tilde{y},\ \tilde{p}\in H^{\frac{\alpha}{2}+\frac{1}{2}-\epsilon}(\Omega)\cap\widetilde{H}^{\frac{\alpha}{2}}(\Omega), 0\leq \epsilon< \min\{\frac{\alpha}{2}, \frac{1}{2}-\frac{\alpha}{2}\}$, the estimator is also efficient
  \begin{eqnarray*}
   E^2_y(y_{\mathcal{T}_h},\mathcal{T}_h)&\leq&  \mathbb{C}_{yeff} \Big(\|\tilde{y}-y_{\mathcal{T}_h}\|_{\widetilde{H}^{\frac{\alpha}{2}}(\Omega)}^2+\sum\limits_{K\in \mathcal{T}_h}h_K^{1-2\epsilon}\|\tilde{y}-y_{\mathcal{T}_h}\|_{H^{\frac{\alpha}{2}+\frac{1}{2}-\epsilon}(\Omega^3_h(K))}^2\Big),\\
    E^2_p(p_{\mathcal{T}_h},\mathcal{T}_h)&\leq& \mathbb{C}_{peff}   \Big(\|\tilde{p}-p_{\mathcal{T}_h}\|_{\widetilde{H}^{\frac{\alpha}{2}}(\Omega)}^2+\sum\limits_{K\in \mathcal{T}_h}h_K^{1-2\epsilon}\|\tilde{p}-p_{\mathcal{T}_h}\|_{H^{\frac{\alpha}{2}+\frac{1}{2}-\epsilon}(\Omega^3_h(K))}^2\Big).
 \end{eqnarray*}
 \end{lemma}
\begin{proof}
 Note that $y_{\mathcal{T}_h}$ and $p_{\mathcal{T}_h}$ are finite element approximations of $\tilde{y}$ and $\tilde{p}$. We refer the reader to \cite{Fau} for details on the proof of the upper and lower bounds in the Lemma.
 \end{proof}

 We define
   \begin{align}
%&a(\hat{y},v)=(f+\hat{u},v), \ \ \ \ \ \ \ \forall v\in \widetilde{H}^{s}(\Omega),\label{y2}\\
%&a(w,\hat{p})=(\hat{y}-y_{d},w), \ \ \ \ \ \ \ \forall w\in \widetilde{H}^{s}(\Omega),\label{p2}\\
&\hat{u}=\Pi_{[a,b]}\left(-\frac{1}{\alpha}(p_{\mathcal{T}_h}+\beta\hat{\lambda})\right),\label{fuzhu2}\\
 &\hat{\lambda}=\Pi_{[-1,1]}\left(-\frac{1}{\beta}p_{\mathcal{T}_h}\right)\label{hatlamda}.
   \end{align}
Here $\hat{\lambda}\in \partial j_2(\hat{u})$. $\hat{u}$ can be described similarly by
 \begin{eqnarray}\label{inequ}
 (p_{\mathcal{T}_h}+\gamma \hat{u}+\beta\hat{\lambda}, v-\hat{u})\geq0,\ \forall v\in U_{ad}.
  \end{eqnarray}
  Due to the variational
approach is considered, we have that $\hat{u}=u_{\mathcal{T}_h},\ \hat{\lambda}=\lambda_{\mathcal{T}_h}.$ Thus a posteriori error indicators and estimators for the optimal control variable and the associated subgradient are zero, i.e.,
 \begin{eqnarray}\label{EuElamda}
 E_{u}^2(u_{\mathcal{T}_h},K):=\|\hat{u}-u_{\mathcal{T}_h}\|^2=0,\ \ \ E_{\lambda}^2(\lambda_{\mathcal{T}_h},K):=\|\hat{\lambda}-\lambda_{\mathcal{T}_h}\|^2=0.
 \end{eqnarray}
 In the subsequent analysis, let $C$ represent a generic constant with distinct values in different instances.
 We define the errors $e_y=y-y_{\mathcal{T}_h},\ e_p=p-p_{\mathcal{T}_h}, e_u=u-u_{\mathcal{T}_h},\ e_\lambda=\lambda-\lambda_{\mathcal{T}_h},$ the vector $\mathbf{e}=(e_y,e_p,e_u,e_\lambda)^T$, and the norm
\begin{eqnarray}\label{e}
\|\mathbf{e}\|_{\Omega}^{2}=\|e_y\|_{\widetilde{H}^{\frac{\alpha}{2}}(\Omega)}^{2}+\|e_p\|_{\widetilde{H}^{\frac{\alpha}{2}}(\Omega)}^{2}+\|e_u\|+\|e_\lambda\|.
\end{eqnarray}
\subsection{Reliability of the error estimator $\mathcal{E}_{ocp}$}
\begin{theorem}\label{Reliability}
 Let $(y,p,u,\lambda)\in \widetilde{H}^{\frac{\alpha}{2}}(\Omega)\times \widetilde{H}^{\frac{\alpha}{2}}(\Omega)\times U_{ad}\times U_{ad}$ and $(y_{\mathcal{T}_h},p_{\mathcal{T}_h},u_{\mathcal{T}_h},\lambda_{\mathcal{T}_h})\in \mathbb{V}_{\mathcal{T}_h}\times \mathbb{V}_{\mathcal{T}_h}\times U_{ad}\times U_{ad}$ be the solutions of problems (\ref{zuiyou}) and (\ref{lisan}), respectively. Then the following upper bound of a posteriori error holds for $h<h_0\ll1$
   \begin{eqnarray*}
\|\mathbf{e}\|_{\Omega}^{2}\leq \mathcal{E}_{ocp}^2(y_{\mathcal{T}_h},p_{\mathcal{T}_h},\mathcal{T}_h).
 \end{eqnarray*}
 Here
 \begin{eqnarray}\label{E}
 \mathcal{E}_{ocp}^2(y_{\mathcal{T}_h},p_{\mathcal{T}_h},\mathcal{T}_h)=\sum\limits_{K\in \mathcal{T}_h, }\mathcal{E}_{K}^2(y_{\mathcal{T}_h},p_{\mathcal{T}_h},K),\ \mathcal{E}_{K}^2(y_{\mathcal{T}_h},p_{\mathcal{T}_h},K)=C_{st} E^2_y(y_{\mathcal{T}_h},K)+C_{ad} E^2_p(p_{\mathcal{T}_h},K).
  \end{eqnarray}
%Both the continuous and discrete optimal variables have no effect on the constants $C_{st}$ and $C_{ad}$.
\end{theorem}
\begin{proof}
We proceed in five steps.

$\underline{Step\ 1.}$
By applying the triangle inequality and Lemma \ref{upper-lower}, we can readily obtain
 \begin{eqnarray}\label{y-yth}
\|y-y_{\mathcal{T}_h}\|^2_{\widetilde{H}^{\frac{\alpha}{2}}(\Omega)}\leq 2\|y-\tilde{y}\|^2_{\widetilde{H}^{\frac{\alpha}{2}}(\Omega)}+2\|\tilde{y}-y_{\mathcal{T}_h}\|^2_{\widetilde{H}^{\frac{\alpha}{2}}(\Omega)}.
\end{eqnarray}
Moreover, by the coercivity of the bilinear form $a(\cdot,\cdot)$, we can derive
\begin{eqnarray*}
\|y-\tilde{y}\|_{\widetilde{H}^{\frac{\alpha}{2}}(\Omega)}\leq \|u-u_{\mathcal{T}_h}\|.
\end{eqnarray*}
This estimate combined with (\ref{y-yth}) imply that
 \begin{eqnarray}\label{y-yh}
\|y-y_{\mathcal{T}_h}\|^2_{\widetilde{H}^{\frac{\alpha}{2}}(\Omega)}
&\leq 2\|u-u_{\mathcal{T}_h}\|^2+2\|\tilde{y}-y_{\mathcal{T}_h}\|^2_{\widetilde{H}^{\frac{\alpha}{2}}(\Omega)}.
\end{eqnarray}

$\underline{Step\ 2.}$
In a similar way, we can obtain that
  \begin{align*}
\|p-p_{\mathcal{T}_h}\|^2_{\widetilde{H}^{\frac{\alpha}{2}}(\Omega)}&\leq 2\|p-\tilde{p}\|^2_{\widetilde{H}^{\frac{\alpha}{2}}(\Omega)}+2\|\tilde{p}-p_{\mathcal{T}_h}\|^2_{\widetilde{H}^{\frac{\alpha}{2}}(\Omega)}\nonumber\\
&\leq 2\|y-y_{\mathcal{T}_h}\|^2+2\|\tilde{p}-p_{\mathcal{T}_h}\|^2_{\widetilde{H}^{\frac{\alpha}{2}}(\Omega)}\nonumber\\
&\leq 2\|y-y_{\mathcal{T}_h}\|_{\widetilde{H}^{\frac{\alpha}{2}}(\Omega)}^2+2\|\tilde{p}-p_{\mathcal{T}_h}\|^2_{\widetilde{H}^{\frac{\alpha}{2}}(\Omega)}.
\end{align*}
%To complete the estimate of $\|p-p_{\mathcal{T}_h}\|_{\widetilde{H}^{\frac{\alpha}{2}}(\Omega)}$, we need to compute $\|y-y_{\mathcal{T}_h}\|$ using the dual argument in the following analysis.
%   Let $\psi$ be the solution of the following problem
%\begin{eqnarray}\nonumber\left\{ \begin{aligned}
% (-\Delta)^s \psi&=y-y_{\mathcal{T}_h},& \mbox{in}\ \Omega,\\
%   \psi&=0, &\mbox{in}\ \Omega^c.
%\end{aligned}\right.
%   \end{eqnarray}
% In an analogous way we obtain
%  \begin{align*}
%  \| y-y_{\mathcal{T}_h}\|^2&=((-\Delta)^{\frac{\alpha}{2}} \psi,y-y_{\mathcal{T}_h})\\
%     &=a(\psi, y-y_{\mathcal{T}_h})\\
%     &=a(\psi-\psi_{\mathcal{T}_h}, y-y_{\mathcal{T}_h})+a(\psi_{\mathcal{T}_h}, y-y_{\mathcal{T}_h})\\
%     &=a(\psi-\psi_{\mathcal{T}_h}, y-y_{\mathcal{T}_h})+ (\psi_{\mathcal{T}_h}-\psi, u-u_{\mathcal{T}_h})+ (\psi, u-u_{\mathcal{T}_h})\\
%         &\leq\|\psi-\psi_{\mathcal{T}_h}\|_{\widetilde{H}^{\frac{\alpha}{2}}(\Omega)}\| y-y_{\mathcal{T}_h}\|_{\widetilde{H}^{\frac{\alpha}{2}}(\Omega)} + \|\psi_{\mathcal{T}_h}-\psi\|\ \| u-u_{\mathcal{T}_h}\|+ \|\psi\|\ \|u-u_{\mathcal{T}_h}\|\\
%          &\leq C\Theta(h)\| y-y_{\mathcal{T}_h}\|\ \| y-y_{\mathcal{T}_h}\|_{\widetilde{H}^{\frac{\alpha}{2}}(\Omega)} +C\Theta^2(h)\| y-y_{\mathcal{T}_h}\|\ \| u-u_{\mathcal{T}_h}\|+ C \| y-y_{\mathcal{T}_h}\|\ \|u-u_{\mathcal{T}_h}\|.
%     \end{align*}
% Thus
% \begin{align}\label{y-yTh}
%\|y-y_{\mathcal{T}_h}\|^2&\leq C\Theta^2(h) \| y-y_{\mathcal{T}_h}\|^2_{\widetilde{H}^{\frac{\alpha}{2}}(\Omega)} +C \|u-u_{\mathcal{T}_h}\|^2.
%     \end{align}
  Therefore, (\ref{y-yh})
and the previous estimate allow us to deduce the a posteriori error estimate
\begin{eqnarray}\label{p-pTh}
\|p-p_{\mathcal{T}_h}\|^2_{\widetilde{H}^{\frac{\alpha}{2}}(\Omega)}
&\leq 4\|u-u_{\mathcal{T}_h}\|^2+4\|\tilde{y}-y_{\mathcal{T}_h}\|^2_{\widetilde{H}^{\frac{\alpha}{2}}(\Omega)}+ 2 \| \tilde{p}-p_{\mathcal{T}_h}\|^2_{\widetilde{H}^{\frac{\alpha}{2}}(\Omega)} .
     \end{eqnarray}

$\underline{Step\ 3.}$
The goal of this step is to estimate the error $\|u-u_{\mathcal{T}_h}\|$.
% \begin{eqnarray}\label{fuzhu3}
%a(y_{\mathcal{T}},w_{\mathcal{T}_h})&=(f+u,w)\ \forall w_{\mathcal{T}_h}\in \mathbb{V}_{\mathcal{T}_h},\\
%a(v_{\mathcal{T}_h},p_{\mathcal{T}})&=(y_{\mathcal{T}}-y_d,v_{\mathcal{T}_h})\ \forall v_{\mathcal{T}_h}\in \mathbb{V}_{\mathcal{T}_h}.
%\end{eqnarray}
%, and write $u-u_{\mathcal{T}_h}=u-u_{\mathcal{T}}+u_{\mathcal{T}}-u_{\mathcal{T}_h}$. Applying (\ref{Eu}), we can obtain
Setting $v =u$ in (\ref{lisan}) and $v =u_{\mathcal{T}_h}$ in (\ref{variational inequality}) we arrive at
 \begin{eqnarray*}
\gamma \|u-u_{\mathcal{T}_h}\|^2\leq(p-p_{\mathcal{T}_h},u_{\mathcal{T}_h}-u)+\beta(\lambda-\lambda_{\mathcal{T}_h},u_{\mathcal{T}_h}-u).
\end{eqnarray*}
Since $\lambda\in j_2(u)$ and $\lambda_{\mathcal{T}_h}\in j_2(u_{\mathcal{T}_h})$, in view of (\ref{subg}), implies that
 \begin{eqnarray*}
\beta(\lambda-\lambda_{\mathcal{T}_h},u_{\mathcal{T}_h}-u)\leq0.
\end{eqnarray*}
Thus we have that
 \begin{eqnarray}\label{u-u}
\gamma \|u-u_{\mathcal{T}_h}\|^2\leq(p-p_{\mathcal{T}_h},u_{\mathcal{T}_h}-u).
\end{eqnarray}
To control the right hand side of (\ref{u-u}), we now invoke the auxiliary states $p_{\mathcal{T}}$ that satisfy the following equation
\begin{align*}
a(y_{\mathcal{T}},v_{\mathcal{T}_h})&=(f+u,v_{\mathcal{T}_h}), \ \ \ \forall v_{\mathcal{T}_h}\in \mathbb{V}_{\mathcal{T}_h}, \\
a(w_{\mathcal{T}_h},p_{\mathcal{T}})&=(y_{\mathcal{T}}-y_{d},w_{\mathcal{T}_h}), \ \ \ \forall w_{\mathcal{T}_h}\in \mathbb{V}_{\mathcal{T}_h}.
\end{align*}
Then (\ref{u-u}) can be rewritten as
\begin{eqnarray}\label{u-uH}
\gamma \|u- u_{\mathcal{T}_h}\|^2\leq(p-p_{\mathcal{T}},u_{\mathcal{T}_h}-u)+(p_{\mathcal{T}}-p_{\mathcal{T}_h},u_{\mathcal{T}_h}-u).
\end{eqnarray}
Note that
\begin{align}\label{fz}
a(y_{\mathcal{T}}-y_{\mathcal{T}_h},v_{\mathcal{T}_h})&=(u-u_{\mathcal{T}_h},v_{\mathcal{T}_h}), \ \ \ \forall v_{\mathcal{T}_h}\in \mathbb{V}_{\mathcal{T}_h}, \\
a(w_{\mathcal{T}_h},p_{\mathcal{T}}-p_{\mathcal{T}_h})&=(y_{\mathcal{T}}-y_{\mathcal{T}_h},w_{\mathcal{T}_h}),\ \ \ \ \forall w_{\mathcal{T}_h}\in \mathbb{V}_{\mathcal{T}_h}.
\end{align}
Setting $v_{\mathcal{T}_h} =p_{\mathcal{T}}-p_{\mathcal{T}_h}$ and $w_{\mathcal{T}_h}= y_{\mathcal{T}}-y_{\mathcal{T}_h}$ in (\ref{fz}) yields
$$(u-u_{\mathcal{T}_h},p_{\mathcal{T}}-p_{\mathcal{T}_h})=a(y_{\mathcal{T}}-y_{\mathcal{T}_h},p_{\mathcal{T}}-p_{\mathcal{T}_h})
=(y_{\mathcal{T}}-y_{\mathcal{T}_h},y_{\mathcal{T}}-y_{\mathcal{T}_h})\geq 0.$$
From (\ref{u-uH}) and the above equation we have
\begin{eqnarray}\label{u-uth}
\|u-u_{\mathcal{T}_h}\|^2\leq \frac{1}{\gamma^2}\|p-p_{\mathcal{T}}\|^2.
\end{eqnarray}

$\underline{Step\ 4.}$
We now go to control $\|p-p_{\mathcal{T}}\|.$  To accomplish this task, we introduce the following problem
\begin{eqnarray}\nonumber\left\{ \begin{aligned}
 (-\Delta)^s \phi&=p_{\mathcal{T}}-p,&  \mbox{in}\ \Omega,\\
   \phi&=0, &\mbox{in}\ \Omega^c.
\end{aligned}\right.
   \end{eqnarray}
   Let $\phi_{\mathcal{T}_h}$ be the finite element approximation of $\phi$. Invoking Lemma \ref{etah}, we have that
\begin{eqnarray}\label{phi-phih}
\|\phi-\phi_{\mathcal{T}_h}\|_{\widetilde{H}^{\frac{\alpha}{2}}(\Omega)}\leq C\Theta(h)\|p-p_{\mathcal{T}}\|\ \mbox{and}\ \|\phi-\phi_{\mathcal{T}_h}\|\leq C\Theta^2(h)\|p-p_{\mathcal{T}}\|.
\end{eqnarray}
Note that $y_{\mathcal{T}}$ is the finite element approximation of $y$, by Lemma \ref{etah}, we immediately arrive at the estimate
\begin{eqnarray}\label{yT-y}
\|y- y_{\mathcal{T}}\|\leq C\Theta(h)\|y-y_{\mathcal{T}}\|_{\widetilde{H}^{\frac{\alpha}{2}}(\Omega)}.
\end{eqnarray}
We bound $\|p-p_{\mathcal{T}}\|$ in view of the previous inequality that
  \begin{align*}
\|p-p_{\mathcal{T}}\|^2 & =((-\Delta)^{\frac{\alpha}{2}} \phi,p-p_{\mathcal{T}})=a(\phi, p-p_{\mathcal{T}})=a(\phi-\phi_{\mathcal{T}_h},p-p_{\mathcal{T}})+a(\phi_{\mathcal{T}_h}, p-p_{\mathcal{T}})\\
     &=a(\phi-\phi_{\mathcal{T}_h}, p-p_{\mathcal{T}})+ (\phi_{\mathcal{T}_h}-\phi,y-y_{\mathcal{T}})+ (\phi,y-y_{\mathcal{T}})\\
         &\leq\|\phi-\phi_{\mathcal{T}_h}\|_{\widetilde{H}^{\frac{\alpha}{2}}(\Omega)}\|p-p_{\mathcal{T}}\|_{\widetilde{H}^{\frac{\alpha}{2}}(\Omega)} + \|\phi_{\mathcal{T}_h}-\phi\|\ \|y-y_{\mathcal{T}}\|+ \|\phi\|\ \|y-y_{\mathcal{T}}\|.
   \end{align*}
    This result combined with (\ref{phi-phih}) and (\ref{yT-y}) allows us to derive that
     \begin{align}
\|p-p_{\mathcal{T}}\|^2&\leq C\Theta(h)\left(\|p-p_{\mathcal{T}}\|_{\widetilde{H}^{\frac{\alpha}{2}}(\Omega)}+\| y-y_{\mathcal{T}}\|_{\widetilde{H}^{\frac{\alpha}{2}}(\Omega)} \right)\|p-p_{\mathcal{T}}\|+C\Theta^3 (h) \|y-y_{\mathcal{T}}\|_{\widetilde{H}^{\frac{\alpha}{2}}(\Omega)}\|p-p_{\mathcal{T}}\|.\nonumber
 \end{align}
Consequently,
    \begin{eqnarray*}
 \|p-p_{\mathcal{T}}\|\leq C \Theta(h)\left(\|p-p_{\mathcal{T}}\|_{\widetilde{H}^{\frac{\alpha}{2}}(\Omega)}  +   \| y-y_{\mathcal{T}}\|_{\widetilde{H}^{\frac{\alpha}{2}}(\Omega)} \right). \nonumber
     \end{eqnarray*}
     Thus, by (\ref{u-uth}) we have
 \begin{equation*}
 \|u-u_{\mathcal{T}_h}\|\leq \frac{C}{\gamma} \Theta(h)\left(\|p-p_{\mathcal{T}}\|_{\widetilde{H}^{\frac{\alpha}{2}}(\Omega)}  +\| y-y_{\mathcal{T}}\|_{\widetilde{H}^{\frac{\alpha}{2}}(\Omega)} \right).
     \end{equation*}

      Further, invoking the discrete state variable $y_{\mathcal{T}_h}$ and the discrete adjoint variable $p_{\mathcal{T}_h}$ in the previous inequality, we derive that
       \begin{align*}
\|u-u_{\mathcal{T}_h}\|
&\leq \frac{C}{\gamma} \Theta(h)\left(\|p-p_{\mathcal{T}_h}\|_{\widetilde{H}^{\frac{\alpha}{2}}(\Omega)}  +  \|p_{\mathcal{T}_h}-p_{\mathcal{T}}\|_{\widetilde{H}^{\frac{\alpha}{2}}(\Omega)} + \| y-y_{\mathcal{T}_h}\|_{\widetilde{H}^{\frac{\alpha}{2}}(\Omega)}+ \| y_{\mathcal{T}_h}-y_{\mathcal{T}}\|_{\widetilde{H}^{\frac{\alpha}{2}}(\Omega)} \right).
     \end{align*}
 We notice that $a(v_{\mathcal{T}_h}, p_{\mathcal{T}_h}-p_{\mathcal{T}} )=(y_{\mathcal{T}_h}-y_{\mathcal{T}},v_{\mathcal{T}_h})$ and  $a( y_{\mathcal{T}_h}-y_{\mathcal{T}} ,w_{\mathcal{T}_h})=(u_{\mathcal{T}_h}-u,w_{\mathcal{T}_h})$, we can obtain that
       \begin{align*}
\|u-u_{\mathcal{T}_h}\|&\leq \frac{C}{\gamma} \Theta(h)\left(\|p-p_{\mathcal{T}_h}\|_{\widetilde{H}^{\frac{\alpha}{2}}(\Omega)}  +  \|p_{\mathcal{T}_h}-p_{\mathcal{T}}\|_{\widetilde{H}^{\frac{\alpha}{2}}(\Omega)}  +   \| y-y_{\mathcal{T}_h}\|_{\widetilde{H}^{\frac{\alpha}{2}}(\Omega)}+ \| y_{\mathcal{T}_h}-y_{\mathcal{T}}\|_{\widetilde{H}^{\frac{\alpha}{2}}(\Omega)} \right)\\
& \leq \frac{C}{\gamma}\Theta(h)\left(\|y_{\mathcal{T}_h}-y_{\mathcal{T}}\|_{\widetilde{H}^{\frac{\alpha}{2}}(\Omega)}  +  \|p-p_{\mathcal{T}_h}\|_{\widetilde{H}^{\frac{\alpha}{2}}(\Omega)}  +   \|u_{\mathcal{T}_h}-u\|+ \|y-y_{\mathcal{T}_h}\|_{\widetilde{H}^{\frac{\alpha}{2}}(\Omega)} \right)\\
&\leq \frac{C}{\gamma}\Theta(h)\left(\|u_{\mathcal{T}_h}-u\| +  \|p-p_{\mathcal{T}_h}\|_{\widetilde{H}^{\frac{\alpha}{2}}(\Omega)} + \| y-y_{\mathcal{T}_h}\|_{\widetilde{H}^{\frac{\alpha}{2}}(\Omega)}\right ).
     \end{align*}
For  $h_0\ll1 $ such that $\Theta(h)\ll 1, h<h_0$, we can obtain
  \begin{eqnarray}\label{u-uTh}
 \|u-u_{\mathcal{T}_h}\|\leq \frac{C}{\gamma}\Theta(h)\left( \|p-p_{\mathcal{T}_h}\|_{\widetilde{H}^{\frac{\alpha}{2}}(\Omega)} + \| y- y_{\mathcal{T}_h}\|_{\widetilde{H}^{\frac{\alpha}{2}}(\Omega)} \right).
     \end{eqnarray}
Replacing this estimate in (\ref{y-yh}) and (\ref{p-pTh}) we arrive at
 \begin{eqnarray}\label{opy-yTh}
\|y-y_{\mathcal{T}_h}\|^2_{\widetilde{H}^{\frac{\alpha}{2}}(\Omega)}
\leq \frac{C}{\gamma^2}\Theta^2(h)\left( \|p-p_{\mathcal{T}_h}\|_{\widetilde{H}^{\frac{\alpha}{2}}(\Omega)} + \| y- y_{\mathcal{T}_h}\|_{\widetilde{H}^{\frac{\alpha}{2}}(\Omega)} \right)+C\|\tilde{y}-y_{\mathcal{T}_h}\|^2_{\widetilde{H}^{\frac{\alpha}{2}}(\Omega)}
\end{eqnarray}
and
\begin{align}\label{opp-pTh}
\|p-p_{\mathcal{T}_h}\|^2_{\widetilde{H}^{\frac{\alpha}{2}}(\Omega)}
&\leq \frac{C}{\gamma^2}\Theta^2(h)\left( \|p_{\mathcal{T}_h}-p\|^2_{\widetilde{H}^{\frac{\alpha}{2}}(\Omega)} + \| y_{\mathcal{T}_h}-y\|^2_{\widetilde{H}^{\frac{\alpha}{2}}(\Omega)} \right)+C\|\tilde{y}-y_{\mathcal{T}_h}\|^2_{\widetilde{H}^{\frac{\alpha}{2}}(\Omega)}+  C \| \tilde{p}-p_{\mathcal{T}_h}\|^2_{\widetilde{H}^{\frac{\alpha}{2}}(\Omega)} .
     \end{align}

     $\underline{Step\ 5.}$
Finally, we need to bound $\|\lambda-\lambda_{\mathcal{T}_h}\|$. By (\ref{lambda}) and (\ref{lamdath}) we have that
\begin{eqnarray*}
\|\lambda-\lambda_{\mathcal{T}_h}\|^2\leq \frac{C}{\beta^2}\|p-p_{\mathcal{T}_h}\|^2\leq \frac{C}{\beta^2}\|p-p_{\mathcal{T}_h}\|^2_{\widetilde{H}^{\frac{\alpha}{2}}(\Omega)}.
\end{eqnarray*}
Using (\ref{opp-pTh}), we can get
\begin{align*}
\|\lambda-\lambda_{\mathcal{T}_h}\|^2&\leq  \frac{C}{(\beta\gamma)^2}\Theta^2(h)\left( \|p- p_{\mathcal{T}_h}\|^2_{\widetilde{H}^{\frac{\alpha}{2}}(\Omega)} + \|y - y_{\mathcal{T}_h}\|^2_{\widetilde{H}^{\frac{\alpha}{2}}(\Omega)} \right)+\frac{C}{\beta^2}\|\tilde{y}-y_{\mathcal{T}_h}\|^2_{\widetilde{H}^{\frac{\alpha}{2}}(\Omega)}+ \frac{C}{\beta^2} \| \tilde{p}-p_{\mathcal{T}_h}\|^2_{\widetilde{H}^{\frac{\alpha}{2}}(\Omega)} .
\end{align*}
Thus by Lemma \ref{upper-lower} we further derive
\begin{align*}
&\quad\|y-y_{\mathcal{T}_h}\|^2_{\widetilde{H}^{\frac{\alpha}{2}}(\Omega)} +\|p-p_{\mathcal{T}_h}\|_{\widetilde{H}^{\frac{\alpha}{2}}(\Omega)} ^2+\|u-u_{\mathcal{T}_h}\|^2+\|\lambda-\lambda_{\mathcal{T}_h}\|^2\\
&\leq \frac{C}{\gamma^2}\Theta^2(h)\left( \|p-p_{\mathcal{T}_h}\|_{\widetilde{H}^{\frac{\alpha}{2}}(\Omega)} + \| y- y_{\mathcal{T}_h}\|_{\widetilde{H}^{\frac{\alpha}{2}}(\Omega)}  \right)+ \frac{C}{(\gamma\beta)^2}\Theta^2(h)\left( \|p-p_{\mathcal{T}_h}\|_{\widetilde{H}^{\frac{\alpha}{2}}(\Omega)} + \| y- y_{\mathcal{T}_h}\|_{\widetilde{H}^{\frac{\alpha}{2}}(\Omega)}  \right)\\
&\quad+C(1+\frac{1}{\beta^2})\|\tilde{y}-y_{\mathcal{T}_h}\|^2_{\widetilde{H}^{\frac{\alpha}{2}}(\Omega)}+  C(1+\frac{1}{\beta^2}) \| \tilde{p}-p_{\mathcal{T}_h}\|^2_{\widetilde{H}^{\frac{\alpha}{2}}(\Omega)} \\
& \leq \mathcal{E}^2_{ocp}({y_{\mathcal{T}_h},p_{\mathcal{T}_h},\mathcal{T}_h)},
\end{align*}
which completes the proof.
\end{proof}

\begin{remark}
Since the control variable is implicitly discretized, the error estimators with respect to  $u$ and $\lambda$ are zeros. By (\ref{opy-yTh}) and (\ref{opp-pTh}) we can derive
the estimate only for state and adjoint state
   \begin{eqnarray*}
\|\mathbf{\bar{e}}\|_{\Omega}^{2}\leq \mathcal{E}_{ocp}^2(y_{\mathcal{T}_h},p_{\mathcal{T}_h},\mathcal{T}_h).
 \end{eqnarray*}Here $\mathbf{\bar{e}}=(e_y,e_p)^T$.
\end{remark}
\subsection{Efficiency of the error estimator $\mathcal{E}_{ocp}$}
\begin{theorem}\label{Efficiency}
Suppose that $(y,u,p,\lambda)$ and $(y_h,u_h,p_h,\lambda_h)$ are the solutions of the optimal control problem of (\ref{zuiyou}) and (\ref{lisan}), respectively. If $\tilde{y},\ \tilde{p}\in {\widetilde{H}}^{\frac{1}{2}+{\frac{\alpha}{2}}-\epsilon}(\Omega)\cap\widetilde{H}^{\frac{\alpha}{2}}(\Omega)$, for some parameter $0\leq\epsilon<\min\{{\frac{\alpha}{2}},\frac{1}{2}-\frac{\alpha}{2}\},$ then we have the error estimator $\mathcal{E}_{ocp}$, defined as in (\ref{E}) satisfied the following lower bound for $h<h_0\ll1$
 \begin{eqnarray}\label{Eocp}
\mathcal{E}^2_{ocp}({y_{\mathcal{T}_h},p_{\mathcal{T}_h},\mathcal{T}_h)}&\leq C\|\mathbf{e}\|_{\Omega}^{2}+C\sum\limits_{K\in \mathcal{T}_h}h_K^{1-2\epsilon}\|y-y_{\mathcal{T}_h}\|_{H^{{\frac{\alpha}{2}}+\frac{1}{2}-\epsilon}(\Omega^3_h(K))}^2+C\sum\limits_{K\in \mathcal{T}_h}h_K^{1-2\epsilon}\|p-p_{\mathcal{T}_h}\|_{H^{{\frac{\alpha}{2}}+\frac{1}{2}-\epsilon}(\Omega^3_h(K))}^2.
     \end{eqnarray}
\end{theorem}
\begin{proof}
The Theorem \ref{upper-lower}, implies that
 \begin{eqnarray}\label{Eybound}
 \hspace{-0.8in}  E^2_y(y_{\mathcal{T}_h},\mathcal{T}_h)&\leq&  \mathbb{C}_{yeff} \Big(\|\tilde{y}-y_{\mathcal{T}_h}\|_{\widetilde{H}^{\frac{\alpha}{2}}(\Omega)}^2+\sum\limits_{K\in \mathcal{T}_h}h_K^{1-2\epsilon}\|\tilde{y}-y_{\mathcal{T}_h}\|_{H^{{\frac{\alpha}{2}}+\frac{1}{2}-\epsilon}(\Omega^3_h(K))}^2\Big),\\
 \hspace{-0.8in}  E^2_p(p_{\mathcal{T}_h},\mathcal{T}_h)&\leq&  \mathbb{C}_{peff} \Big(\|\tilde{p}-p_{\mathcal{T}_h}\|_{\widetilde{H}^{\frac{\alpha}{2}}(\Omega)}^2+\sum\limits_{K\in \mathcal{T}_h}h_K^{1-2\epsilon}\|\tilde{p}-p_{\mathcal{T}_h}\|_{H^{{\frac{\alpha}{2}}+\frac{1}{2}-\epsilon}(\Omega^3_h(K))}^2\Big).
 \end{eqnarray}
 To bound $\|\tilde{y}-y_{\mathcal{T}_h}\|_{\widetilde{H}^{\frac{\alpha}{2}}(\Omega)}$, by (\ref{u-uTh}), we obtain that
\begin{align}\label{ywan-yth}
\|\tilde{y}-y_{\mathcal{T}_h}\|^2_{\widetilde{H}^{\frac{\alpha}{2}}(\Omega)}&\leq C\|\tilde{y}-y\|^2_{\widetilde{H}^{\frac{\alpha}{2}}(\Omega)}+C\|y-y_{\mathcal{T}_h}\|^2_{\widetilde{H}^{\frac{\alpha}{2}}(\Omega)}\nonumber\\
 &\leq C\|u_{\mathcal{T}_h}-u\|^2+C\|y-y_{\mathcal{T}_h}\|^2_{\widetilde{H}^{\frac{\alpha}{2}}(\Omega)}\nonumber\\
 &\leq  \frac{C}{\gamma^2}\Theta^2(h)\left( \|p-p_{\mathcal{T}_h}\|^2_{\widetilde{H}^{\frac{\alpha}{2}}(\Omega)} + \| y-y_{\mathcal{T}_h}\|^2_{\widetilde{H}^{\frac{\alpha}{2}}(\Omega)} \right)+C\|y-y_{\mathcal{T}_h}\|^2_{\widetilde{H}^{\frac{\alpha}{2}}(\Omega)}.
\end{align}
We can deal with the second term of (\ref{Eybound}) in a similar way
\begin{align*}
&\quad\sum\limits_{K\in \mathcal{T}_h}h_K^{1-2\epsilon}\|\tilde{y}-y_{\mathcal{T}_h}\|_{H^{{\frac{\alpha}{2}}+\frac{1}{2}-\epsilon}(\Omega^3_h(K))}^2\\
&\leq C\sum\limits_{K\in \mathcal{T}_h}h_K^{1-2\epsilon}\|\tilde{y}-y\|_{H^{{\frac{\alpha}{2}}+\frac{1}{2}-\epsilon}(\Omega^3_h(K))}^2+C\sum\limits_{K\in \mathcal{T}_h}h_K^{1-2\epsilon}\|y-y_{\mathcal{T}_h}\|_{H^{{\frac{\alpha}{2}}+\frac{1}{2}-\epsilon}(\Omega^3_h(K))}^2\\
&\leq CMh^{1-2\epsilon}\|\tilde{y}-y\|_{H^{{\frac{\alpha}{2}}+\frac{1}{2}-\epsilon}(\Omega)}^2+C\sum\limits_{K\in \mathcal{T}_h}h_K^{1-2\epsilon}\|y-y_{\mathcal{T}_h}\|_{H^{{\frac{\alpha}{2}}+\frac{1}{2}-\epsilon}(\Omega^3_h(K))}^2\\
&\leq \frac{C}{\gamma^2}Mh^{1-2\epsilon} \Theta^2(h)\left( \|p-p_{\mathcal{T}_h}\|^2_{\widetilde{H}^{\frac{\alpha}{2}}(\Omega)} + \| y-y_{\mathcal{T}_h}\|^2_{\widetilde{H}^{\frac{\alpha}{2}}(\Omega)}\right )+C\sum\limits_{K\in \mathcal{T}_h}h_K^{1-2\epsilon} \|y-y_{\mathcal{T}_h}\|^2_{H^{{\frac{\alpha}{2}}+\frac{1}{2}-\epsilon}(\Omega^3_h(K))},
\end{align*}
where $M$ denotes the  maximum times of an element $K$ appearing in all element patch $\Omega^3_h(K)$.
On the basis of (\ref{Eybound}) and the previous estimate, we immediately obtain the local efficiency of $E_y$
\begin{align*}
&E^2_y(y_{\mathcal{T}_h},\mathcal{T}_h)\\
 &\leq\mathbb{C}_{yeff} \bigg\{ \|p_{\mathcal{T}_h}-p\|^2_{\widetilde{H}^{\frac{\alpha}{2}}(\Omega)} + \| y_{\mathcal{T}_h}-y\|^2_{\widetilde{H}^{\frac{\alpha}{2}}(\Omega)}
+\sum\limits_{K\in \mathcal{T}_h}h_K^{1-2\epsilon} \|y-y_{\mathcal{T}_h}\|^2_{H^{{\frac{\alpha}{2}}+\frac{1}{2}-\epsilon}(\Omega^3_h(K))}\\
&\quad+\frac{C}{\gamma^2}\Theta^2(h)\left( \|p-p_{\mathcal{T}_h}\|^2_{\widetilde{H}^{\frac{\alpha}{2}}(\Omega)} + \| y-y_{\mathcal{T}_h}\|^2_{\widetilde{H}^{\frac{\alpha}{2}}(\Omega)} \right)+CMh^{1-2\epsilon} \Theta^2(h) \left(\|p- p_{\mathcal{T}_h}\|^2_{\widetilde{H}^{\frac{\alpha}{2}}(\Omega)} + \|y- y_{\mathcal{T}_h}\|^2_{\widetilde{H}^{\frac{\alpha}{2}}(\Omega)}\right)\bigg\}.
\end{align*}
 Assuming that the initial size of the mesh fulfills the following condition:  $Mh_0^{1-2\epsilon}\Theta^2(h_0)\leq C.$ For  $h_0\ll1 $, we can obtain
\begin{align}\label{Ey}
 E^2_y(y_{\mathcal{T}_h},\mathcal{T}_h)\leq C \left( \|p- p_{\mathcal{T}_h}\|^2_{\widetilde{H}^{\frac{\alpha}{2}}(\Omega)} + \|y - y_{\mathcal{T}_h}\|^2_{\widetilde{H}^{\frac{\alpha}{2}}(\Omega)}  + \sum\limits_{K\in \mathcal{T}_h}h_K^{1-2\epsilon} \|y-y_{\mathcal{T}_h}\|^2_{H^{{\frac{\alpha}{2}}+\frac{1}{2}-\epsilon}(\Omega^3_h(K))}\right).
\end{align}
Note that
\begin{eqnarray*}
 \|\tilde{p}-p_{\mathcal{T}_h}\|^2_{\widetilde{H}^{\frac{\alpha}{2}}(\Omega)}&\leq C\|\tilde{p}-p\|^2_{\widetilde{H}^{\frac{\alpha}{2}}(\Omega)}+C\|p-p_{\mathcal{T}_h}\|^2_{\widetilde{H}^{\frac{\alpha}{2}}(\Omega)}\\
 &\leq C\|y_{\mathcal{T}_h}-y\|^2+C\|p-p_{\mathcal{T}_h}\|^2_{\widetilde{H}^{\frac{\alpha}{2}}(\Omega)}.
\end{eqnarray*}
   We need to estimate $\|y-y_{\mathcal{T}_h}\|$ using the dual argument in the following analysis.
   Let $\psi$ be the solution of the following problem
\begin{eqnarray}\nonumber\left\{ \begin{aligned}
 (-\Delta)^s \psi&=y-y_{\mathcal{T}_h},& \mbox{in}\ \Omega,\\
   \psi&=0, &\mbox{in}\ \Omega^c.
\end{aligned}\right.
   \end{eqnarray}
 In an analogous way we obtain
  \begin{align*}
\| y-y_{\mathcal{T}_h}\|^2&=((-\Delta)^{\frac{\alpha}{2}} \psi,y-y_{\mathcal{T}_h})=a(\psi, y-y_{\mathcal{T}_h})\\
&=a(\psi-\psi_{\mathcal{T}_h}, y-y_{\mathcal{T}_h})+a(\psi_{\mathcal{T}_h}, y-y_{\mathcal{T}_h})\\
&=a(\psi-\psi_{\mathcal{T}_h}, y-y_{\mathcal{T}_h})+ (\psi_{\mathcal{T}_h}-\psi, u-u_{\mathcal{T}_h})+ (\psi, u-u_{\mathcal{T}_h})\\
&\leq\|\psi-\psi_{\mathcal{T}_h}\|_{\widetilde{H}^{\frac{\alpha}{2}}(\Omega)}\| y-y_{\mathcal{T}_h}\|_{\widetilde{H}^{\frac{\alpha}{2}}(\Omega)} + \|\psi_{\mathcal{T}_h}-\psi\|\ \| u-u_{\mathcal{T}_h}\|+ \|\psi\|\ \|u-u_{\mathcal{T}_h}\|\\
&\leq C\Theta(h)\| y-y_{\mathcal{T}_h}\| \| y-y_{\mathcal{T}_h}\|_{\widetilde{H}^{\frac{\alpha}{2}}(\Omega)} +C\Theta^2(h)\| y-y_{\mathcal{T}_h}\|\ \| u-u_{\mathcal{T}_h}\|+ C \| y-y_{\mathcal{T}_h}\|\ \|u-u_{\mathcal{T}_h}\|.
     \end{align*}
 Then we can get
 \begin{align}\label{yythor}
\|y-y_{\mathcal{T}_h}\|^2&\leq C\Theta^2(h) \| y-y_{\mathcal{T}_h}\|^2_{\widetilde{H}^{\frac{\alpha}{2}}(\Omega)} +C \|u-u_{\mathcal{T}_h}\|^2\nonumber\\
&\leq C\Theta^2(h) \| y-y_{\mathcal{T}_h}\|^2_{\widetilde{H}^{\frac{\alpha}{2}}(\Omega)} + \frac{C}{\gamma^2}\Theta^2(h)\left( \|p- p_{\mathcal{T}_h}\|^2_{\widetilde{H}^{\frac{\alpha}{2}}(\Omega)} + \|y- y_{\mathcal{T}_h}\|^2_{\widetilde{H}^{\frac{\alpha}{2}}(\Omega)}\right )\nonumber\\
&\leq C(1+\frac{1}{\gamma^2})\Theta^2(h) \| y-y_{\mathcal{T}_h}\|^2_{\widetilde{H}^{\frac{\alpha}{2}}(\Omega)} + \frac{C}{\gamma^2}\Theta^2(h) \|p- p_{\mathcal{T}_h}\|^2_{\widetilde{H}^{\frac{\alpha}{2}}(\Omega)}.
     \end{align}
     Using the previous inequality, we can obtain
     \begin{align*}
\|\tilde{p}-p_{\mathcal{T}_h}\|^2_{\widetilde{H}^{\frac{\alpha}{2}}(\Omega)}
 &\leq  C(1+\frac{1}{\gamma^2})\Theta^2(h) \| y-y_{\mathcal{T}_h}\|^2_{\widetilde{H}^{\frac{\alpha}{2}}(\Omega)} + C(1+\frac{1}{\gamma^2}\Theta^2(h)) \|p-p_{\mathcal{T}_h}\|^2_{\widetilde{H}^{\frac{\alpha}{2}}(\Omega)}
\end{align*}
and
\begin{align*}
&\quad\sum\limits_{K\in \mathcal{T}_h}h_K^{1-2\epsilon}\|\tilde{p}-p_{\mathcal{T}_h}\|_{H^{{\frac{\alpha}{2}}+\frac{1}{2}-\epsilon}(\Omega^3_h(K))}^2\\
&\leq C\sum\limits_{K\in \mathcal{T}_h}h_K^{1-2\epsilon}\|\tilde{p}-p\|_{H^{{\frac{\alpha}{2}}+\frac{1}{2}-\epsilon}(\Omega^3_h(K))}^2+C\sum\limits_{K\in \mathcal{T}_h}h_K^{1-2\epsilon}\|p-p_{\mathcal{T}_h}\|_{H^{{\frac{\alpha}{2}}+\frac{1}{2}-\epsilon}(\Omega^3_h(K))}^2\\
&\leq CMh^{1-2\epsilon}\|\tilde{p}-p\|_{H^{{\frac{\alpha}{2}}+\frac{1}{2}-\epsilon}(\Omega)}^2+C\sum\limits_{K\in \mathcal{T}_h}h_T^{1-2\epsilon}\|p-p_{\mathcal{T}_h}\|_{H^{{\frac{\alpha}{2}}+\frac{1}{2}-\epsilon}(\Omega^3_h(K))}^2\\
&\leq CMh^{1-2\epsilon}\|y_{\mathcal{T}_h}-y\|^2+C\sum\limits_{K\in \mathcal{T}_h}h_T^{1-2\epsilon}\|p-p_{\mathcal{T}_h}\|_{H^{{\frac{\alpha}{2}}+\frac{1}{2}-\epsilon}(\Omega^3_h(K))}^2\\
&\leq CMh^{1-2\epsilon} \Theta^2(h)\left( (1+\frac{1}{\gamma^2}) \| y-y_{\mathcal{T}_h}\|^2_{\widetilde{H}^{\frac{\alpha}{2}}(\Omega)} + \frac{1 }{\gamma^2} \|p- p_{\mathcal{T}_h}\|^2_{\widetilde{H}^{\frac{\alpha}{2}}(\Omega)} \right) +C\sum\limits_{K\in \mathcal{T}_h}h_K^{1-2\epsilon} \|p-p_{\mathcal{T}_h}\|^2_{H^{{\frac{\alpha}{2}}+\frac{1}{2}-\epsilon}(\Omega^3_h(K))}.
\end{align*}
Thus we have
\begin{align}\label{Ep}
 E^2_p(p_{\mathcal{T}_h},\mathcal{T}_h)\leq C\left( \|p- p_{\mathcal{T}_h}\|^2_{\widetilde{H}^{\frac{\alpha}{2}}(\Omega)} + \|y - y_{\mathcal{T}_h}\|^2_{\widetilde{H}^{\frac{\alpha}{2}}(\Omega)}+\sum\limits_{K\in \mathcal{T}_h}h_K^{1-2\epsilon} \|p-p_{\mathcal{T}_h}\|^2_{H^{{\frac{\alpha}{2}}+\frac{1}{2}-\epsilon}(\Omega^3_h(K))}\right).
\end{align}
Combining the estimate  (\ref{Ey}) and (\ref{Ep}) we derive
\begin{align*}
\mathcal{E}_{ocp}^2(y_{\mathcal{T}_h},p_{\mathcal{T}_h},\mathcal{T}_h)& = E^2_y(y_{\mathcal{T}_h},\mathcal{T}_h)+ E^2_p(p_{\mathcal{T}_h},\mathcal{T}_h)\\
& \leq C( \|e_p\|^2_{\widetilde{H}^{\frac{\alpha}{2}}(\Omega)}  +   \| e_y\|^2_{\widetilde{H}^{\frac{\alpha}{2}}(\Omega)} ) +C\sum\limits_{K\in \mathcal{T}_h}h_K^{1-2\epsilon} \|y-y_{\mathcal{T}_h}\|^2_{H^{{\frac{\alpha}{2}}+\frac{1}{2}-\epsilon}(\Omega^3_h(K))}+C\sum\limits_{K\in \mathcal{T}_h}h_K^{1-2\epsilon} \|p-p_{\mathcal{T}_h}\|^2_{H^{{\frac{\alpha}{2}}+\frac{1}{2}-\epsilon}(\Omega^3_h(K))}\\
& \leq C \|\mathbf{e}\|^2_{\Omega}  +C\sum\limits_{K\in \mathcal{T}_h}h_K^{1-2\epsilon} \|y-y_{\mathcal{T}_h}\|^2_{H^{{\frac{\alpha}{2}}+\frac{1}{2}-\epsilon}(\Omega^3_h(K))}+C\sum\limits_{K\in \mathcal{T}_h}h_K^{1-2\epsilon} \|p-p_{\mathcal{T}_h}\|^2_{H^{{\frac{\alpha}{2}}+\frac{1}{2}-\epsilon}(\Omega^3_h(K))}.
\end{align*}
This concludes the proof.
\end{proof}
\section{ AFEMs and convergence analysis}
 $\bullet$ The optimal control $u$, due to the sparsity term $j_2(u)$ in the cost functional, is sparse and has sparsely support sets within $\Omega$.\\
 $\bullet$ The fractional Laplacian operator $ (-\Delta)^{\frac{\alpha}{2}}$ is nonlocal (\cite{cab,caf,caf2}) and can lead to singularity of the state variable and the adjoint variable near the boundary (\cite{cap}), which leads to a lower convergence rate (\cite{ban,noch}).

To overcome these hurdle, adaptive mesh refinement methods can be employed.  The method facilitate more comprehensive mesh refinement in areas where the solution singularity is intense, and consequently improving the numerical solution's accuracy.
\subsection{AFEMs}
Utilizing the residual error estimator $\mathcal{E}_{ocp}^2(y_{\mathcal{T}_{h}},p_{\mathcal{T}_{h}},\mathcal{T}_{h})$ to measure local contributions, we explore an established technique for adaptive mesh refinement known as $\mathbf{SOLVE}-\mathbf{ESTIMATE}-\mathbf{MARK}-\mathbf{REFINE}$, which employs D$\rm{\ddot{o}}$rfler's marking criterion to designate elements for refinement.\\

\begin{algorithm}[!htbp]
\caption{Design\ of\ the\ AFEMs:}
 $\star\ \mathbf{SOLVE}$: Initial mesh ${\mathcal{T}_{h_{0}}}$ with mesh size $h_0$, constraints $a$ and $b$, regularization parameter $\gamma$, sparsity parameter $\beta$.  Set $k=0$ and solve   (\ref{lisan})  to obtain
 $$(y_{{\mathcal{T}_{h_k}}}, p_{{\mathcal{T}_{h_k}}}, u_{{\mathcal{T}_{h_k}}})=\mathbf{SOLVE}\left( \mathbb{V}_{{\mathcal{T}_{h_k}}}\times \mathbb{V}_{{\mathcal{T}_h}_k}\times U_{ad}\right).$$

 $\star\ \mathbf{ESTIMATE}$: Compute the local error indicator
$$
\mathcal{E}_{ocp}^2(y_{\mathcal{T}_{h_{k}}},p_{\mathcal{T}_{h_{k}}},\mathcal{T}_{h_{k}})=\sum\limits_{K\in \mathcal{T}_{h_{k}}}\left( E^2_y(y_{\mathcal{T}_{h_{k}}},\mathcal{T}_{h_{k}})+ E^2_p(p_{\mathcal{T}_{h_{k}}},\mathcal{T}_{h_{k}})\right)=\mathbf{ESTIMATE}\left(y_{\mathcal{T}_{h_{k}}},p_{\mathcal{T}_{h_{k}}},\mathcal{T}_{h_{k}}\right)
$$
defined by (\ref{Eydefin}), (\ref{Epdefin}) and (\ref{Eocp}).

 $\star\ \mathbf{MARK}$: Given a parameter $0<\theta<1$; Construct a minimal subset $\mathcal{M}_k\subset{\mathcal{T}_{h_k}}$ such that
  $$\mathcal{M}_k=\mathbf{MARK}\left\{\mathcal{E}_{ocp}^2(y_{\mathcal{T}_{h_{k}}},p_{\mathcal{T}_{h_{k}}},\mathcal{M}_k)\right\}\geq \theta\mathcal{E}_{ocp}^2(y_{\mathcal{T}_{h_{k}}},p_{\mathcal{T}_{h_{k}}},\mathcal{T}_{h_{k}}) .$$

 $\star\ \mathbf{ REFINE}$: We bisect all the elements $K\in\mathcal{T}_{h_k}$ that are contained in
$\mathcal{M}_k$ with the newest-vertex bisection method and create a new mesh
$\mathcal{T}_{h_{k+1}}$. Refine
$$\mathcal{M}_{k+1}=\mathbf{ REFINE}\left(\mathcal{M}_k\right).$$
\end{algorithm}

In the first step of $\mathbf{SOLVE}$, we used the following projection gradient algorithm:

\begin{algorithm}[!htbp]
\caption{Projection gradient algorithm}

$\mathbf{Input}$

  Start with the mesh $\mathcal{T}_{h_{t}}$ with mesh size $h_t$.

$\mathbf{Start}$

Given the initial value $ u^0_{\mathcal{T}_{h_{t}}}$,  and a tolerance
   $\mathrm{Tol}_{\mathrm{space}}>0$.

$\mathbf{While}$\ $error>\mathrm{Tol}_{\mathrm{space}}$

   $\mathbf{1.}$ \
   Solving the state equation in  (\ref{lisan}) to get state variable $y_{\mathcal{T}_{h_t}}$;

  $\mathbf{2.}$  \ Solving the adjoint state equation in  (\ref{lisan})  to obtain adjoint state variable $p_{\mathcal{T}_{h_t}}$;

  $\mathbf{3.}$ \ Using (\ref{lamdath}) to compute the associted subgradient and control variable
$$\lambda_{\mathcal{T}_{h_{t}}}=\min\{ 1, \max\{-1,-\frac{1}{\beta}p_{\mathcal{T}_{h_t}} \} \},\ u^{new}_{\mathcal{T}_{h_{t}}}=\min\{b,\max\{a,-\frac{1}{\gamma}(p_{\mathcal{T}_{h_t}}+\beta\lambda_{\mathcal{T}_{h_t}})\}\}.$$
$\mathbf{4.}$  Calculate the error:
 $ error=norm(u^0_{\mathcal{T}_{h_t}}-u^{new}_{\mathcal{T}_{h_{t}}},inf).$

$\mathbf{5.}$\ Update the control variable $u^0_{\mathcal{T}_{h_t}}=u^{new}_{\mathcal{T}_{h_{t}}}.$

$\mathbf{End\ While}$
\end{algorithm}

\subsection{Convergence analysis}
\ \ \ To establish the quasi-optimality of Adaptive Finite Element Methods (AFEMs), we employ the framework proposed by Carstensen et al. in \cite{Car}. The fulfillment of several prerequisites is necessary to establish quasi-optimality in adaptive algorithms: (1) Stability, (2) Reduction, (3) Discrete reliability and (4) Quasi-orthogonality. The stability prerequisite ensures the stability of error estimate on non-refined elements, while the reduction prerequisite guarantees a reduction in error on refined elements. The discrete reliability ensures the ability of the error estimators on refined elements to effectively control the error between coarse and fine grid solutions. The quasi-orthogonality prerequisite involves providing a measure for the relationship between the error estimators and the exact errors. These requirements will be rigorously validated through a series of mathematical proofs.
\begin{theorem}(Stability) \label{Stability}
We use $\mathcal{T}_h$ to denote the refinements of $\mathcal{T}_H$. For any subsets $\mathcal{U}\subset {\mathcal{T}_h}\cap \mathcal{T}_H$, there holds
\begin{eqnarray*}
\left|\left(\sum\limits_{K\in\mathcal{U}}\mathcal{E}_{K}^2(y_{\mathcal{T}_h},p_{\mathcal{T}_h},K)\right)^{\frac{1}{2}}- \left(\sum\limits_{K\in\mathcal{U}}\mathcal{E}_{K}^2(y_{\mathcal{T}_H},p_{\mathcal{T}_H},K)\right)^{\frac{1}{2}}\right|\leq C_{ stab}\left(\|y_{\mathcal{T}_h}-y_{\mathcal{T}_H}\|_{\widetilde{H}^{\frac{\alpha}{2}}(\Omega)}+\|p_{\mathcal{T}_h}-p_{\mathcal{T}_H}\|_{\widetilde{H}^{\frac{\alpha}{2}}(\Omega)}\right),
	\end{eqnarray*}
where the constant $C_{ stab}>0$.
\end{theorem}
\begin{proof}
From the definition of (\ref{Eydefin}), (\ref{Epdefin}) and (\ref{Eocp}) we can obtain
	\begin{align*}
&\left|\left(\sum\limits_{K\in\mathcal{U}}\mathcal{E}_{K}^2(y_{\mathcal{T}_h},p_{\mathcal{T}_h},K)\right)^{\frac{1}{2}}- \left(\sum\limits_{K\in\mathcal{U}}\mathcal{E}_{K}^2(y_{\mathcal{T}_H},p_{\mathcal{T}_H},K)\right)^{\frac{1}{2}}\right|\\
&\leq \|\widetilde{h}^{\frac{\alpha}{2}}_{\mathcal{T}_h}(f+u_{\mathcal{T}_h}-(-\Delta )^{\frac{\alpha}{2}}y_{\mathcal{T}_h})\|_{L^2(\omega)}+\|\widetilde{h}^{\frac{\alpha}{2}}_{\mathcal{T}_h}(y_{\mathcal{T}_h}-y_d-(-\Delta )^{\frac{\alpha}{2}}p_{\mathcal{T}_h})\|_{L^2(\omega)}\\
&\quad-\|\widetilde{h}^{\frac{\alpha}{2}}_{\mathcal{T}_h}(f+u_{\mathcal{T}_H}-(-\Delta )^{\frac{\alpha}{2}}y_{\mathcal{T}_H})\|_{L^2(\omega)}-\|\widetilde{h}^{\frac{\alpha}{2}}_{\mathcal{T}_h}(y_{\mathcal{T}_H}-y_d-(-\Delta )^{\frac{\alpha}{2}}p_{\mathcal{T}_H})\|_{L^2(\omega)}\\
&\leq \|\widetilde{h}^{\frac{\alpha}{2}}_{\mathcal{T}_h}(-\Delta )^{\frac{\alpha}{2}}(y_{\mathcal{T}_H}-y_{\mathcal{T}_h})\|_{L^2(\omega)}+\|\widetilde{h}^{\frac{\alpha}{2}}_{\mathcal{T}_h}(-\Delta )^{\frac{\alpha}{2}}(p_{\mathcal{T}_H}-p_{\mathcal{T}_h})\|_{L^2(\omega)} +\|\widetilde{h}^{\frac{\alpha}{2}}_{\mathcal{T}_h}(u_{\mathcal{T}_h}-u_{\mathcal{T}_H})\|_{L^2(\omega)}+\|\widetilde{h}^{\frac{\alpha}{2}}_{\mathcal{T}_h}(y_{\mathcal{T}_h}-y_{\mathcal{T}_H})\|_{L^2(\omega)},
	\end{align*}
where $\omega:=\mathrm{interior}(\bigcup\limits_{{K}\in \mathcal{U}}\overline{K}).$ Note that
$u_{\mathcal{T}_h}=\Pi_{[a,b]}\left(-\frac{1}{\gamma}(p_{\mathcal{T}_h}+\beta\lambda_{\mathcal{T}_h})\right)$
and
$u_{\mathcal{T}_H}=\Pi_{[a,b]}\left(-\frac{1}{\gamma}(p_{\mathcal{T}_H}+\beta\lambda_{\mathcal{T}_H})\right).$
By the Lipschitz continuity of the operator $\Pi_{[a,b]}$, we have that
\begin{align*}
\|u_{\mathcal{T}_h}-u_{\mathcal{T}_H}\|&=\left\|\Pi_{[a,b]}\left(-\frac{1}{\gamma}(p_{\mathcal{T}_h}+\beta\lambda_{\mathcal{T}_h})\right)
-\Pi_{[a,b]}\left(-\frac{1}{\gamma}(p_{\mathcal{T}_H}+\beta\lambda_{\mathcal{T}_H})\right)\right\|\\
&\leq\frac{C}{\gamma}\|p_{\mathcal{T}_H}-p_{\mathcal{T}_h}\|+\frac{C\beta}{\gamma}\|\lambda_{\mathcal{T}_H}-\lambda_{\mathcal{T}_h}\|.
	\end{align*}
An application of
$\lambda_{\mathcal{T}_h}=\Pi_{[-1,1]}\left(-\frac{1}{\beta}p_{\mathcal{T}_h}\right)$
 and
 $ \lambda_{\mathcal{T}_H}=\Pi_{[-1,1]}\left(-\frac{1}{\beta}p_{\mathcal{T}_H}\right)$ yields
\begin{eqnarray*}
\|u_{\mathcal{T}_h}-u_{\mathcal{T}_H}\|\leq \frac{C}{\gamma}\|p_{\mathcal{T}_H}-p_{\mathcal{T}_h}\|\leq \frac{C}{\gamma} \|p_{\mathcal{T}_H}-p_{\mathcal{T}_h}\|_{\widetilde{H}^{\frac{\alpha}{2}}(\Omega)}.
	\end{eqnarray*}
Further by the inverse estimate for fractional Laplacian (\cite{Fau}) we have
\begin{align*}
\|\widetilde{h}^{\frac{\alpha}{2}}_{\mathcal{T}_h}(-\Delta )^{\frac{\alpha}{2}}y_{{\mathcal{T}_h}}\|_{L^2(\Omega)}\leq C\|y_{\mathcal{T}_h}\|_{\widetilde{H}^{\frac{\alpha}{2}}(\Omega)}\ \mbox{and} \  \|\widetilde{h}^{\frac{\alpha}{2}}_{\mathcal{T}_h}(-\Delta )^{\frac{\alpha}{2}}p_{{\mathcal{T}_h}}\|_{L^2(\Omega)}\leq C\|p_{\mathcal{T}_h}\|_{\widetilde{H}^{\frac{\alpha}{2}}(\Omega)}.
 	\end{align*}
This result allows us to derive that
\begin{eqnarray*}
\left|\left(\sum\limits_{K\in\mathcal{U}}\mathcal{E}_{K}^2(y_{\mathcal{T}_h},p_{\mathcal{T}_h},K)\right)^{\frac{1}{2}}- \left(\sum\limits_{K\in\mathcal{U}}\mathcal{E}_{K}^2(y_{\mathcal{T}_H},p_{\mathcal{T}_H},K)\right)^{\frac{1}{2}}\right|\leq C_{ stab}\left(\|y_{\mathcal{T}_h}-y_{\mathcal{T}_H}\|_{\widetilde{H}^{\frac{\alpha}{2}}(\Omega)}+\|p_{\mathcal{T}_h}-p_{\mathcal{T}_H}\|_{\widetilde{H}^{\frac{\alpha}{2}}(\Omega)}\right).
	\end{eqnarray*}
\end{proof}

\begin{theorem}(Reduction)\label{thtH} We use $\mathcal{T}_h$ to denote the refinements of $\mathcal{T}_H$. Then we have
\begin{eqnarray*}
	\mathcal{E}_{ocp}^2(y_{\mathcal{T}_h},p_{\mathcal{T}_h},\mathcal{T}_{h}\setminus\mathcal{T}_{H})\leq Q_{ red}\ \mathcal{E}_{ocp}^2(y_{\mathcal{T}_H},p_{\mathcal{T}_H},\mathcal{T}_{H}\setminus\mathcal{T}_{h})+C_{ red}\left(\|y_{\mathcal{T}_h}-y_{\mathcal{T}_H}\|_{\widetilde{H}^{\frac{\alpha}{2}}(\Omega)}^2+\|p_{\mathcal{T}_h}-p_{\mathcal{T}_H}\|_{\widetilde{H}^{\frac{\alpha}{2}}(\Omega)}^2\right),\nonumber
	\end{eqnarray*}
where the constant $C_{red}>0$,\ $Q_{red}=2^{-\frac{\rho\alpha}{2d}},$ for $0<\alpha\leq1;$ $Q_{red}=2^{-\frac{\rho(\alpha-2\sigma)}{2d}},$ for $1<\alpha<2.$ Here $0<\sigma=\frac{\alpha}{2}-\frac{1}{2}<\frac{\alpha}{2}$ and $\frac{\alpha}{2}-\sigma>0.$
\end{theorem}
\begin{proof}
Bisection ensures that $|K'|\leq|\frac{K}{2}|$ for any $K\in \mathcal{T}_H\backslash \mathcal{T}_h$ and its descendants $K'\in \mathcal{T}_h\backslash \mathcal{T}_H$ with $K'\subset K.$ Note that \begin{eqnarray}\label{H1}\widetilde{h}^{\frac{\alpha}{2}}_{K'}=(|K'|^{\frac{1}{d}})^{\frac{\alpha}{2}}\leq (2^{-\rho}|K|)^{\frac{\alpha}{2d}}=2^{-\frac{\rho\alpha}{2d}}\widetilde{h}^{\frac{\alpha}{2}}_{K}, \ \ \ \mathrm{for}\ \ 0<\alpha\leq1, \end{eqnarray}
we prove the Theorem with $Q_{red}=2^{-\frac{\rho\alpha}{2d}}.$ Here  $\rho$ denotes the bisection time of every element $K\in \mathcal{T}_H$ in the refinement.
By the definition of (\ref{Eydefin}), the relationship between $\mathcal{T}_h$ and $\mathcal{T}_H$ we can get
\begin{align}\label{ythtH}
\left(\sum\limits_{K'\in\mathcal{T}_h\backslash {\mathcal{T}_H}}E_{y}^2(y_{{\mathcal{T}_h}},K')\right)^{\frac{1}{2}}
&=\left(\sum\limits_{K'\in\mathcal{T}_h\backslash {\mathcal{T}_H}}\int_{L^2(K')}\widetilde{h}^{\alpha}_{K'}(f+u_{\mathcal{T}_h}-(-\Delta )^{\frac{\alpha}{2}}y_{\mathcal{T}_h})^2\right)^{\frac{1}{2}}\nonumber\\
&=\left(\sum\limits_{K'\in\mathcal{T}_h\backslash {\mathcal{T}_H}}|K'|^{\frac{\alpha}{d}}\|f+u_{\mathcal{T}_h}-(-\Delta )^{\frac{\alpha}{2}}y_{\mathcal{T}_h}\|^2_{L^2(K')}\right)^{\frac{1}{2}}\nonumber\\
&\leq2^{-\frac{\rho\alpha}{2d}}\left(\sum\limits_{K\in\mathcal{T}_H\backslash {\mathcal{T}_h}}|K|^{\frac{\alpha}{d}}\|f+u_{\mathcal{T}_H}-(-\Delta )^{\frac{\alpha}{2}}y_{\mathcal{T}_H}\|^2_{L^2(K)}\right)^{\frac{1}{2}}\nonumber\\
&= 2^{-\frac{\rho\alpha}{2d}}\left(\sum\limits_{K\in\mathcal{T}_H\backslash {\mathcal{T}_h}}E_{y}^2(y_{{\mathcal{T}_H}},K)\right)^{\frac{1}{2}}.
\end{align}
Similarly,
\begin{align*}
\left(\sum\limits_{K'\in\mathcal{T}_h\backslash {\mathcal{T}_H}}E_{p}^2(p_{{\mathcal{T}_h}},K')\right)^{\frac{1}{2}}
\leq 2^{-\frac{\rho\alpha}{2d}}\left(\sum\limits_{K\in\mathcal{T}_H\backslash {\mathcal{T}_h}}E_{p}^2(p_{{\mathcal{T}_H}},K)\right)^{\frac{1}{2}}.
\end{align*}
Then we have
\begin{align*}
\left(\sum\limits_{K'\in\mathcal{T}_h\backslash {\mathcal{T}_H}}\mathcal{E}_{K}^2(y_{\mathcal{T}_h},p_{\mathcal{T}_h},K')\right)^{\frac{1}{2}}
\leq 2^{-\frac{\rho\alpha}{2d}}\left(\sum\limits_{K\in\mathcal{T}_H\backslash {\mathcal{T}_h}}\mathcal{E}_{K}^2(y_{\mathcal{T}_H},p_{\mathcal{T}_H},K)\right)^{\frac{1}{2}}.
\end{align*}
Therefore, the previous estimate allows us to deduce the reduction property on the refined elements
\begin{align*}
&\quad\sum\limits_{K\in\mathcal{T}_h\backslash {\mathcal{T}_H}}\mathcal{E}_{K}^2(y_{\mathcal{T}_h},p_{\mathcal{T}_h},K)\\
&\leq\sum\limits_{K\in\mathcal{T}_h\backslash {\mathcal{T}_H}}\mathcal{E}_{K}^2(y_{\mathcal{T}_h},p_{\mathcal{T}_h},K)-\sum\limits_{K\in\mathcal{T}_h\backslash {\mathcal{T}_H}}\mathcal{E}_{K}^2(y_{\mathcal{T}_H},p_{\mathcal{T}_H},K)+\sum\limits_{K\in\mathcal{T}_h\backslash {\mathcal{T}_H}}\mathcal{E}_{K}^2(y_{\mathcal{T}_H},p_{\mathcal{T}_H},K)\\
&\leq\sum\limits_{K\in\mathcal{T}_h\backslash {\mathcal{T}_H}}\left(\|\widetilde{h}^{\frac{\alpha}{2}}_{K}(f+u_{\mathcal{T}_h}-(-\Delta )^{\frac{\alpha}{2}}y_{\mathcal{T}_h})\|_{L^2(K)}+\|\widetilde{h}^{\frac{\alpha}{2}}_{K}(y_{\mathcal{T}_h}-y_d-(-\Delta )^{\frac{\alpha}{2}}p_{\mathcal{T}_h})\|_{L^2(K)}\right.\\
&\left.\quad-\|\widetilde{h}^{\frac{\alpha}{2}}_{K}(f+u_{\mathcal{T}_H}-(-\Delta )^{\frac{\alpha}{2}}y_{\mathcal{T}_H})\|_{L^2(K)}-\|\widetilde{h}^{\frac{\alpha}{2}}_{K}(y_{\mathcal{T}_H}-y_d-(-\Delta )^{\frac{\alpha}{2}}p_{\mathcal{T}_H})\|_{L^2(K)}\right)+\sum\limits_{K\in\mathcal{T}_h\backslash {\mathcal{T}_H}}\mathcal{E}_{K}^2(y_{\mathcal{T}_H},p_{\mathcal{T}_H},K)\\
&\leq C_{ stab}\left(\|y_{\mathcal{T}_h}-y_{\mathcal{T}_H}\|^2_{\widetilde{H}^{\frac{\alpha}{2}}(\Omega)}+\|p_{\mathcal{T}_h}-p_{\mathcal{T}_H}\|^2_{\widetilde{H}^{\frac{\alpha}{2}}(\Omega)}\right)+
\sum\limits_{K\in\mathcal{T}_h\backslash {\mathcal{T}_H}}\mathcal{E}_{K}^2(y_{\mathcal{T}_H},p_{\mathcal{T}_H},K)\\
&\leq C_{ stab}\left(\|y_{\mathcal{T}_h}-y_{\mathcal{T}_H}\|^2_{\widetilde{H}^{\frac{\alpha}{2}}(\Omega)}+\|p_{\mathcal{T}_h}-p_{\mathcal{T}_H}\|^2_{\widetilde{H}^{\frac{\alpha}{2}}(\Omega)}\right)
+ 2^{-\frac{\rho\alpha}{d}}\sum\limits_{K\in\mathcal{T}_H\backslash {\mathcal{T}_h}}\mathcal{E}_{K}^2(y_{\mathcal{T}_H},p_{\mathcal{T}_H},K).
\end{align*}
 For $1<\alpha<2,$ we note that $0<\sigma=\frac{\alpha}{2}-\frac{1}{2}<\frac{\alpha}{2}$ and $\frac{\alpha}{2}-\sigma>0.$ Moreover, $\omega:=\bigcup\limits_{K'\in\mathcal{T}_h\backslash \mathcal{T}_H}\overline{K'}=\bigcup\limits_{K\in\mathcal{T}_H\backslash \mathcal{T}_h}\overline{K},$ then we have
\begin{eqnarray}\label{H2}\widetilde{h}^{\frac{\alpha}{2}}_{K'}=(|K'|^{\frac{1}{d}})^{{\frac{\alpha}{2}}-\sigma}\omega_{\mathcal{T}_h}^{\sigma}\leq (2^{-\rho}|K|)^{\frac{\alpha-2\sigma}{2d}}\omega_{{\mathcal{T}}_H}^{\sigma}=2^{-\frac{\rho(\alpha-2\sigma)}{2d}}\widetilde{h}^{\frac{\alpha}{2}}_{K}. \end{eqnarray}
Arguing as before, we prove the Theorem with $Q_{red}=2^{-\frac{\rho(\alpha-2\sigma)}{2d}}.$
\end{proof}

\begin{remark}\label{r2}
According to (\ref{ythtH}), we can obtain
\begin{eqnarray}
\sum\limits_{K\in\mathcal{T}_{h}\setminus\mathcal{T}_{H}}E_y^2(y_{\mathcal{T}_{H}},K)\leq 2^{-\frac{\rho\xi}{d}}\sum\limits_{K\in\mathcal{T}_{H}\setminus\mathcal{T}_{h}}E_y^2(y_{\mathcal{T}_{H}},K).
\end{eqnarray}
Here $\xi=\alpha,\ 0<\alpha<1$ and $\xi=\alpha-2\sigma,\ 1<\alpha<2.$
Further we can derive
\begin{eqnarray*}
\hspace{-0.8in}(1-2^{-\frac{\rho\xi}{d}})\sum\limits_{K\in\mathcal{T}_{H}\setminus\mathcal{T}_{h}}E_y^2(y_{\mathcal{T}_{H}},K)&\leq \sum\limits_{K\in\mathcal{T}_{H}\setminus\mathcal{T}_{h}}E_y^2(y_{\mathcal{T}_{H}},K)-\sum\limits_{K\in\mathcal{T}_{h}\setminus\mathcal{T}_{H}}E_y^2(y_{\mathcal{T}_{H}},K)\\
&=\sum\limits_{K\in\mathcal{T}_{H}}E_y^2(y_{\mathcal{T}_{H}},K)-\sum\limits_{K\in\mathcal{T}_{h}}E_y^2(y_{\mathcal{T}_{H}},K).
\end{eqnarray*}
Thus, it implies that
\begin{eqnarray}
\hspace{-0.8in}\sum\limits_{K\in\mathcal{T}_{H}\setminus\mathcal{T}_{h}}E_y^2(y_{\mathcal{T}_{H}},K)\leq \frac{1}{1-2^{-\frac{\rho\xi}{d}}}
\left(\sum\limits_{K\in\mathcal{T}_{H}}E_y^2(y_{\mathcal{T}_{H}},K)-\sum\limits_{K\in\mathcal{T}_{h}}E_y^2(y_{\mathcal{T}_{H}},K)\right).
\end{eqnarray}
Similar arguments can be applied to $\sum\limits_{K\in\mathcal{T}_{H}\setminus\mathcal{T}_{h}}E_p^2(p_{\mathcal{T}_{H}},K)$. Using the definition of (\ref{E}), we can thus arrive at the estimate
\begin{eqnarray}\label{remark}
\hspace{-0.8in}\mathcal{E}_{ocp}^2(y_{\mathcal{T}_H},p_{\mathcal{T}_H},\mathcal{T}_{H}\setminus\mathcal{T}_{h})\leq \frac{1}{1-2^{-\frac{\rho\xi}{d}}}
\left(\mathcal{E}_{ocp}^2(y_{\mathcal{T}_H},p_{\mathcal{T}_H},\mathcal{T}_{H})-\mathcal{E}_{ocp}^2(y_{\mathcal{T}_H},p_{\mathcal{T}_H},\mathcal{T}_{h})\right).
\end{eqnarray}
\end{remark}
\begin{lemma}\label{ythk-y}
Set $\breve{y}=\mathcal{S}_{\mathcal{T}_{h}}(f+u_{\mathcal{T}_{H}})$ and $\breve{p}=\mathcal{S}^*_{\mathcal{T}_{h}}(\mathcal{S}_{\mathcal{T}_{H}}(f+u_{\mathcal{T}_{H}})-y_d)$. Then the following estimates hold
  \begin{eqnarray*}
\|y_{\mathcal{T}_{H}}-\breve{y}\|_{\widetilde{H}^{\frac{\alpha}{2}}(\Omega)}^2&\leq& \mathbb{C}_{yaux} \sum\limits_{K\in\mathcal{T}_{H}\setminus\mathcal{T}_{h}}E_y^2(y_{\mathcal{T}_{H}},K),\\
\| p_{\mathcal{T}_{H}}-\breve{p} \|_{\widetilde{H}^{\frac{\alpha}{2}}(\Omega)}^2&\leq&  \mathbb{C}_{paux} \sum\limits_{K\in\mathcal{T}_{H}\setminus\mathcal{T}_{h}}E_p^2(p_{\mathcal{T}_{H}},K).
 \end{eqnarray*}
\end{lemma}
\begin{proof}
 By the coercivity and Galerkin orthogonality we derive
   \begin{eqnarray*}
\|y_{\mathcal{T}_{H}}-\breve{y}\|_{\widetilde{H}^{\frac{\alpha}{2}}(\Omega)}^2
&\leq& a(\breve{y}-{y}_{\mathcal{T}_H},\breve{y}-{y}_{\mathcal{T}_H})\\
&\leq& a(\breve{y}-{y}_{\mathcal{T}_H},(1-\Pi_{\mathcal{T}_H})(\breve{y}-{y}_{\mathcal{T}_H})\\
&=& (f+u_{\mathcal{T}_H}-(-\Delta)^{\frac{\alpha}{2}}{y}_{\mathcal{T}_H},(1-\Pi_{\mathcal{T}_H})(\breve{y}-{y}_{\mathcal{T}_H})).
 \end{eqnarray*}
Assume $\omega:=\mathrm{interior}(\bigcup\limits_{{K}\in  \mathcal{T}_H\cap  \mathcal{T}_h }\bar{K}).$ We obtained by applying (\ref{sz1}) and (\ref{sz3}) in the above estimation
    \begin{align*}
\hspace{-1in}\|y_{\mathcal{T}_{H}}-\breve{y}\|_{\widetilde{H}^{\frac{\alpha}{2}}(\Omega)}^2
&\leq C\|\widetilde{h}^{\frac{\alpha}{2}}_{\mathcal{T}_H}(f+u_{\mathcal{T}_H}-(-\Delta)^{\frac{\alpha}{2}}{y}_{\mathcal{T}_H})\|_{L^2(\Omega\setminus\omega)}\|\widetilde{h}^{-{\frac{\alpha}{2}}}_{\mathcal{T}_H}(1-\Pi_{\mathcal{T}_H})(\breve{y}-{y}_{\mathcal{T}_H})\|_{L^2(\Omega\setminus\omega)}\\
&\quad+ C\|\widetilde{h}^{\frac{\alpha}{2}}_{\mathcal{T}_H}(f+u_{\mathcal{T}_H}-(-\Delta)^{\frac{\alpha}{2}}{y}_{\mathcal{T}_H})\|_{L^2(\omega)}\|\widetilde{h}^{-{\frac{\alpha}{2}}}_{\mathcal{T}_H}(1-\Pi_{\mathcal{T}_H})(\breve{y}-{y}_{\mathcal{T}_H})\|_{L^2(\omega)} \\
&\leq   C\|\widetilde{h}^{\frac{\alpha}{2}}_{\mathcal{T}_H}(f+u_{\mathcal{T}_H}-(-\Delta)^{\frac{\alpha}{2}}{y}_{\mathcal{T}_H})\|_{L^2(\Omega\setminus\omega)}\|\widetilde{h}^{-{\frac{\alpha}{2}}}_{\mathcal{T}_H}(1-\Pi_{\mathcal{T}_H})(\breve{y}-{y}_{\mathcal{T}_H})\|_{L^2(\Omega\setminus\omega)}\\
&\leq \left(\mathbb{C}_{yaux} \sum\limits_{K\in\mathcal{T}_{H}\setminus\mathcal{T}_{h}}E_y^2(y_{\mathcal{T}_{H}},K) \right)^{\frac{1}{2}} \|\widetilde{h}^{-\frac{\alpha}{2}}_{\mathcal{T}_{H}}(1-\Pi_{\mathcal{T}_H})(\breve{y}-{y}_{\mathcal{T}_H})\|_{L^2(\Omega\setminus \omega)}\\
&\leq \left(\mathbb{C}_{yaux} \sum\limits_{K\in\mathcal{T}_{H}\setminus\mathcal{T}_{h}}E_y^2(y_{\mathcal{T}_{H}},K) \right)^{\frac{1}{2}} \|\breve{y}-{y}_{\mathcal{T}_H}\|_{\widetilde{H}^{\frac{\alpha}{2}}(\Omega)},
 \end{align*}
which yields the first result. The second result can be derived in an analogous way.
\end{proof}

\begin{theorem}(Discrete\ reliability)
We use $\mathcal{T}_H$ to denote the refinements of $\mathcal{T}_h$. There holds
\begin{eqnarray*}
\quad\|y_{\mathcal{T}_{H}}-y_{\mathcal{T}_{h}}\|^2_{\widetilde{H}^{\frac{\alpha}{2}}(\Omega)}
+\|p_{\mathcal{T}_{H}}-p_{\mathcal{T}_{h}}\|^2_{\widetilde{H}^{\frac{\alpha}{2}}(\Omega)}\leq\mathcal{E}_{ocp}^2(y_{\mathcal{T}_H},p_{\mathcal{T}_H},\mathcal{T}_{H}\setminus\mathcal{T}_{h}).
\end{eqnarray*}
\end{theorem}
\begin{proof}
Taking $y_{\mathcal{T}_h}$ and $u_{\mathcal{T}_h}$ as the continuous solutions and $y_{\mathcal{T}_H}$ and $u_{\mathcal{T}_H}$ as its approximation, respectively. It can be deduced from
 the coercivity of $a(\cdot,\cdot)$, Galerkin orthogonality and Lemma \ref{ythk-y} that
	\begin{align*}
  \|y_{\mathcal{T}_H}-y_{\mathcal{T}_h}\|_{\widetilde{H}^{\frac{\alpha}{2}}(\Omega)}^2&\leq \|  {y}_{\mathcal{T}_H}-\breve{y}\|_{\widetilde{H}^{\frac{\alpha}{2}}(\Omega)}^2+ \| \breve{y}-{y}_{\mathcal{T}_h}\|_{\widetilde{H}^{\frac{\alpha}{2}}(\Omega)}^2\\
                  &\leq \mathbb{C}_{yaux} \sum\limits_{K\in\mathcal{T}_{H}\setminus\mathcal{T}_{h}}E_y^2(y_{\mathcal{T}_{H}},K)+\|u_{\mathcal{T}_H}-u_{\mathcal{T}_h}\|^2.
\end{align*}
By (\ref{u-uTh}), we can obtain
	\begin{align*}
\|y_{\mathcal{T}_H}-y_{\mathcal{T}_h}\|_{\widetilde{H}^{\frac{\alpha}{2}}(\Omega)}^2 &\leq \mathbb{C}_{yaux} \sum\limits_{K\in\mathcal{T}_{H}\setminus\mathcal{T}_{h}}E_y^2(y_{\mathcal{T}_{H}},K)+
\frac{C}{\gamma^2} \Theta^2(h)\left( \|p_{\mathcal{T}_H}-p_{\mathcal{T}_h}\|^2_{\widetilde{H}^{\frac{\alpha}{2}}(\Omega)} + \| y_{\mathcal{T}_H}-y_{\mathcal{T}_h}\|^2_{\widetilde{H}^{\frac{\alpha}{2}}(\Omega)} \right).
	\end{align*}

Similar arguments can be applied to bound $\|p_{\mathcal{T}_H}-p_{\mathcal{T}_h}\|_{\widetilde{H}^{\frac{\alpha}{2}}(\Omega)}^2$. Using the previous inequality, we can thus arrive at the estimate
	\begin{align*}
\|p_{\mathcal{T}_H}-p_{\mathcal{T}_h}\|_{\widetilde{H}^{\frac{\alpha}{2}}(\Omega)}^2 &\leq \mathbb{C}_{paux} \sum\limits_{K\in\mathcal{T}_{H}\setminus\mathcal{T}_{h}}E_p^2(p_{\mathcal{T}_{H}},K)+
\mathbb{C}_{yaux} \sum\limits_{K\in\mathcal{T}_{H}\setminus\mathcal{T}_{h}}E_y^2(y_{\mathcal{T}_{H}},K)\\
&\quad+
\frac{C}{\gamma^2} \Theta^2(h)\left( \|p_{\mathcal{T}_H}-p_{\mathcal{T}_h}\|^2_{\widetilde{H}^{\frac{\alpha}{2}}(\Omega)} + \| y_{\mathcal{T}_H}-y_{\mathcal{T}_h}\|^2_{\widetilde{H}^{\frac{\alpha}{2}}(\Omega)} \right).
	\end{align*}

%We control the terms $\|\lambda_{\mathcal{T}_{H}}-\lambda_{\mathcal{T}_{h}}\|^2.$ Invoking the estimate (\ref{lamda-lamdath}), that
%	\begin{eqnarray*}
%\|\lambda_{\mathcal{T}_{H}}-\lambda_{\mathcal{T}_{h}}\|^2\leq  C\Theta^2(h)( \|p_{\mathcal{T}_h}-p_{\mathcal{T}_H}\|^2_{\widetilde{H}^{\frac{\alpha}{2}}(\Omega)} + \| y_{\mathcal{T}_h}-y_{\mathcal{T}_H}\|^2_{\widetilde{H}^{\frac{\alpha}{2}}(\Omega)} ).
%	\end{eqnarray*}

%Viewing $u_{\mathcal{T}_H}$ as the continuous solution and $u_{\mathcal{T}_h}$
%as its approximation it follows from (\ref{u-uTh}) that
%	\begin{eqnarray*}
%\|u_{\mathcal{T}_{H}}-u_{\mathcal{T}_{h}}\|^2\leq  C\Theta^2(h)( \|p_{\mathcal{T}_h}-p_{\mathcal{T}_H}\|^2_{\widetilde{H}^{\frac{\alpha}{2}}(\Omega)} + \| y_{\mathcal{T}_h}-y_{\mathcal{T}_H}\|^2_{\widetilde{H}^{\frac{\alpha}{2}}(\Omega)} ).
%	\end{eqnarray*}
Combining the above estimates yields
	\begin{align*}
\|y_{\mathcal{T}_{H}}-y_{\mathcal{T}_{h}}\|^2_{\widetilde{H}^{\frac{\alpha}{2}}(\Omega)}
+\|p_{\mathcal{T}_{H}}-p_{\mathcal{T}_{h}}\|^2_{\widetilde{H}^{\frac{\alpha}{2}}(\Omega)}&\leq \mathbb{C}_{yaux} \sum\limits_{K\in\mathcal{T}_{H}\setminus\mathcal{T}_{h}}E_y^2(y_{\mathcal{T}_{H}},K)
+\mathbb{C}_{paux} \sum\limits_{K\in\mathcal{T}_{H}\setminus\mathcal{T}_{h}}E_p^2(p_{\mathcal{T}_{H}},K)
\\
&\leq \mathcal{E}_{ocp}^2(y_{\mathcal{T}_H},p_{\mathcal{T}_H},\mathcal{T}_{H}\setminus\mathcal{T}_{h}).
\end{align*}
\end{proof}
The optimal control system is a coupled system with nonlinear characteristics. These nonlinear characteristics lead to a lack of support for orthogonality when attempting to prove a contraction. Therefore, we need to prove quasi-orthogonality next.

Let $(y_{\mathcal{T}_{h_{k}}},p_{\mathcal{T}_{h_{k}}})$ be the solution associated to the discrete problem (\ref{lisanstate}) with respect to ${\mathcal{T}_{h_{k}}}$ and $(y_{\mathcal{T}_{h_{k+1}}},p_{\mathcal{T}_{h_{k+1}}})$ be the solution associated to the discrete problem (\ref{lisanstate}) with respect to ${\mathcal{T}_{h_{k+1}}}$. We assume ${\mathcal{T}_{h_{k+1}}}$ is a refinement of ${\mathcal{T}_{h_{k}}}$, and define the following norm
	\begin{eqnarray}
\|e_{\mathcal{T}_{h_{k}}}\|_{\Omega}^2:=\|y-y_{\mathcal{T}_{h_{k}}}\|^2_{\widetilde{H}^{\frac{\alpha}{2}}(\Omega)}
+\|p-p_{\mathcal{T}_{h_{k}}}\|^2_{\widetilde{H}^{\frac{\alpha}{2}}(\Omega)}
	\end{eqnarray}
and
\begin{eqnarray}
\|r_{\mathcal{T}_{h_{k}}}\|_{\Omega}^2:=\|y_{\mathcal{T}_{h_{k}}}-y_{\mathcal{T}_{h_{k+1}}}\|^2_{\widetilde{H}^{\frac{\alpha}{2}}(\Omega)}
+\|p_{\mathcal{T}_{h_{k}}}-p_{\mathcal{T}_{h_{k+1}}}\|^2_{\widetilde{H}^{\frac{\alpha}{2}}(\Omega)},
	\end{eqnarray}
where $(y,p)$ is the optimal solution of the problem (\ref{weak_object})-(\ref{weak_state}). Then the following relation is satisfied
\begin{eqnarray}\label{eth-etH+r}
\|e_{\mathcal{T}_{h_{k+1}}}\|_{\Omega}^2=\|e_{\mathcal{T}_{h_{k}}}\|_{\Omega}^2-\|r_{\mathcal{T}_{h_{k}}}\|_{\Omega}^2+2a(y-y_{\mathcal{T}_{h_{k+1}}},y_{\mathcal{T}_{h_{k}}}-y_{\mathcal{T}_{h_{k+1}}})
+2a(p-p_{\mathcal{T}_{h_{k+1}}},p_{\mathcal{T}_{h_{k}}}-p_{\mathcal{T}_{h_{k+1}}}).
	\end{eqnarray}
%It follows from (\ref{u-uTh}) and (\ref{lamda-lamdath}) that both $u$ and  $\lambda$ are controlled by the state variable $y$ and adjoint variable $p$, so we reduce the notation mentioned above to
%\begin{eqnarray}
%\|e_{k}\|_{\Omega}^2\approx\|y-y_{\mathcal{T}_{h_{k}}}\|^2_{\widetilde{H}^{\frac{\alpha}{2}}(\Omega)}
%+\|p-p_{\mathcal{T}_{h_{k}}}\|^2_{\widetilde{H}^{\frac{\alpha}{2}}(\Omega)},
%	\end{eqnarray}
%\begin{eqnarray}
%\|r_{k}\|_{\Omega}^2\approx\|y_{\mathcal{T}_{h_{k}}}-y_{\mathcal{T}_{h_{k+1}}}\|^2_{\widetilde{H}^{\frac{\alpha}{2}}(\Omega)}
%+\|p_{\mathcal{T}_{h_{k}}}-p_{\mathcal{T}_{h_{k+1}}}\|^2_{\widetilde{H}^{\frac{\alpha}{2}}(\Omega)}
%	\end{eqnarray}
%and
%\begin{eqnarray}\label{eth-etH+r}
%\|e_{k+1}\|_{\Omega}^2&\approx\|e_{k}\|_{\Omega}^2-\|r_{k}\|_{\Omega}^2+2a(y-y_{\mathcal{T}_{h_{k+1}}},y_{\mathcal{T}_{h_{k}}}-y_{\mathcal{T}_{h_{k+1}}})+2a(p-p_{\mathcal{T}_{h_{k+1}}},p_{\mathcal{T}_{h_{k}}}-p_{\mathcal{T}_{h_{k+1}}}).
%	\end{eqnarray}
\begin{theorem}(Quasi-orthogonality) By the above definitions, there holds
\begin{eqnarray*}
\sum\limits_{k=l}^{N}\left\{\|r_{k}\|^2_{\Omega}-2C\Theta^2(h_0)\left(\|y-y_{\mathcal{T}_{h_{k}}}\|_{\widetilde{H}^{\frac{\alpha}{2}}(\Omega)}^2+\|p-p_{\mathcal{T}_{h_{k}}}\|_{\widetilde{H}^{\frac{\alpha}{2}}(\Omega)}^2\right)\right\}
\leq C_{orth}\mathcal{E}_{ocp}^2(y_{\mathcal{T}_{h_{l}}},p_{\mathcal{T}_{h_{l}}},\mathcal{T}_{h_{l}}).
\end{eqnarray*}
Here, $h_0\ll 1$, for all $l, N\in  \mathbb{N}_{0},$ the constant $C_{orth}>0$ is depend
on $\Omega, d, \alpha,$ and the $\gamma$-shape regularity of the initial triangulation ${\mathcal{T}_{h_0}}$.
\end{theorem}

\begin{proof}
 At first we prove the case $k\geq1.$ For convenience, we use $(y_{\mathcal{T}_H},p_{\mathcal{T}_H},u_{\mathcal{T}_H},\lambda_{\mathcal{T}_H}),\ (y_{\mathcal{T}_h},p_{\mathcal{T}_h},u_{\mathcal{T}_h},\lambda_{\mathcal{T}_h})$ to denote $(y_{\mathcal{T}_{h_{k}}},p_{\mathcal{T}_{h_{k}}},u_{\mathcal{T}_{h_{k}}},\lambda_{\mathcal{T}_{h_{k}}})$  and $(y_{\mathcal{T}_{h_{k+1}}},p_{\mathcal{T}_{h_{k+1}}},u_{\mathcal{T}_{h_{k+1}}},\lambda_{\mathcal{T}_{h_{k+1}}})$. So it suffices to proceed in five steps.

$\underline{Step\ 1.}$
 Since $y_{\mathcal{T}_H}-y_{\mathcal{T}_h}\in V_{\mathcal{T}_h}$,  we have that
	\begin{align}\label{ay}
 a(y-y_{\mathcal{T}_h},y_{\mathcal{T}_H}-y_{\mathcal{T}_h})&=(u-u_{\mathcal{T}_h},y_{\mathcal{T}_H}-y_{\mathcal{T}_h})\nonumber\\
 &\leq \frac{1}{2}\|u-u_{\mathcal{T}_h}\|^2+\frac{1}{2}\|y_{\mathcal{T}_H}-y_{\mathcal{T}_h}\|^2.
\end{align}
To control the right hand side of (\ref{ay}), we utilize the auxiliary state $\breve{y}$, defined as $\breve{y}=S_{\mathcal{T}_h}(f+u_{\mathcal{T}_H})$, the control $u$ defined in (\ref{zuiyou}) and combine (\ref{u-uTh}), Lemma \ref{etah} to obtain
	\begin{align*}
&\quad a(y-y_{\mathcal{T}_h},y_{\mathcal{T}_H}-y_{\mathcal{T}_h})\\
 &\leq  \frac{C}{\gamma^2}\Theta^2(h)\left( \|p-p_{\mathcal{T}_h}\|^2_{\widetilde{H}^{\frac{\alpha}{2}}(\Omega)} + \| y-y_{\mathcal{T}_h}\|^2_{\widetilde{H}^{\frac{\alpha}{2}}(\Omega)} \right)+C\|y_{\mathcal{T}_H}-\breve{y}\|^2+C\|\breve{y}-y_{\mathcal{T}_h}\|^2\\
 &\leq \frac{C}{\gamma^2}\Theta^2(h)\left( \|p-p_{\mathcal{T}_h}\|^2_{\widetilde{H}^{\frac{\alpha}{2}}(\Omega)} + \| y-y_{\mathcal{T}_h}\|^2_{\widetilde{H}^{\frac{\alpha}{2}}(\Omega)} \right)+C\Theta^2(H) \|y_{\mathcal{T}_H}-\breve{y}\|_{\widetilde{H}^{\frac{\alpha}{2}}(\Omega)}^2+C\|u_{\mathcal{T}_H}-u_{\mathcal{T}_h}\|^2\\
 &\leq  \frac{C}{\gamma^2}\Theta^2(h)\left( \|p-p_{\mathcal{T}_h}\|^2_{\widetilde{H}^{\frac{\alpha}{2}}(\Omega)} + \| y-y_{\mathcal{T}_h}\|^2_{\widetilde{H}^{\frac{\alpha}{2}}(\Omega)} \right)+C\Theta^2(H) \|y_{\mathcal{T}_H}-\breve{y}\|_{\widetilde{H}^{\frac{\alpha}{2}}(\Omega)}^2+C\|u_{\mathcal{T}_H}-u\|^2+C\|u-u_{\mathcal{T}_h}\|^2 \\
 &\leq  \frac{C}{\gamma^2}\Theta^2(h)\left( \|p-p_{\mathcal{T}_h}\|^2_{\widetilde{H}^{\frac{\alpha}{2}}(\Omega)} + \| y-y_{\mathcal{T}_h}\|^2_{\widetilde{H}^{\frac{\alpha}{2}}(\Omega)} \right)+C\Theta^2(H) \|y_{\mathcal{T}_H}-\breve{y}\|_{\widetilde{H}^{\frac{\alpha}{2}}(\Omega)}^2+ \frac{C}{\gamma^2}\Theta^2(H)\left( \|p-p_{\mathcal{T}_H}\|^2_{\widetilde{H}^{\frac{\alpha}{2}}(\Omega)} + \| y-y_{\mathcal{T}_H}\|^2_{\widetilde{H}^{\frac{\alpha}{2}}(\Omega)} \right)  .
\end{align*}

$\underline{Step\ 2.}$
The goal of this step is to bound $a(p-p_{\mathcal{T}_h},p_{\mathcal{T}_H}-p_{\mathcal{T}_h})$. Similarly, Since $p_{\mathcal{T}_H}-p_{\mathcal{T}_h}\in V_{\mathcal{T}_h}$,  we have that
\begin{align*}
 a(p-p_{\mathcal{T}_h},p_{\mathcal{T}_H}-p_{\mathcal{T}_h})&=(y-y_{\mathcal{T}_h},p_{\mathcal{T}_H}-p_{\mathcal{T}_h})\nonumber\\
 &\leq \frac{1}{2}\|y-y_{\mathcal{T}_h}\|^2+\frac{1}{2}\|p_{\mathcal{T}_H}-p_{\mathcal{T}_h}\|^2.
 \end{align*}

In an analogous way, we utilize the auxiliary adjoint state $\breve{p}$, defined as $\breve{p}=\mathcal{S}^*_{\mathcal{T}_h}(\mathcal{S}_{\mathcal{T}_H}(f+u_{\mathcal{T}_H})-y_d)$, the control $u$ defined in (\ref{zuiyou}) and combine (\ref{opy-yTh}), Lemma \ref{etah}, (\ref{yythor}) to obtain
	\begin{align*}
 &\quad a(p-p_{\mathcal{T}_h},p_{\mathcal{T}_H}-p_{\mathcal{T}_h})\\
 &\leq \frac{C}{\gamma^2}\Theta^2(h)\left( \|p-p_{\mathcal{T}_h}\|^2_{\widetilde{H}^{\frac{\alpha}{2}}(\Omega)} + \| y-y_{\mathcal{T}_h}\|^2_{\widetilde{H}^{\frac{\alpha}{2}}(\Omega)} \right )+C\|p_{\mathcal{T}_H}-\breve{p}\|^2+C\|\breve{p}-p_{\mathcal{T}_h}\|^2\\
 &\leq \frac{C}{\gamma^2}\Theta^2(h)\left( \|p-p_{\mathcal{T}_h}\|^2_{\widetilde{H}^{\frac{\alpha}{2}}(\Omega)} + \| y-y_{\mathcal{T}_h}\|^2_{\widetilde{H}^{\frac{\alpha}{2}}(\Omega)} \right)+C\Theta^2(H) \|p_{\mathcal{T}_H}-\breve{p}\|_{\widetilde{H}^{\frac{\alpha}{2}}(\Omega)}^2+C\|y_{\mathcal{T}_H}-y_{\mathcal{T}_h}\|^2\\
 &\leq \frac{C}{\gamma^2}\Theta^2(h)\left( \|p-p_{\mathcal{T}_h}\|^2_{\widetilde{H}^{\frac{\alpha}{2}}(\Omega)} + \| y-y_{\mathcal{T}_h}\|^2_{\widetilde{H}^{\frac{\alpha}{2}}(\Omega)} \right)+C\Theta^2(H) \|p_{\mathcal{T}_H}-\breve{p}\|_{\widetilde{H}^{\frac{\alpha}{2}}(\Omega)}^2\\
 &\quad+\frac{C}{\gamma^2}\Theta^2(H)\left( \|p-p_{\mathcal{T}_H}\|^2_{\widetilde{H}^{\frac{\alpha}{2}}(\Omega)} + \| y-y_{\mathcal{T}_H}\|^2_{\widetilde{H}^{\frac{\alpha}{2}}(\Omega)} \right)+ C\Theta^2(H) \|y_{\mathcal{T}_H}-\breve{y}\|_{\widetilde{H}^{\frac{\alpha}{2}}(\Omega)}^2.
\end{align*}

$\underline{Step\ 3.}$
By the relation (\ref{eth-etH+r}) and Lemma \ref{ythk-y}, combining above estimates leads to
\begin{align*}
&\quad\|e_{k+1}\|_{\Omega}^2-\|e_{k}\|_{\Omega}^2+\|r_{k}\|_{\Omega}^2\\
&=2a(y-y_{\mathcal{T}_{h}},y_{\mathcal{T}_{H}}-y_{\mathcal{T}_{h}})
+2a(p-p_{\mathcal{T}_{h}},p_{\mathcal{T}_{H}}-p_{\mathcal{T}_{h}})\nonumber\\
&\leq  \frac{C}{\gamma^2}\Theta^2(h)\left( \|p-p_{\mathcal{T}_h}\|^2_{\widetilde{H}^{\frac{\alpha}{2}}(\Omega)} + \| y-y_{\mathcal{T}_h}\|^2_{\widetilde{H}^{\frac{\alpha}{2}}(\Omega)} \right)+ \frac{C}{\gamma^2}\Theta^2(H)\left( \|p-p_{\mathcal{T}_H}\|^2_{\widetilde{H}^{\frac{\alpha}{2}}(\Omega)} + \| y-y_{\mathcal{T}_H}\|^2_{\widetilde{H}^{\frac{\alpha}{2}}(\Omega)} \right)\\
&\quad+C\Theta^2(H) \left(\|y_{\mathcal{T}_H}-\breve{y}\|_{\widetilde{H}^{\frac{\alpha}{2}}(\Omega)}^2+\|p_{\mathcal{T}_H}-\breve{p}\|_{\widetilde{H}^{\frac{\alpha}{2}}(\Omega)}^2\right)\\
&\leq \frac{C}{\gamma^2}\Theta^2(h)\left( \|p-p_{\mathcal{T}_h}\|^2_{\widetilde{H}^{\frac{\alpha}{2}}(\Omega)} + \| y-y_{\mathcal{T}_h}\|^2_{\widetilde{H}^{\frac{\alpha}{2}}(\Omega)} \right)+ \frac{C}{\gamma^2}\Theta^2(H)\left( \|p-p_{\mathcal{T}_H}\|^2_{\widetilde{H}^{\frac{\alpha}{2}}(\Omega)} + \| y-y_{\mathcal{T}_H}\|^2_{\widetilde{H}^{\frac{\alpha}{2}}(\Omega)} \right)\\
&\quad+C\Theta^2(H) \left(\sum\limits_{K\in\mathcal{T}_{H}\setminus\mathcal{T}_{h}}E_y^2(y_{\mathcal{T}_{H}},K)
+\sum\limits_{K\in\mathcal{T}_{H}\setminus\mathcal{T}_{h}}E_p^2(p_{\mathcal{T}_{H}},K)\right). \end{align*}

Further, a simple application of the triangle inequality reveal that
\begin{align*}
&\quad\|r_{k}\|_{\Omega}^2\\
%&\|y_{\mathcal{t}_{h}}-y_{\mathcal{t}_{h}}\|^2_{\widetilde{h}^{\frac{\alpha}{2}}(\omega)}
%+\|p_{\mathcal{t}_{h}}-p_{\mathcal{t}_{h}}\|^2_{\widetilde{h}^{\frac{\alpha}{2}}(\omega)}
%+\|u_{\mathcal{t}_{h}}-u_{\mathcal{t}_{h}}\|^2+\|\lambda_{\mathcal{t}_{h}}-\lambda_{\mathcal{t}_{h}}\|^2\\
&\leq(\|y-y_{\mathcal{T}_{H}}\|^2_{\widetilde{H}^{\frac{\alpha}{2}}(\Omega)}
+\|p-p_{\mathcal{T}_{H}}\|^2_{\widetilde{H}^{\frac{\alpha}{2}}(\Omega)}
)-(\|y-y_{\mathcal{T}_{h}}\|^2_{\widetilde{H}^{\frac{\alpha}{2}}(\Omega)}
+\|p-p_{\mathcal{T}_{h}}\|^2_{\widetilde{H}^{\frac{\alpha}{2}}(\Omega)})\\
&\quad+ \frac{C}{\gamma^2}\Theta^2(h)\left( \|p-p_{\mathcal{T}_h}\|^2_{\widetilde{H}^{\frac{\alpha}{2}}(\Omega)} + \| y-y_{\mathcal{T}_h}\|^2_{\widetilde{H}^{\frac{\alpha}{2}}(\Omega)}\right)+ C\Theta^2(H) \sum\limits_{K\in\mathcal{T}_{H}\setminus\mathcal{T}_{h}}\mathcal{E}_{K}^2(y_{\mathcal{T}_H},p_{\mathcal{T}_H},K)\\
&\quad+\frac{C}{\gamma^2}\Theta^2(H)\left( \|p-p_{\mathcal{T}_H}\|^2_{\widetilde{H}^{\frac{\alpha}{2}}(\Omega)} + \| y-y_{\mathcal{T}_H}\|^2_{\widetilde{H}^{\frac{\alpha}{2}}(\Omega)}\right)\\
&\leq \left(1+\frac{C}{\gamma^2}\Theta^2(h_0) \right)\left( \|y-y_{\mathcal{T}_H}\|_{\widetilde{H}^{\frac{\alpha}{2}}(\Omega)}^2+\|p-p_{\mathcal{T}_H} \|^2_{\widetilde{H}^{\frac{\alpha}{2}}(\Omega)}\right)+ C\Theta^2(H) \sum\limits_{K\in\mathcal{T}_{H}\setminus\mathcal{T}_{h}}\mathcal{E}_{K}^2(y_{\mathcal{T}_H},p_{\mathcal{T}_H},K)\\
&\quad-\left(1-\frac{C}{\gamma^2}\Theta^2(h_0)\right ) \left( \|y-y_{\mathcal{T}_h}\|_{\widetilde{H}^{\frac{\alpha}{2}}(\Omega)}^2+\|p-p_{\mathcal{T}_h} \|^2_{\widetilde{H}^{\frac{\alpha}{2}}(\Omega)}\right )\\
&\leq \left(1+\frac{C}{\gamma^2}\Theta^2(h_0) \right)\left( \|y-y_{\mathcal{T}_H}\|_{\widetilde{H}^{\frac{\alpha}{2}}(\Omega)}^2+\|p-p_{\mathcal{T}_H} \|^2_{\widetilde{H}^{\frac{\alpha}{2}}(\Omega)}\right)+ C\Theta^2(H) \sum\limits_{K\in\mathcal{T}_{H}\setminus\mathcal{T}_{h}}\mathcal{E}_{K}^2(y_{\mathcal{T}_H},p_{\mathcal{T}_H},K)\\
&\quad-\left(1-\frac{C}{\gamma^2}\Theta^2(h_0) \right) \left( \|y-y_{\mathcal{T}_h}\|_{\widetilde{H}^{\frac{\alpha}{2}}(\Omega)}^2+\|p-p_{\mathcal{T}_h} \|^2_{\widetilde{H}^{\frac{\alpha}{2}}(\Omega)}\right ),
\end{align*}
provided $h_0\ll1.$
We apply the Lemma \ref{thtH} to conclude that
\begin{align*}
%&\|y_{\mathcal{T}_{H}}-y_{\mathcal{T}_{h}}\|^2_{\widetilde{H}^{\frac{\alpha}{2}}(\Omega)}
%+\|p_{\mathcal{T}_{H}}-p_{\mathcal{T}_{h}}\|^2_{\widetilde{H}^{\frac{\alpha}{2}}(\Omega)}
%+\|u_{\mathcal{T}_{H}}-u_{\mathcal{T}_{h}}\|^2+\|\lambda_{\mathcal{T}_{H}}-\lambda_{\mathcal{T}_{h}}\|^2\\
&\quad\|r_{k}\|_{\Omega}^2
-2\frac{C}{\gamma^2}\Theta^2(h_0)\left(\|y-y_{\mathcal{T}_H}\|_{\widetilde{H}^{\frac{\alpha}{2}}(\Omega)}^2+\|p-p_{\mathcal{T}_H}\|_{\widetilde{H}^{\frac{\alpha}{2}}(\Omega)}^2\right)\\
&\leq\left(1-\frac{C}{\gamma^2}\Theta^2(h_0) \right)\left(\|y-y_{\mathcal{T}_H}\|_{\widetilde{H}^{\frac{\alpha}{2}}(\Omega)}^2+\|p-p_{\mathcal{T}_H}\|_{\widetilde{H}^{\frac{\alpha}{2}}(\Omega)}^2\right)-\left(1-\frac{C}{\gamma^2}\Theta^2(h_0) \right)\left(\|y-y_{\mathcal{T}_h}\|_{\widetilde{H}^{\frac{\alpha}{2}}(\Omega)}^2+\|p-p_{\mathcal{T}_h}\|_{\widetilde{H}^{\frac{\alpha}{2}}(\Omega)}^2\right)\\
&\quad+\frac{ C\Theta^2(h_0)}{1-2^{-\frac{\rho\xi}{d}}}
\left(\mathcal{E}_{ocp}^2(y_{\mathcal{T}_H},p_{\mathcal{T}_H},\mathcal{T}_{H})-\mathcal{E}_{ocp}^2(y_{\mathcal{T}_H},p_{\mathcal{T}_H},\mathcal{T}_{h})\right),
\end{align*}
where $h_0\ll1.$

To conclude the previous estimate combined with Theorem \ref{Reliability} and Remark \ref{r2} leads to the general quasi-orthogonality as follows
\begin{align*}
&\quad\sum\limits_{k=l}^{N}\left\{\|r_{k}\|^2_{\Omega}-2\frac{C}{\gamma^2}\Theta^2(h_0)\left(\|y-y_{\mathcal{T}_{h_{k}}}\|_{\widetilde{H}^{\frac{\alpha}{2}}(\Omega)}^2+\|p-p_{\mathcal{T}_{h_{k}}}\|_{\widetilde{H}^{\frac{\alpha}{2}}(\Omega)}^2\right)\right\}
\\&\leq\sum\limits_{k=l}^N\left\{\left(1-\frac{C}{\gamma^2}\Theta^2(h_0) \right)\left(\|y-y_{\mathcal{T}_{h_{k}}}\|_{\widetilde{H}^{\frac{\alpha}{2}}(\Omega)}^2-\|y-y_{\mathcal{T}_{h_{k+1}}}\|_{\widetilde{H}^{\frac{\alpha}{2}}(\Omega)}^2\right) -\left(1-\frac{C}{\gamma^2}\Theta^2(h_0) \right) \left(\|p-p_{\mathcal{T}_{h_{k}}}\|_{\widetilde{H}^{\frac{\alpha}{2}}(\Omega)}^2-\|p-p_{\mathcal{T}_{h_{k+1}}}\|_{\widetilde{H}^{\frac{\alpha}{2}}(\Omega)}^2\right)\right.\\
&\quad\left.+\frac{ C\Theta^2(h_0)}{1-2^{-\frac{\rho\xi}{d}}} \left(\mathcal{E}_{ocp}^2(y_{\mathcal{T}_{h_{k}}},p_{\mathcal{T}_{h_{k}}},\mathcal{T}_{h_{k}})-\mathcal{E}_{ocp}^2(y_{\mathcal{T}_{h_{k}}},p_{\mathcal{T}_{h_{k}}},\mathcal{T}_{h_{k+1}})\right)\right\}\\
&\leq\left(1-\frac{C}{\gamma^2}\Theta^2(h_0) \right)\left(\|y-y_{\mathcal{T}_{h_{l}}}\|_{\widetilde{H}^{\frac{\alpha}{2}}(\Omega)}^2+\|p-p_{\mathcal{T}_{h_{l}}}\|_{\widetilde{H}^{\frac{\alpha}{2}}(\Omega)}^2\right)+\frac{ C\Theta^2(h_0)}{1-2^{-\frac{\rho\xi}{d}}} \mathcal{E}_{ocp}^2(y_{\mathcal{T}_{h_{l}}},p_{\mathcal{T}_{h_{l}}},\mathcal{T}_{h_{l}})\\
&\leq C_{orth}\mathcal{E}_{ocp}^2(y_{\mathcal{T}_{h_{l}}},p_{\mathcal{T}_{h_{l}}},\mathcal{T}_{h_{l}}).
\end{align*}
This concludes the proof.
\end{proof}

For each $t>0$, if there exists constants $C_{ar},\ Q_{ar}>0$ such that
\begin{eqnarray}\label{rate}
C_{ar}\mathbb{A}_{t}(v)\leq \sup\limits_{\ell\in\mathbb{N}_{0}} (\#{{\mathcal{T}_{h_l}}})^t \mathcal{E}_{ocp}^2(y_{\mathcal{T}_{h_{l}}},p_{\mathcal{T}_{h_{l}}},\mathcal{T}_{h_{l}})<  Q_{ar}\mathbb{A}_{t}(v),\label{optimality_2}
\end{eqnarray}
we say that   the adaptive Algorithm 2 is rate optimal with respect to the
error estimator, where
\begin{eqnarray*}
\mathbb{A}_{t}(v)=\sup\limits_{N\in\mathbb{N}_{0}}(N+1)^t\mathop{\min}\limits_{\mathcal{T}_h\in\mathbb{T} \atop \#{\mathcal{T}_h}-\#{{\mathcal{T}_{h_0}}}\leq N } \mathcal{E}_{ocp}^2(y_{\mathcal{T}_{h_{l}}},p_{\mathcal{T}_{h_{l}}},\mathcal{T}_{h_{l}}).
\end{eqnarray*}
According to \cite{Car}, through the proof of the above Theorem  we indeed also verify that Algorithm 1 can reach the optimal convergence order in the sense of (\ref{rate}).

\section{ Numerical\ results}
In this section, three numerical experiments are presented, and the exact solutions in the circle domain of the first example are given. The solutions in the square domain of the second and third examples are not known. We proceed by establishing fixed values for the optimal state and adjoint state variable, by employing the projection formulas
 $u=\Pi_{[a,b]}\left(-\frac{1}{\gamma}(p+\beta\lambda)\right)$
and
 $\lambda=\Pi_{[-1,1]}\left(-\frac{1}{\beta}p\right)$.

%Before that, we define the following effective index:
%$$\mbox{effectivity}:=\frac{\mathcal{R}(u_{\mathcal{T}_h},z_{\mathcal{T}_h},\mathcal{T}_h)}{\|(u-u_{\mathcal{T}_h}, z-z_{\mathcal{T}_h})\|_{\widetilde{H}^{s}(\Omega)}}.$$
\begin{example}\label{Exm:1}
We set $\Omega=B(0,1)$, $a=-0.5$, $b=0.5$, $c=3$ and the exact solutions are as follows:
\begin{align*}
y&=\frac{2^{-\alpha}(1-|x|^2)^{\frac{\alpha}{2}}}{\Gamma(1+\frac{\alpha}{2})^2},\ p=c y,\\
u&=\Pi_{[a,b]}\left(-\frac{1}{\gamma}(c y+\beta\lambda)\right),\ \lambda=\Pi_{[-1,1]}\left(-\frac{1}{\beta}c y\right).
\end{align*}
\end{example}
%In general, we set the parameters respectively. The specific values may vary depending on the values of $\gamma$ and $\beta$. Due to truncation of the solution, different value strategies may need to be applied.

Figure \ref{adpmesh} shows the initial mesh and the final refinement mesh with $\alpha=0.5,\ \theta=0.7$. Since the exact solutions of the state variable and adjoint variable exhibit smoothdness within the unit circle, with singularities $\partial\Omega$ on the boundary, so the mesh is refined mainly in the region close to the boundary.

 \begin{figure}[!htbp]
\centering
%\flushleft
\label{2a}
\includegraphics[width=7.2cm]{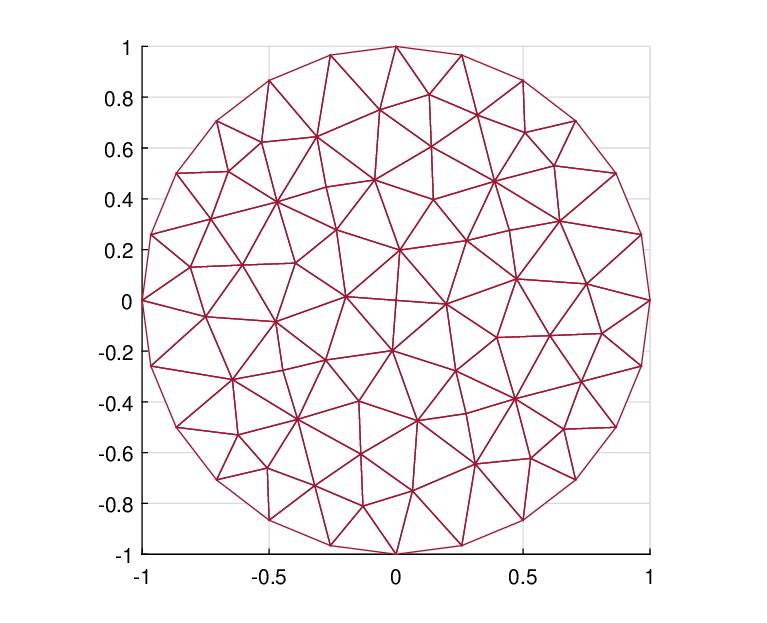}
\hspace{-0.01mm}
\label{2b}
\includegraphics[width=6.7cm]{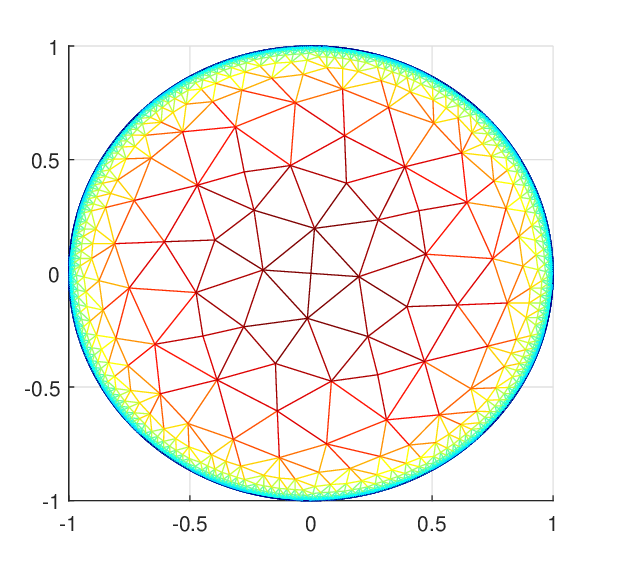}
%\hspace{-1mm}
%\label{2c}
%\includegraphics[width=4cm]{control_0.25.eps}
%\hspace{-5mm}
\caption{The initial mesh (left)  and the final refinement mesh (right) with $\alpha=0.5,\theta=0.7$ on the circle.}
\label{adpmesh}
\end{figure}

In Figure \ref{025075estimator}, the computational rates of convergence for the computable error estimators and indicators $\mathcal{E}_{ocp}, E_y$ and $E_p$ for $\alpha=0.5$ and $\alpha=1.5$ are presented, respectively. It can be observed that, in both cases, each contribution decays with the optimal rate $N^{-\frac{1}{2}}$.
 \begin{figure}[!htbp]
\centering
%\flushleft
\label{2a}
\includegraphics[width=6.5cm]{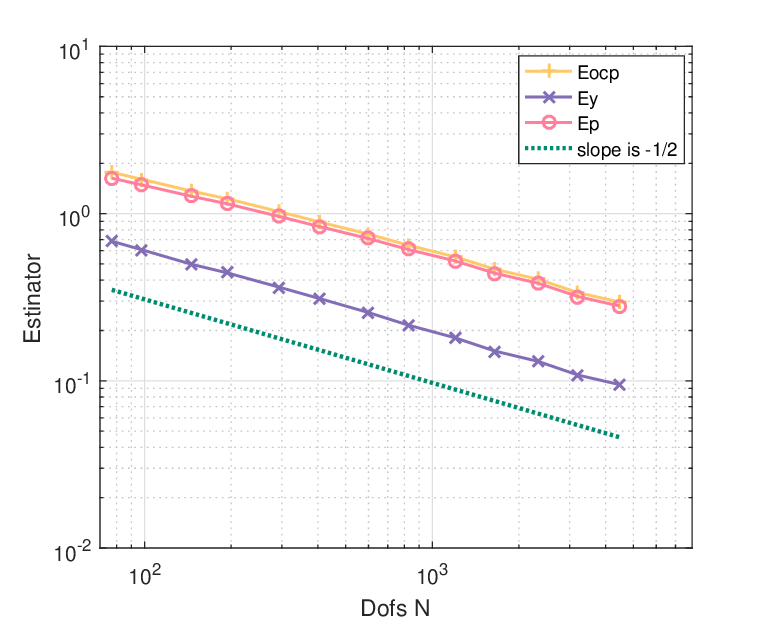}
\hspace{-0.01mm}
\label{2b}
\includegraphics[width=6.5cm]{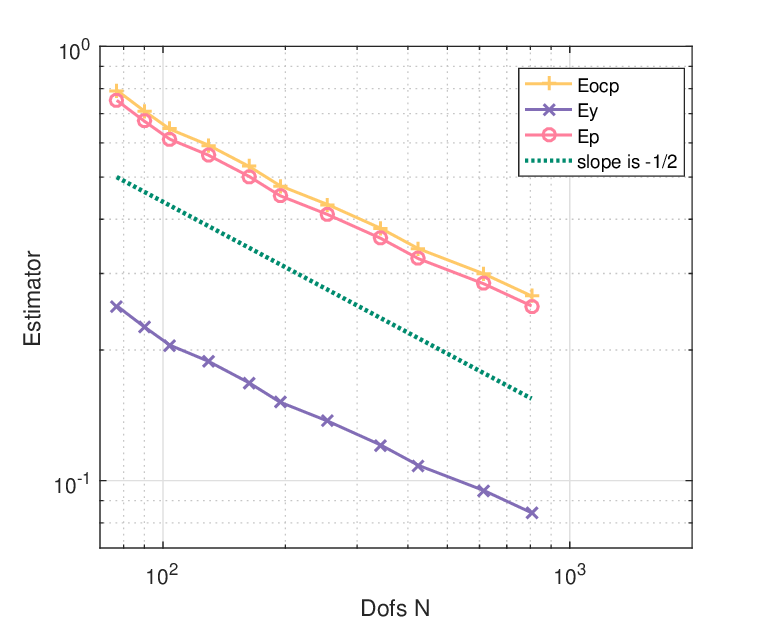}
%\hspace{-1mm}
%\label{2c}
%\includegraphics[width=4cm]{control_0.25.eps}
%\hspace{-5mm}
\caption{The convergent behaviors of the indicators and estimators for $\alpha=0.5,\ \theta=0.7$ (left)  and $\alpha=1.5,\ \theta=0.5 $ (right).}
\label{025075estimator}
\end{figure}

We set  $\alpha=0.5$ and the parameter $\theta= 0.5,\ 0.7,\ 1$ that governs the module $\mathbf{MARK}$. The left plot of Figure \ref{fig025} illustrates the convergence orders of the error estimators $\mathcal{E}_{ocp}$ and error indicators $E_y,\ E_p$ under different values of $\theta$. On the right plot of Figure \ref{fig025}, the convergence orders of the errors for the state and adjoint variables in $\|\cdot\|_{\widetilde{H}^{\frac{\alpha}{2}}(\Omega)}$ and the effectivity indices which are given by $\mathcal{E}_{ocp}/\|e\|_{\Omega}$ are presented for different $\theta$ values. From the Figure \ref{fig025}, it is observed that when $\theta=1$, indicating uniform refinement, the displayed convergence rates do not reach optimality. However, for $\theta<1,$ the convergence rates of the error eatimators $\mathcal{E}_{ocp}$ and errors clearly converge to $N^{-\frac{1}{2}}$. Thus, our theoretical analysis is effectively verified.

 \begin{figure}[!htbp]
\centering
%\flushleft
\label{2a}
\includegraphics[width=6.5cm]{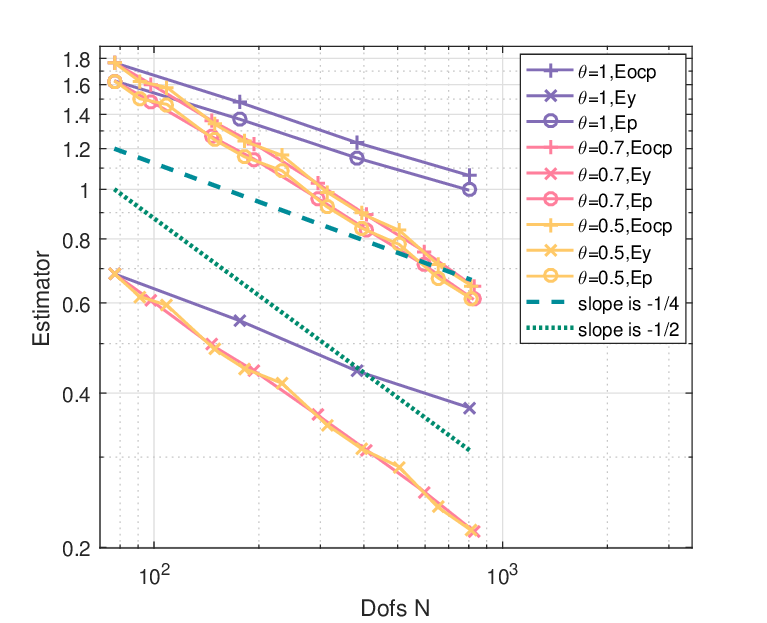}
\hspace{-0.01mm}
\label{2b}
\includegraphics[width=6.5cm]{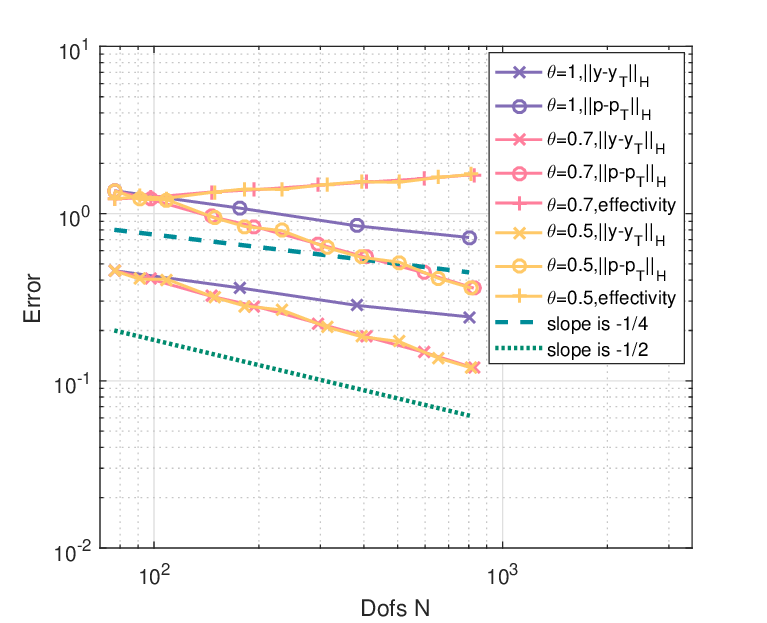}
%\hspace{-1mm}
%\label{2c}
%\includegraphics[width=4cm]{control_0.25.eps}
%\hspace{-5mm}
\caption{The convergent behaviors of the errors, indicators and estimators for fixed $\alpha=0.5$ and $\theta=0.5,0.7,1$ respectively on the circle.}
\label{fig025}
\end{figure}

Next, we consider the effect of changing the regularization parameter $\gamma$ on the system with $\alpha=1.5,\ \theta=0.5$ and $\beta=1$. Specifically, we examine the cases where $\gamma$ takes the values of $$\gamma\in \{10^0,\ 10^{-1},\ 10^{-2},\ 10^{-3},\ 10^{-4}\}.$$
 It can be seen from the Figure \ref{figbeta} that the error estimators $\mathcal{E}_{ocp}$, the error indicators $E_y,\ E_p,$ errors of state and adjoint variable in $\|\cdot\|_{\widetilde{H}^{\frac{\alpha}{2}}(\Omega)}$ can reach the optimal convergence order for all the values of the parameter $\gamma$ considered.
 \begin{figure}[!htbp]
\centering
%\flushleft
\label{2a}
\includegraphics[width=6.5cm]{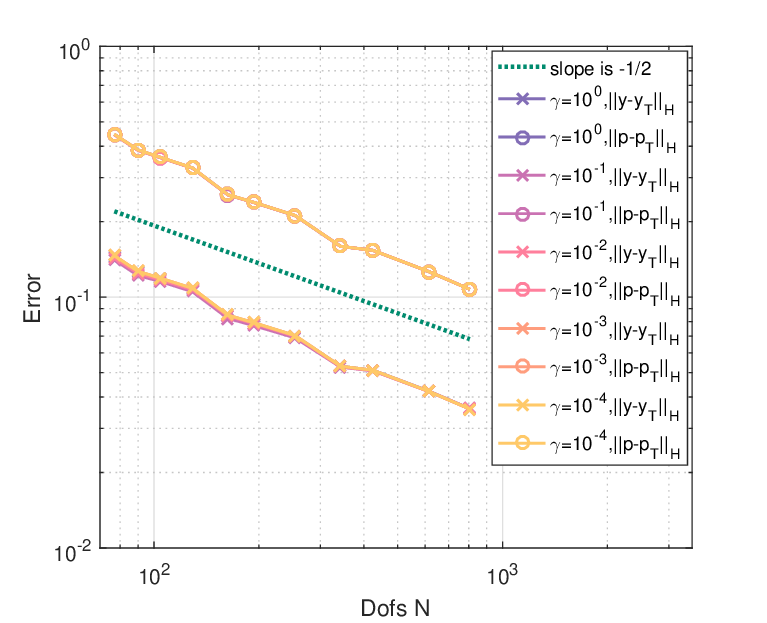}
\hspace{-0.01mm}
\label{2b}
\includegraphics[width=6.5cm]{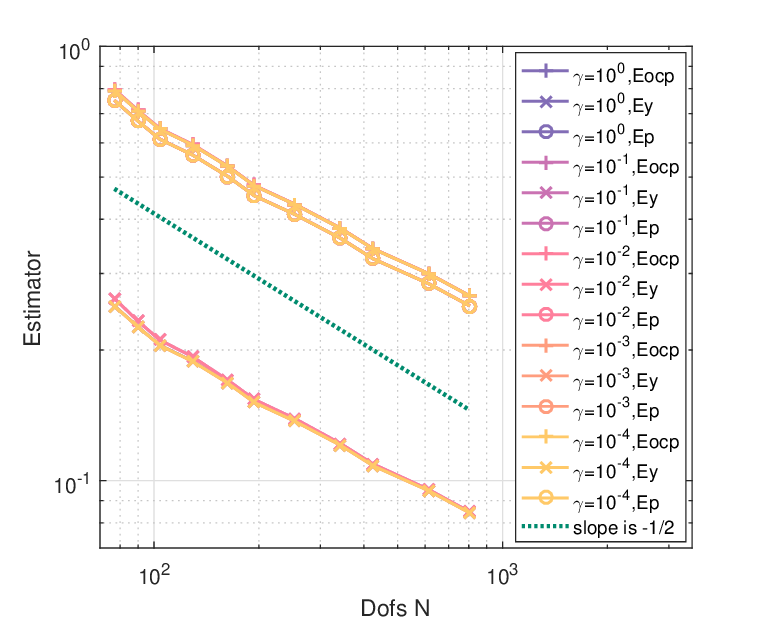}
%\hspace{-1mm}
%\label{2c}
%\includegraphics[width=4cm]{control_0.25.eps}
%\hspace{-5mm}
\caption{The convergent behaviors of the errors, indicators and estimators for all the values of the parameter $\gamma$ respectively on the circle.}
\label{figbeta}
\end{figure}

In Figures \ref{figbetanum}-\ref{statenum}, we show the profiles of the numerical solutions for the control and state when $\alpha=1.5,\ \theta=0.5$, respectively. It can be seen that as $\gamma$ decreases, the $L^1$ term dominates and the numerical solutions of the control become sparsier.
 \begin{figure}[!htbp]
\centering
%\flushleft
\begin{subfigure}[]

\includegraphics[width=4.5cm]{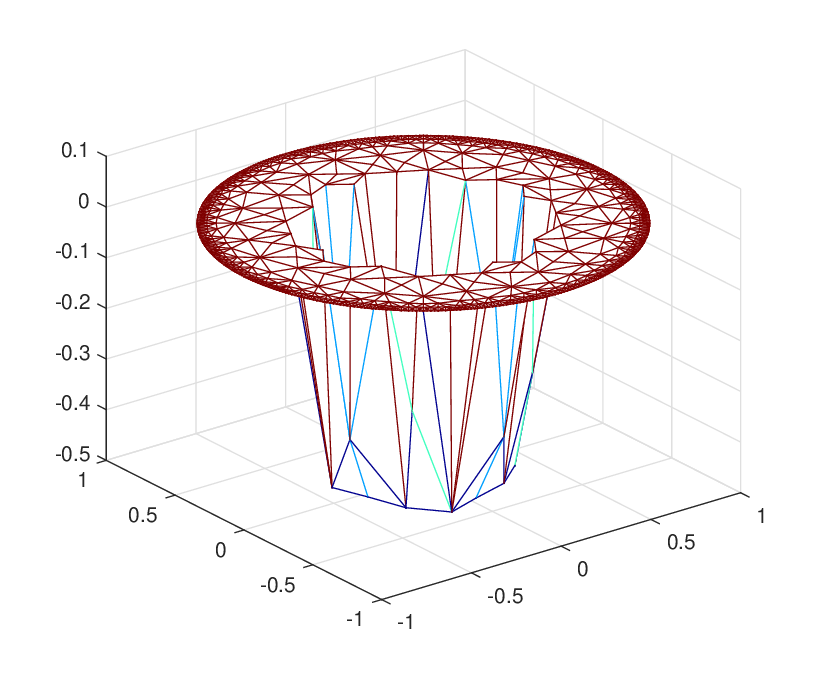}
\hspace{-0.01mm}
\end{subfigure}
\begin{subfigure}[]

\includegraphics[width=4.5cm]{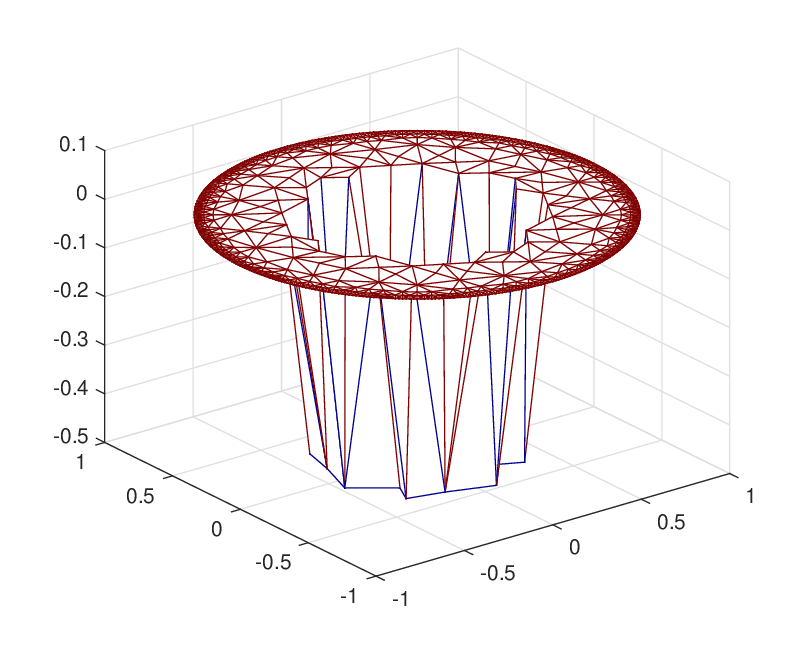}
\hspace{-0.01mm}
\end{subfigure}
\begin{subfigure}[]

\includegraphics[width=4.1cm]{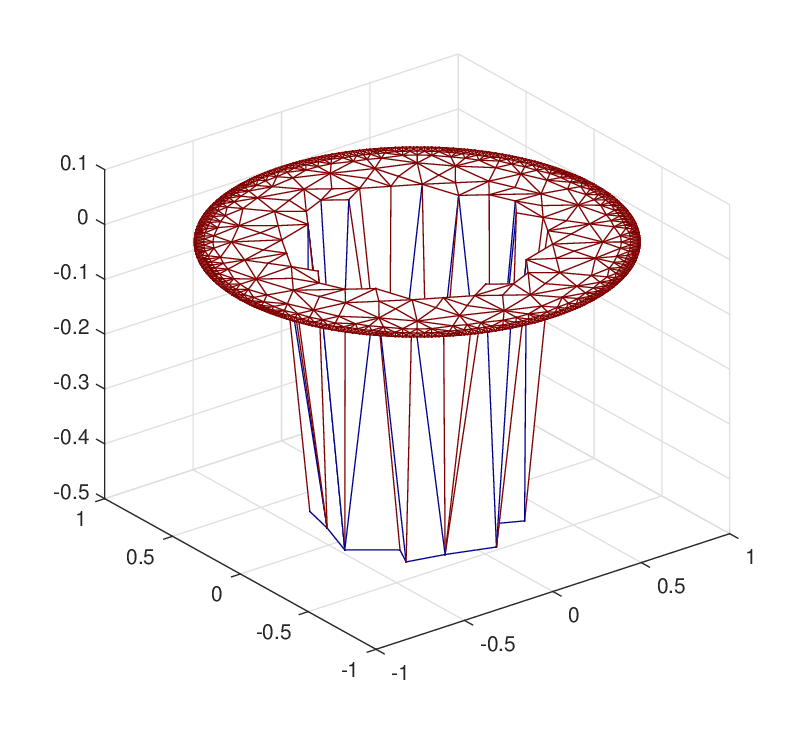}
\hspace{-0.01mm}
\end{subfigure}

\caption{The profiles of the numerically computed control with $\gamma=10^{-1}$ (a), $\gamma=10^{-2}$  (b), and $\gamma=10^{-3}$ (c).}
\label{figbetanum}
\end{figure}

 \begin{figure}[!htbp]
\centering
%\flushleft
\begin{subfigure}[]

\includegraphics[width=4.5cm]{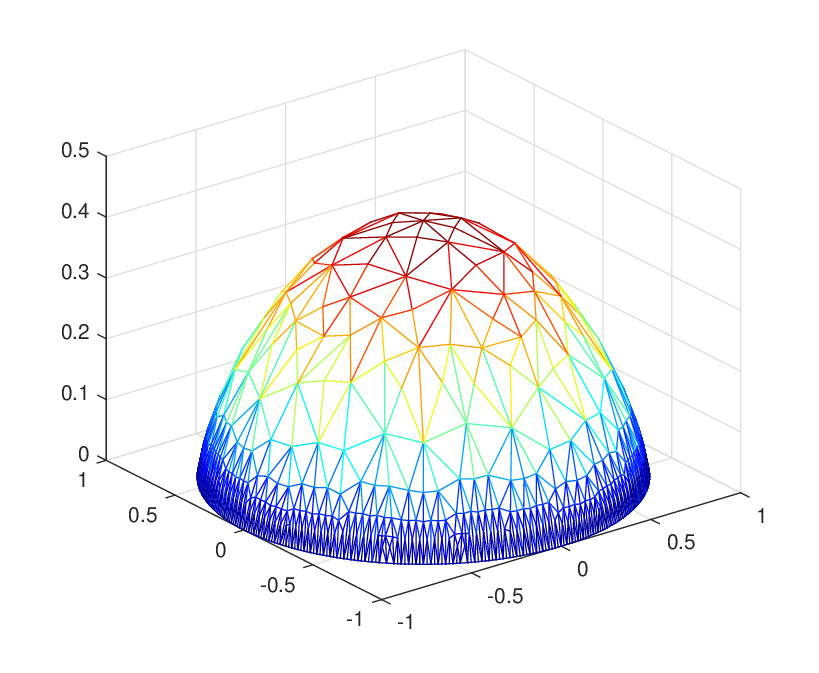}
\hspace{-0.01mm}
\end{subfigure}
\begin{subfigure}[]

\includegraphics[width=4.5cm]{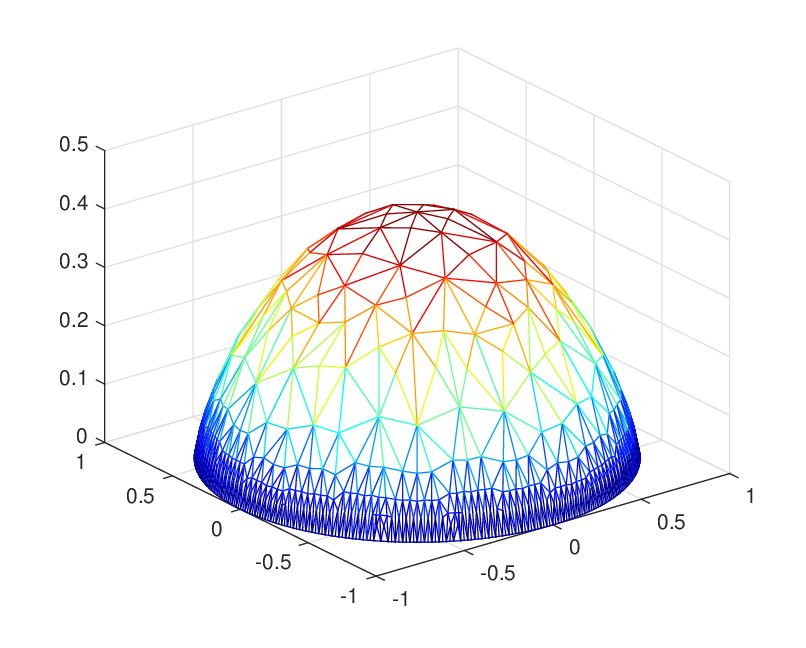}
\hspace{-0.01mm}
\end{subfigure}
\begin{subfigure}[]

\includegraphics[width=4.5cm]{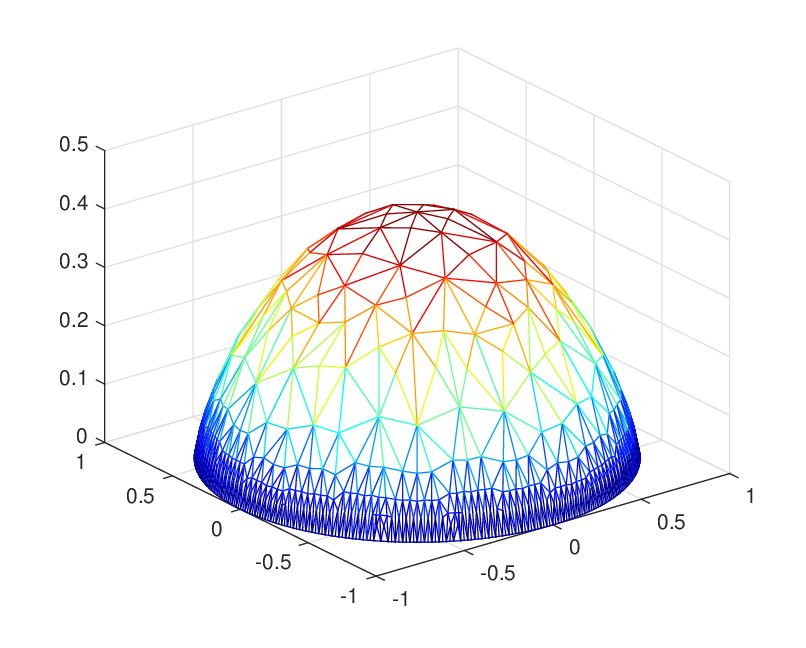}
\hspace{-0.01mm}
\end{subfigure}

\caption{The profiles of the numerically computed state with $\gamma=10^{-1}$ (a), $\gamma=10^{-2}$  (b), and $\gamma=10^{-3}$ (c).}
\label{statenum}
\end{figure}
\begin{example}\label{Exm:2}
In the second example we consider an optimal control problem with $f=-6,\ y_d=1$. We set $\Omega=(-1,1)^{2}$, $\gamma\in\{10^0,\ 10^{-1},\ 10^{-2}\}$, $\beta=1$, $a=-0.3$, $b=0.3$, respectively.
\end{example}

In Figure \ref{2mesh}  we show the initial mesh and the final refinement mesh with $\alpha=0.5,\ \theta=0.7$. The primary refinement behavior is observed to occur exclusively along the boundaries of the entire square domain.  This observation suggests that the estimators effectively capture the singularities of the exact solution along the entire boundary, thus guiding the mesh refinement process.

\begin{figure}[!htbp]
\centering
%\flushleft
\label{2a}
\includegraphics[width=6.5cm]{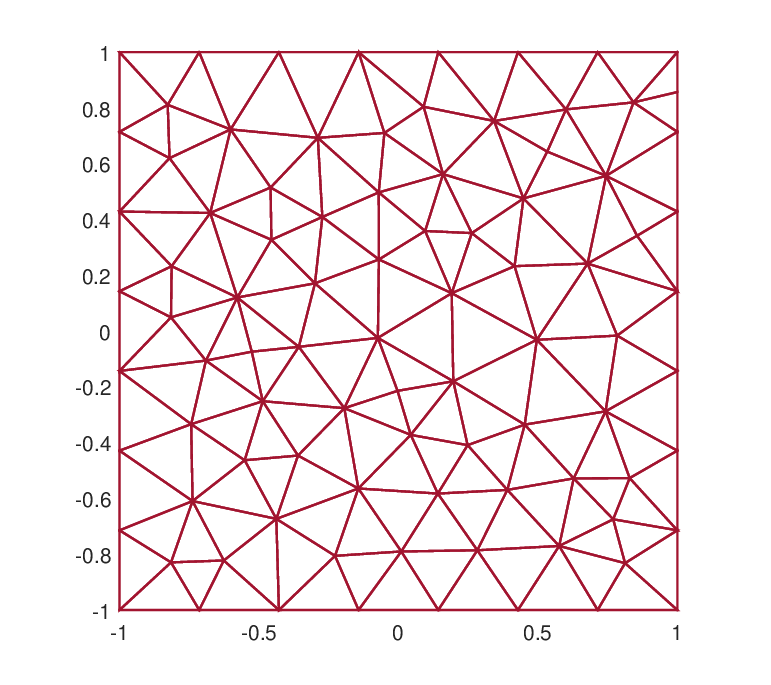}
\hspace{-0.01mm}
\label{2b}
\includegraphics[width=6.5cm,height=5.9cm]{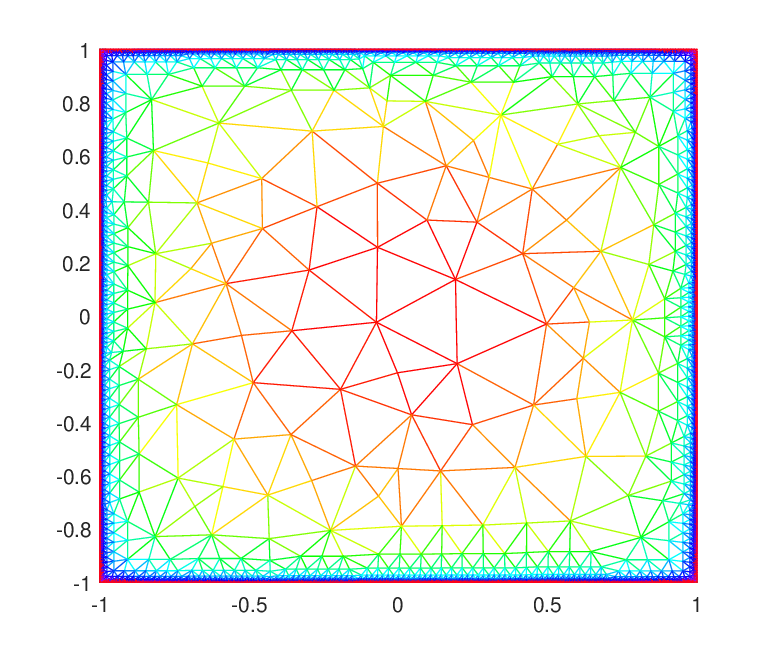}
%\hspace{-1mm}
%\label{2c}
%\includegraphics[width=4cm]{control_0.25.eps}
%\hspace{-5mm}
\caption{The initial mesh (left)  and the final refinement mesh (right) with $\alpha=0.5,\theta=0.7$ on the square.}
\label{2mesh}
\end{figure}

For $\alpha = 0.5$ and $\alpha = 1.5$, the Figure \ref{2e} shows that the AFEM proposed in Section 6 delivers
optimal experimental rates of convergence for the error estimators $\mathcal{E}_{ocp}$, the error indicators $E_y,\ E_p.$ The results obtained empirically are consistent with those of the previous example. The convergence rates of the estimators and indicators are $N^{-\frac{1}{4}}$ for uniform refinement, while adaptive refinement leads to optimal convergence rates of $N^{-\frac{1}{2}}$.
\begin{figure}[!htbp]
\centering
%\flushleft
\label{2a}
\includegraphics[width=6.4cm]{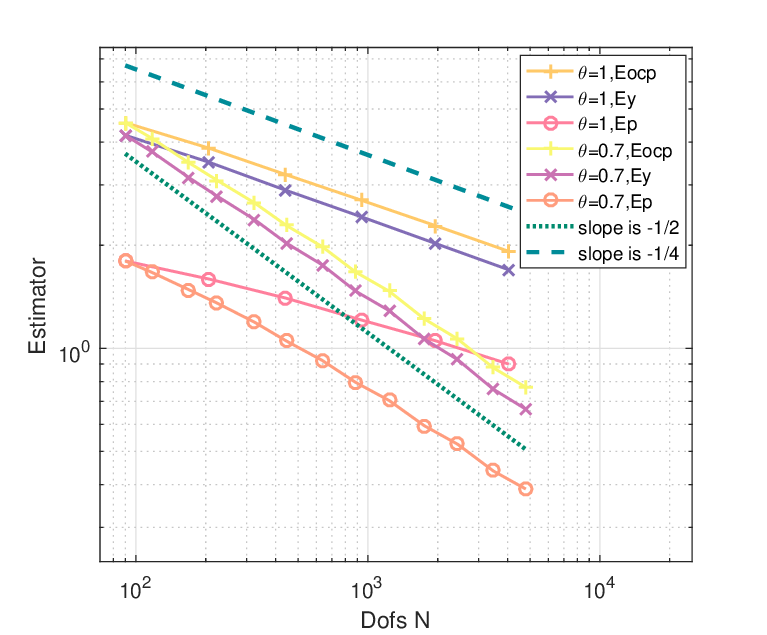}
\hspace{-0.03mm}
\label{2b}
\includegraphics[width=6.4cm]{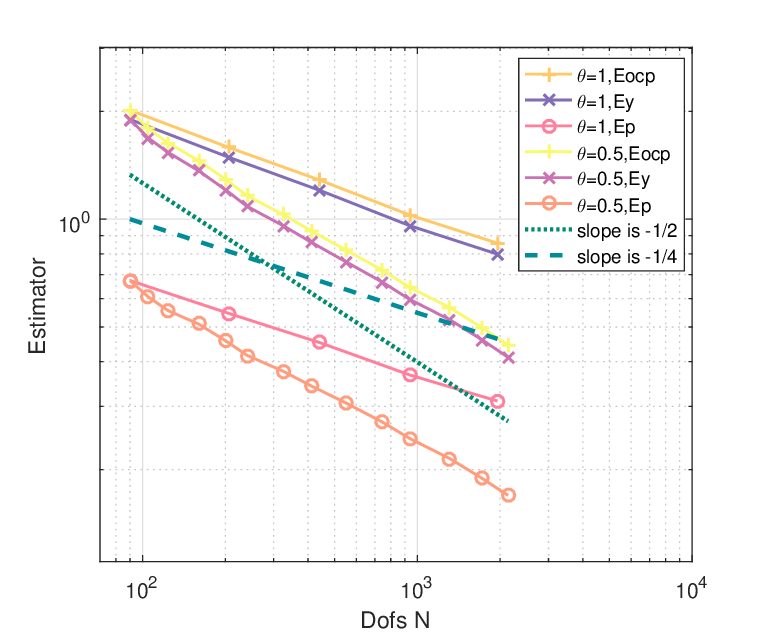}
%\hspace{-1mm}
%\label{2c}
%\includegraphics[width=4cm]{control_0.25.eps}
%\hspace{-5mm}
\caption{The convergent behaviors of the indicators and estimators for $\alpha=0.5,\ \theta=0.7,\ 1$ (left)  and  $\alpha=1.5,\ \theta=0.5,\ 1$(right).}
\label{2e}
\end{figure}
For  all choices of the parameter $\gamma$ considered, the Figure \ref{2gamma} shows the decrease of the total error estimators $\mathcal{E}_{ocp}$ and the error indicators $E_y,\ E_p$ with respect to the number of degrees of freedom (Dofs). In all the values of the parameter $\gamma$ cases the optimal rate $N^{-\frac{1}{2}}$ is achieved.

\begin{figure}[!htbp]
\centering
%\flushleft
\label{2a}
\includegraphics[width=6.4cm]{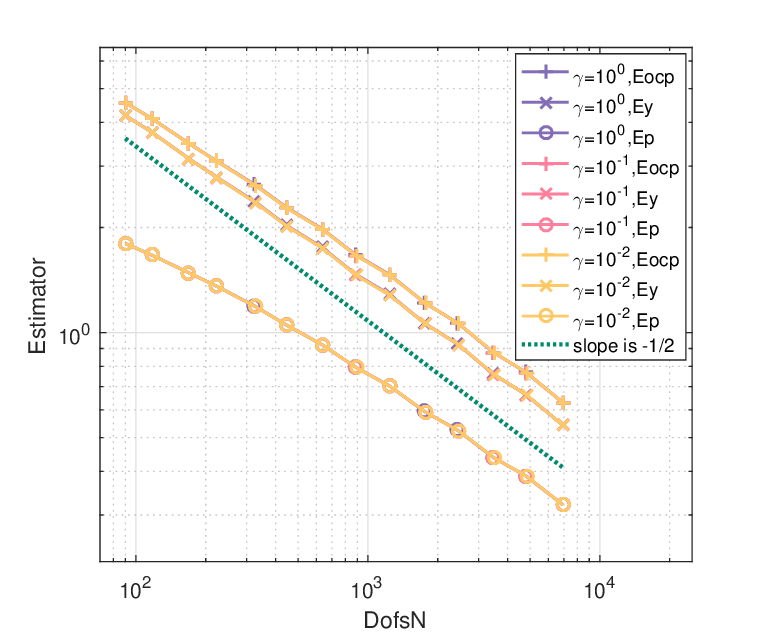}
\hspace{-0.03mm}
\label{2b}
\includegraphics[width=6.4cm]{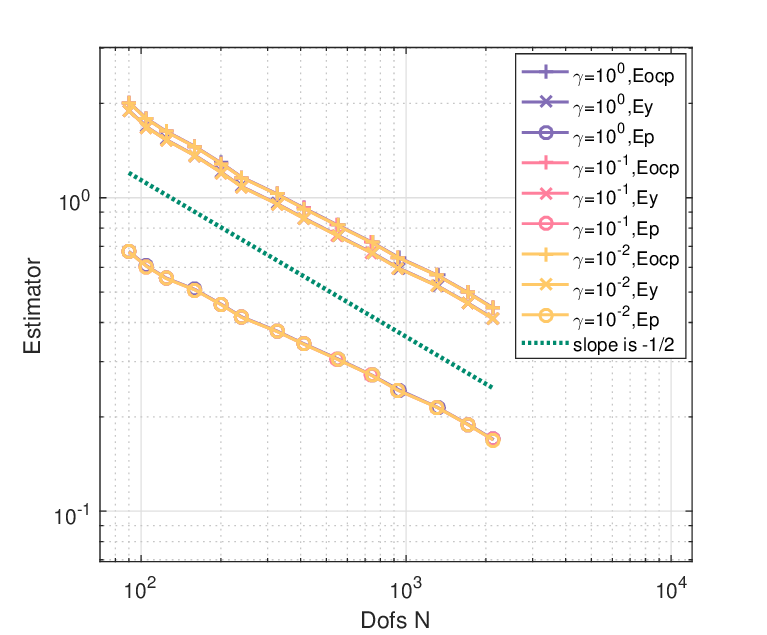}
%\hspace{-1mm}
%\label{2c}
%\includegraphics[width=4cm]{control_0.25.eps}
%\hspace{-5mm}
\caption{ The convergent behaviors of the indicators and estimators for all the values of the parameter $\gamma$ respectively on the square. We have considered $\alpha=0.5,\ \theta=0.7$ (left), $\alpha=1.5\ \theta=0.5$ (right) }
\label{2gamma}
\end{figure}
In Figures \ref{2numcontrol}-\ref{2numcontrol2}, we show the profiles of the numerical solutions for the control when $\alpha=0.5,\ \theta=0.7$ and $\alpha=1.5,\ \theta=0.5$, respectively. As $\gamma$ decreases, the numerical solutions of the control become sparsier.

 \begin{figure}[!htbp]
\centering
%\flushleft
\begin{subfigure}[]

\includegraphics[width=4.5cm]{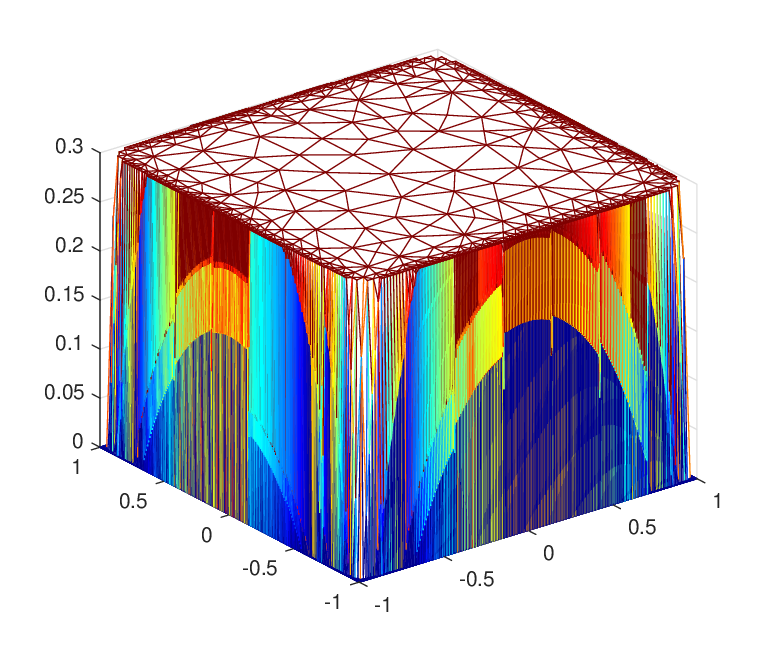}
\hspace{-0.01mm}
\end{subfigure}
\begin{subfigure}[]

\includegraphics[width=4.5cm]{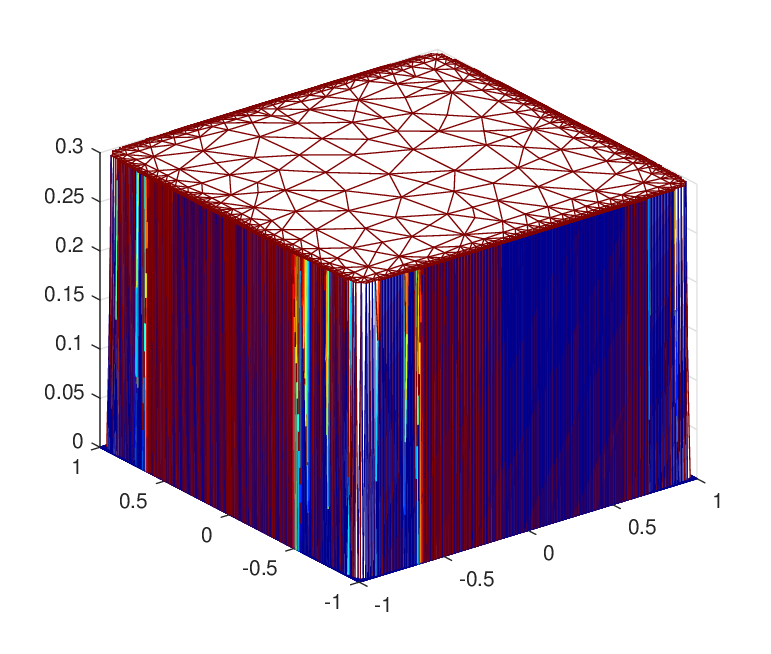}
\hspace{-0.01mm}
\end{subfigure}
\begin{subfigure}[]

\includegraphics[width=4.5cm]{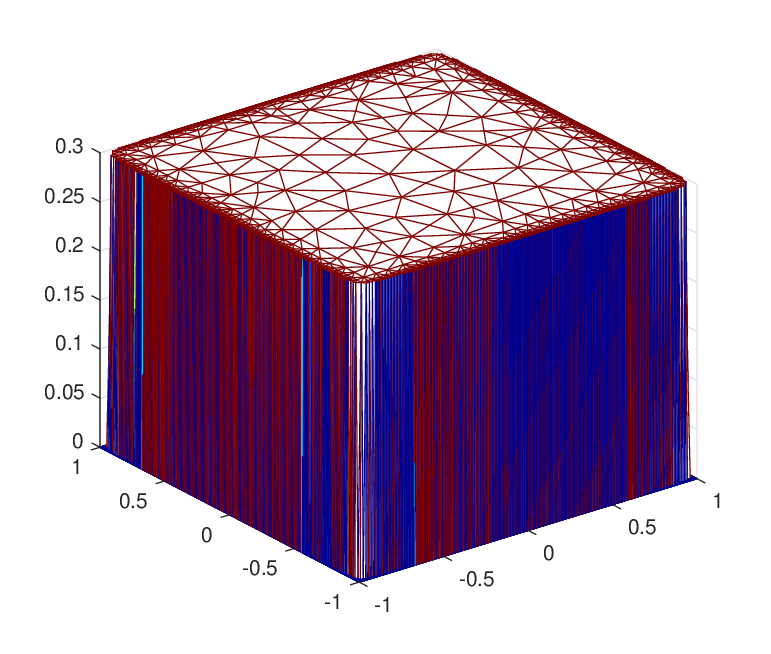}
\hspace{-0.01mm}
\end{subfigure}

\caption{The profiles of the numerically computed  control with $\alpha=0.5,\ \theta=0.7,\ \gamma=10^{0}$ (a), $\alpha=0.5,\ \theta=0.7,\ \gamma=10^{-1}$ (b) and $\alpha=0.5,\ \theta=0.7,\ \gamma=10^{-2}$  (c).}
\label{2numcontrol}
\end{figure}
 \begin{figure}[!htbp]
\centering
%\flushleft
\begin{subfigure}[]

\includegraphics[width=4.5cm]{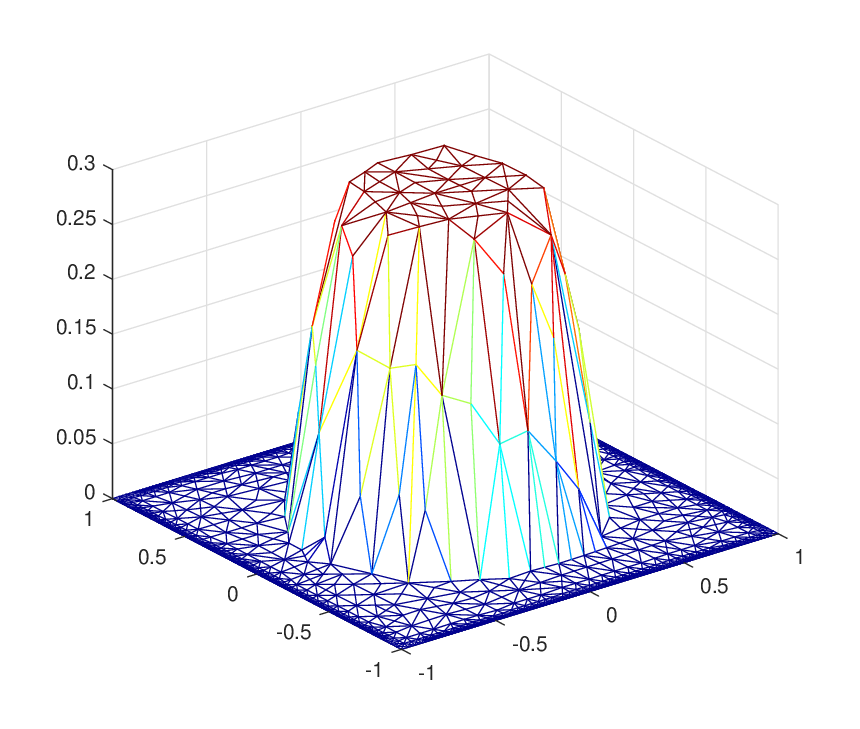}
\hspace{-0.01mm}
\end{subfigure}
\begin{subfigure}[]

\includegraphics[width=4.5cm]{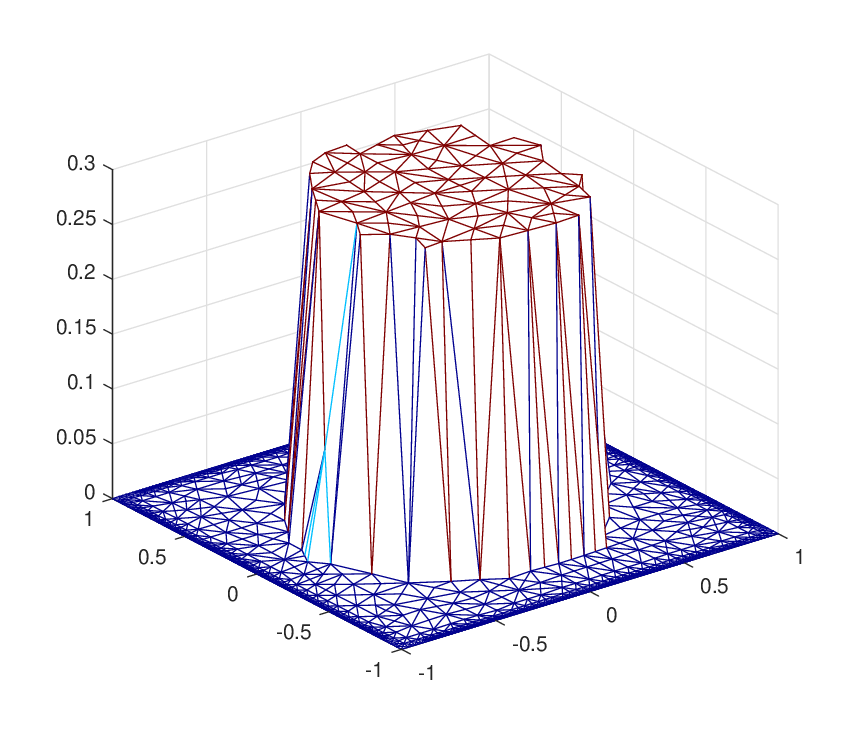}
\hspace{-0.01mm}
\end{subfigure}
\begin{subfigure}[]

\includegraphics[width=4.5cm]{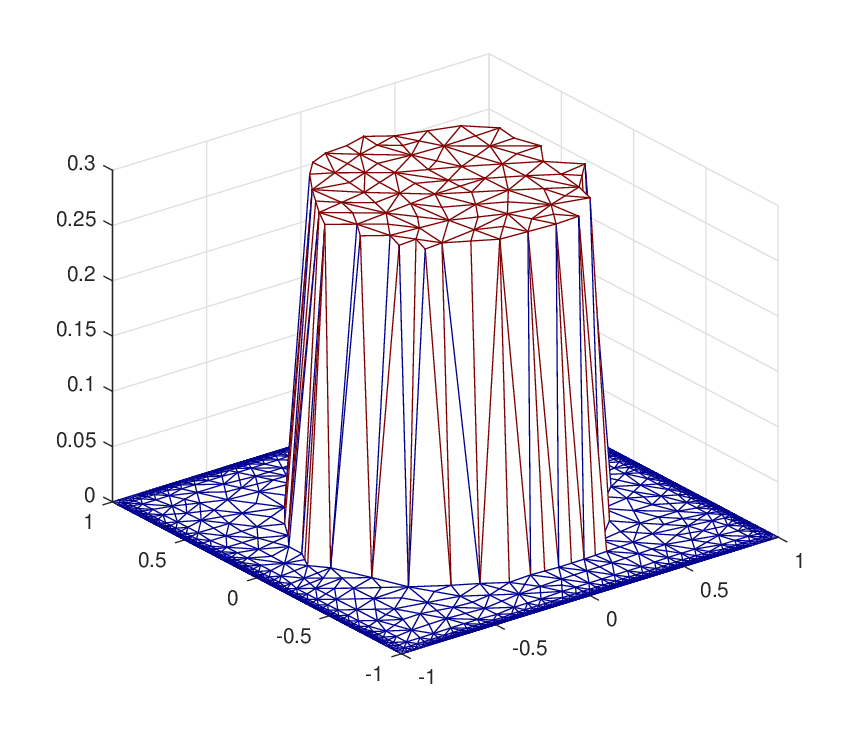}
\hspace{-0.01mm}
\end{subfigure}

\caption{The profiles of the numerically computed  control with $\alpha=1.5,\ \theta=0.5,\ \gamma=10^{0}$ (a), $\alpha=1.5,\ \theta=0.5,\ \gamma=10^{-1}$ (b) and $\alpha=1.5,\ \theta=0.5,\ \gamma=10^{-2}$  (c).}
\label{2numcontrol2}
\end{figure}
%
%In Figure \ref{2numstate}-\ref{2numstate2}, we show the profile of the numerical solution for the state when $\alpha=0.5,\ \theta=0.7$ and $\alpha=1.5,\ \theta=0.5$, respectively. As can be seen in Figures \ref{2numstate}-\ref{2numstate2}, the $L^1$ sparse term has little effect on the state %variables.

 % \begin{figure}[!htbp]
%\centering
%%\flushleft
%\begin{subfigure}[]
%
%\includegraphics[width=4.5cm]{0251y.eps}
%\hspace{-0.01mm}
%\end{subfigure}
%\begin{subfigure}[]
%
%\includegraphics[width=4.5cm]{02501y.eps}
%\hspace{-0.01mm}
%\end{subfigure}
%\begin{subfigure}[]
%
%\includegraphics[width=4.5cm]{025001y.eps}
%\hspace{-0.01mm}
%\end{subfigure}
%
%\caption{The profiles of the numerically computed  state with $\alpha=0.5,\ \theta=0.7,\ \gamma=10^{0}$ (a), state with $\alpha=0.5,\ \theta=0.7,\ \gamma=10^{-1}$ (b) and state with $\alpha=0.5,\ \theta=0.7,\ \gamma=10^{-2}$  (c).}
%\label{2numstate}
%\end{figure}
% \begin{figure}[!htbp]
%\centering
%%\flushleft
%\begin{subfigure}[]
%
%\includegraphics[width=4.5cm]{0751y.eps}
%\hspace{-0.01mm}
%\end{subfigure}
%\begin{subfigure}[]
%
%\includegraphics[width=4.5cm]{07501y.eps}
%\hspace{-0.01mm}
%\end{subfigure}
%\begin{subfigure}[]
%
%\includegraphics[width=4.5cm]{075001y.eps}
%\hspace{-0.01mm}
%\end{subfigure}
%
%\caption{The profiles of the numerically computed state with $\alpha=1.5,\ \theta=0.5,\ \gamma=10^{0}$ (a), state with $\alpha=1.5,\ \theta=0.5,\ \gamma=10^{-1}$ (b) and state with $\alpha=1.5,\ \theta=0.5,\ \gamma=10^{-2}$  (c).}
%\label{2numstate2}
%\end{figure}
\begin{example}\label{Exm:3}
In the third example we consider an optimal control problem with $f=6\sin(4y)\cos(4x)e^x,\ y_d=-4\sin(4y)\cos(4x)e^x$. We set $\Omega=(-1,1)^{2}$, $\gamma=0.1$, $\beta=1$, $a=-0.3$, $b=0.3$, respectively.
\end{example}
In Figure \ref{3mesh}  we show the initial mesh and the refinement mesh after 13 adaptive steps with $\alpha=0.5,\ \theta=0.7$. We observe that the mesh nodes are distributed around the domain where the solutions have a large gradient as well as at the boundarys. In Figure \ref{3con}, the convergence rates of error estimators and indicators for $\alpha=0.5$ and $\alpha=1.5$ are presented, respectively. It can be observed that, in both cases, each contribution decays with the optimal convergence rate  $N^{-\frac{1}{2}}$. In Figure \ref{3num},  the profiles of the numerical control and state with $\alpha=0.5,\ \theta=0.7$ are provided.
\begin{figure}[!htbp]
\centering
%\flushleft
\label{2a}
\includegraphics[width=6.5cm]{025uni.eps}
\hspace{-0.01mm}
\label{2b}
\includegraphics[width=6.4cm,height=5.8cm]{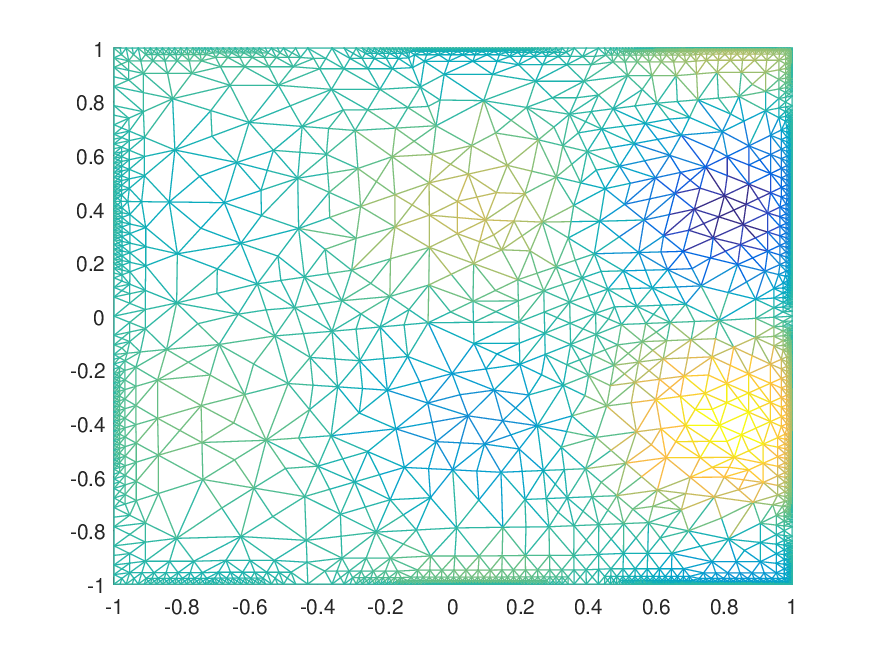}
%\hspace{-1mm}
%\label{2c}
%\includegraphics[width=4cm]{control_0.25.eps}
%\hspace{-5mm}
\caption{The initial mesh (left)  and the final refinement mesh (right) with $\alpha=0.5,\theta=0.7$ on the square.}
\label{3mesh}
\end{figure}
\begin{figure}[!htbp]
\centering
%\flushleft
\label{2a}
\includegraphics[width=6.5cm]{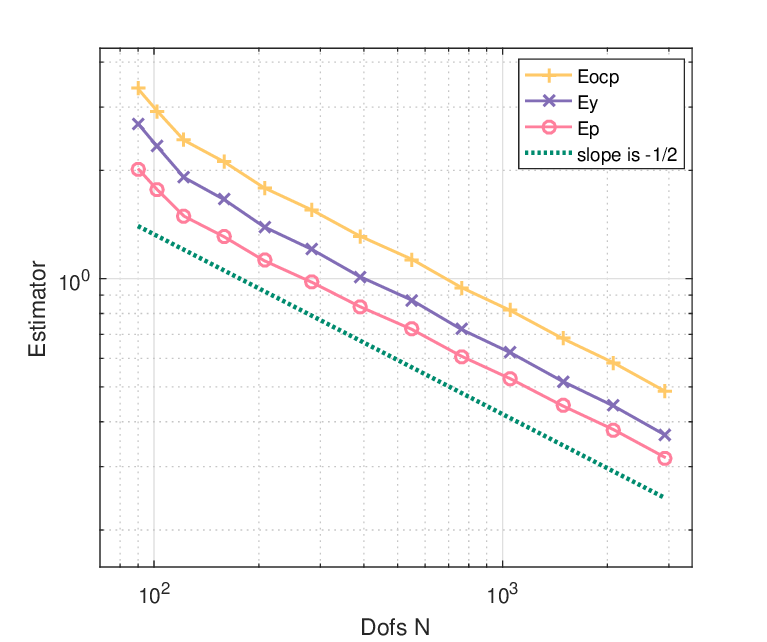}
\hspace{-0.01mm}
\label{2b}
\includegraphics[width=6.5cm]{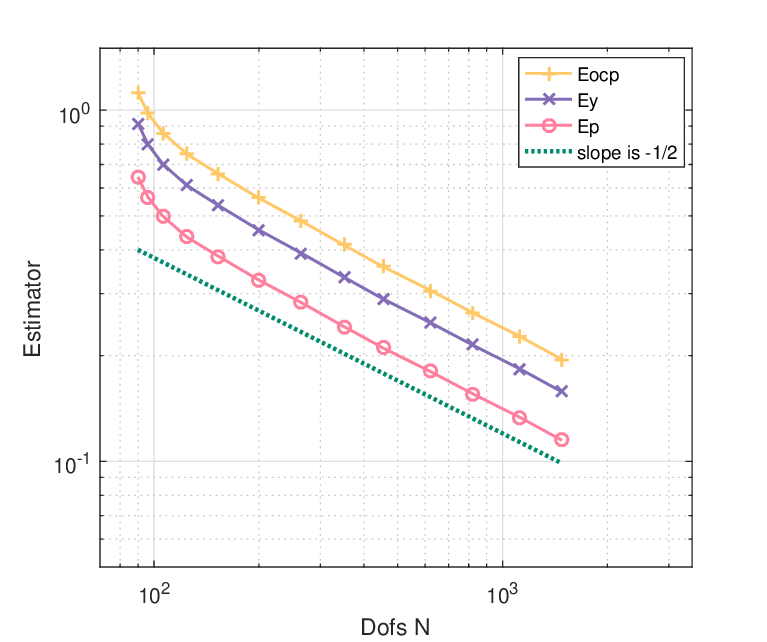}
%\hspace{-1mm}
%\label{2c}
%\includegraphics[width=4cm]{control_0.25.eps}
%\hspace{-5mm}
\caption{The convergent behaviors of the indicators and estimators for $\alpha=0.5,\ \theta=0.7$  (left)  and $\alpha=1.5,\ \theta=0.5$ (right).}
\label{3con}
\end{figure}

\begin{figure}[!htbp]
\centering
%\flushleft
\label{2a}
\includegraphics[width=6.5cm]{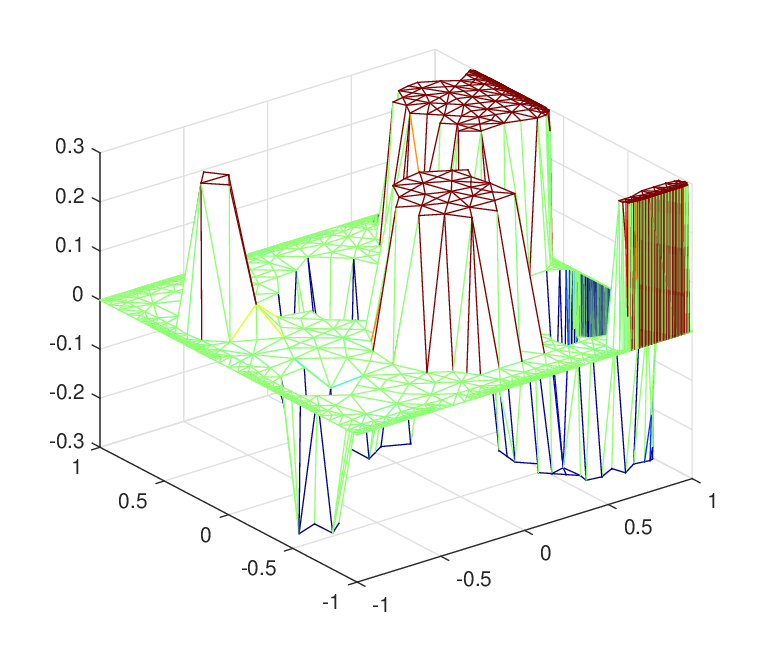}
\hspace{-0.01mm}
\label{2b}
\includegraphics[width=6.5cm]{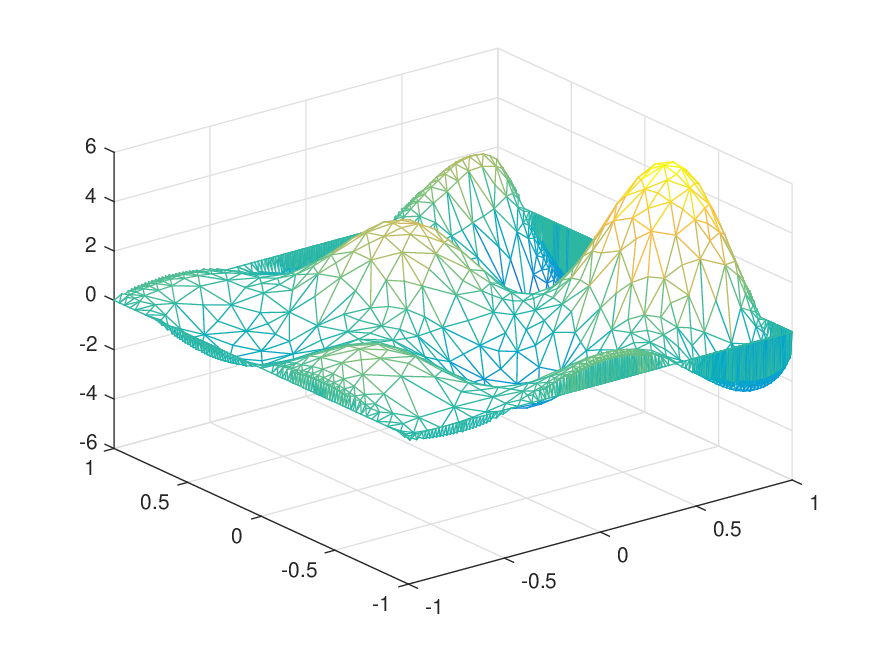}
%\hspace{-1mm}
%\label{2c}
%\includegraphics[width=4cm]{control_0.25.eps}
%\hspace{-5mm}
\caption{The numerical control(left) and state (right) with $\alpha=0.5,\ \theta=0.7$. }
\label{3num}
\end{figure}
\section{Conclusion}

In this paper, we present and analyze a weighted residual a posteriori error estimate for an optimal control problem. The problem involves a cost functional that is nondifferentiable, a state equation with an integral fractional Laplacian, and control constraints. We provide first-order optimality conditions and derive upper and lower bounds on the a posteriori error estimates for the finite element approximation of the optimal control problem. Moreover, we demonstrate that the approximation sequence generated by the adaptive algorithm converges at the optimal algebraic rate. Finally, we validate the theoretical findings through numerical experiments.
%\section*{Data availability statement}
%All data that support the findings of this study are included within the article (and any supplebmentary files).
%\section*{Declaration of Interest Statement}
% The authors declare that they have no known competing financial interests or personal relationships that could have appeared to influence the work reported in this paper.
\section*{Acknowledgements}

The work was supported by the National Natural
Science Foundation of China under Grant No. 11971276  and 12171287.

...
%%-----------------------------
%%      your bibliography
%%-----------------------------
\end{document}